\documentclass[a4paper,twoside]{amsart}
\usepackage[utf8]{inputenc}
\usepackage[hidelinks]{hyperref}
\usepackage{tikz}
\usepackage{tikz-cd}
\usetikzlibrary{shapes,arrows}
\tikzstyle{block}=[draw opacity=0.7,line width=1.4cm]
\usepackage{graphicx}
\usepackage{enumitem}
\usepackage{mathrsfs}
\usepackage{etoolbox}
\usepackage{amssymb}
\DeclareRobustCommand{\SkipTocEntry}[5]{}

\def\dom{\mathop{\mathrm{Dom}}\nolimits}
\def\age{\mathop{\mathrm{Age}}\nolimits}
\def\range{\mathop{\mathrm{Range}}\nolimits}
\def\Forb{\mathop{\mathrm{Forb_{he}}}\nolimits}
\def\Aut{\mathop{\mathrm{Aut}}\nolimits}
\def\cl{\mathop{\mathrm{Cl}}\nolimits}

\def\K{{\mathcal K}}
\def\str#1{\mathbf {#1}}
\def\arity#1{a(\rel{}{#1})}
\def\arityf#1{a(\func{}{#1})}
\def\nbrel#1#2{R\ifstrempty{#1}{}{_{#1}}\ifstrempty{#2}{}{^{#2}}}
\def\rel#1#2{\nbrel{\ifstrempty{#1}{}{\str{#1}}}{#2}}
\def\nbfunc#1#2{F\ifstrempty{#1}{}{_{#1}}\ifstrempty{#2}{}{^{#2}}}
\def\func#1#2{\nbfunc{\ifstrempty{#1}{}{\str{#1}}}{#2}}
\def\permnbrel#1#2#3{#1(R\ifstrempty{#2}{}{}\ifstrempty{#3}{}{^{#3}})_{#2}}
\def\permrel#1#2#3{#1(R\ifstrempty{#2}{}{}\ifstrempty{#3}{}{^{#3}})_{\mathbf{#2}}}

\def\permnbfunc#1#2#3{#1(F\ifstrempty{#2}{}{\ifstrempty{#3}{}{^{#3}})_{#2}}}
\def\permfunc#1#2#3{#1(F\ifstrempty{#2}{}{\ifstrempty{#3}{}{^{#3}})_{\mathbf{#2}}}}
\def\powerset#1{\mathscr{P}(#1)}
\def\support#1{\mathrm{support}(#1)}
\def\roots#1#2{\mathrm{roots}_{#1}(#2)}

\def\Sym{\mathop{\mathrm{Sym}}\nolimits}
\def\id{\mathrm{id}}
\def\Lconst{L_\mathcal F^0}

\def\nbrelE#1#2{E\ifstrempty{#1}{}{_{#1}}\ifstrempty{#2}{}{^{#2}}}
\def\relE#1#2{\nbrelE{\ifstrempty{#1}{}{\str{#1}}}{#2}}

\theoremstyle{plain}
\newtheorem{theorem}{Theorem}[section]
\newtheorem{corollary}[theorem]{Corollary}
\newtheorem{prop}[theorem]{Proposition}
\newtheorem{observation}[theorem]{Observation}
\newtheorem{lemma}[theorem]{Lemma} 
\newtheorem{conjecture}[theorem]{Conjecture}
\newtheorem{question}[theorem]{Question}

\newtheorem{fact}[theorem]{Fact}

\theoremstyle{definition}
\newtheorem{definition}[theorem]{Definition}
\newtheorem{example}[theorem]{Example}
\newtheorem{claim}[theorem]{Claim}

\theoremstyle{remark}
\newtheorem{remark}[theorem]{Remark}

\def\dom{\mathop{\mathrm{Dom}}\nolimits}

\def\str#1{\mathbf {#1}}
\def\Fraisse{Fra\"{\i}ss\' e}

\def\GammaL{\Gamma\!_L}
\def\GammaM{\Gamma\!_M}
\def\GammaN{\Gamma\!_N}

\def\GammaLstar{\Gamma\!_{L^\star}}

\DeclareMathOperator{\tr}{ctr}
\renewcommand{\restriction}{\mathord{\upharpoonright}}

\begin{document}
\bibliographystyle{alpha}

\title[All those EPPA classes]{All those EPPA classes\\
(Strengthenings of the Herwig--Lascar theorem)}
\authors{
\author[J. Hubi\v cka]{Jan Hubi\v cka}
\address{Charles University, Faculty of Mathematics and Physics\\Department of Applied Mathematics (KAM)\\
Prague, Czech Republic}
\email{hubicka@kam.mff.cuni.cz}
\author[M. Kone\v cn\'y]{Mat\v ej Kone\v cn\'y}
\address{Charles University, Faculty of Mathematics and Physics\\Department of Applied Mathematics (KAM)\\
Prague, Czech Republic}
\email{matej@kam.mff.cuni.cz}
\author[J. Ne\v set\v ril]{Jaroslav Ne\v set\v ril}
\address{Charles University, Faculty of Mathematics and Physics\\
Computer Science Institute of Charles University (IUUK)\\
Prague, Czech Republic}
\email{nesetril@iuuk.mff.cuni.cz}

\thanks{This paper is part of a project that has received funding from the European Research Council (ERC) under the European Union’s Horizon 2020 research and innovation programme (grant agreement No 810115). Jan Hubi\v cka and Mat\v ej Kone\v cn\'y are further supported by project 18-13685Y of the Czech Science Foundation (GA\v CR), Jan Hubi\v cka is also supported by Center for Foundations of Modern Computer Science (Charles University project UNCE/SCI/004) and by the PRIMUS/17/SCI/3 project of Charles University and Mat\v ej Kone\v cn\'y is also supported by the Charles University Grant Agency (GA UK), project 378119.}
}

\begin{abstract}
Let $\str A$ be a finite structure. We say that a finite structure $\str B$ is an EPPA-witness for $\str A$ if it contains $\str A$ as a substructure and every isomorphism of substructures of $\str A$ extends to an automorphism of $\str B$.  Class $\mathcal C$ of finite structures has the extension property for partial automorphisms (EPPA, also called the Hrushovski property) if it contains an EPPA-witness for every structure in $\mathcal C$.

We develop a systematic framework for combinatorial constructions of EPPA-witnesses satisfying additional local properties and thus for proving EPPA for a given class $\mathcal C$.
Our constructions are elementary, self-contained and lead to
a common strengthening of the Herwig--Lascar theorem on EPPA for relational classes defined by forbidden homomorphisms, the Hodkinson--Otto theorem on EPPA for relational free amalgamation classes, its strengthening for unary functions by Evans, Hubi\v cka and Ne\v set\v ril and their coherent variants by Siniora and Solecki.
We also prove an EPPA analogue of the main results of J.~Hubi\v cka and J.~Ne\v set\v ril: \emph{All those Ramsey classes (Ramsey classes with closures and forbidden homomorphisms)}, thereby establishing a common framework for proving EPPA and the Ramsey property.

There are numerous applications of our results, we include a solution of a problem related to
a class constructed by the Hrushovski predimension construction. We also characterize free amalgamation classes of finite $\GammaL$-structures with relations and unary functions which have EPPA.
\end{abstract}
\maketitle

\newpage

\tableofcontents

\newpage

\section{Introduction}
Let $\str A$ and $\str B$ be finite structures (e.g. graphs, hypergraphs or metric spaces) such that $\str A$ is a substructure of $\str B$. We say that $\str B$ is an \emph{EPPA-witness} for $\str A$ if every isomorphism of substructures of $\str A$ (a \emph{partial automorphism of $\str A$}) extends to an automorphism of $\str B$. We say that a class $\mathcal C$ of finite structures has the \emph{extension property for partial automorphisms} (\emph{EPPA}, also called the \emph{Hrushovski property}) if for every $\str A\in \mathcal C$ there is $\str B\in \mathcal C$ which is an EPPA-witness for $\str A$.

In 1992, Hrushovski~\cite{hrushovski1992} established that the class of all finite graphs has EPPA. This result was used by Hodges, Hodkinson, Lascar, and Shelah to show the small index property for the random graph~\cite{hodges1993b}. After this, the quest of identifying new classes of structures with EPPA continued with a series of papers including~\cite{Herwig1995,herwig1998,herwig2000,hodkinson2003,solecki2005,vershik2008,Conant2015,otto2017,Aranda2017,Hubicka2018metricEPPA,Konecny2018b,Hubicka2017sauer,eppatwographs}.

In particular, Herwig and Lascar~\cite{herwig2000} proved EPPA for certain relational classes with forbidden homomorphisms. Solecki~\cite{solecki2005} used this result to prove EPPA for the class of all
finite metric spaces. This was independently obtained by
Vershik~\cite{vershik2008}, see also~\cite{Pestov2008,rosendal2011,rosendal2011b,sabok2017automatic,Hubicka2018metricEPPA} for other proofs. Some of these proofs are combinatorial~\cite{Hubicka2018metricEPPA}, others are using the profinite topology on free groups and the {R}ibes--{Z}alesskii~\cite{Ribes1993} and Marshall Hall~\cite{hall1949} theorems. Solecki's argument was refined by Conant~\cite{Conant2015} for
certain classes of generalised metric spaces and metric spaces with
(some) forbidden subspaces. In~\cite{Aranda2017}, these techniques were carried
further and a layer was added on top of the Herwig--Lascar theorem to show EPPA for many classes of metrically homogeneous graphs from
Cherlin's catalogue~\cite{Cherlin2013} (see also exposition in~\cite{Konecny2018bc}).

There are known EPPA classes for which the Herwig--Lascar theorem is not well suited. In particular, EPPA for free amalgamation classes of relational structures was shown by Siniora and Solecki~\cite{Siniora} using results of Hodkinson and Otto~\cite{hodkinson2003}.
It was noticed by Ivanov~\cite{Ivanov2015} that a lemma on permomorphisms from Herwig's paper~\cite[Lemma~1]{herwig1998} can be used to show EPPA for structures with
definable equivalences on $n$-tuples with infinitely many equivalence classes. Evans, Hubi\v cka, and Ne\v set\v
ril~\cite{Evans3} strengthened the aforementioned construction of Hodkinson and Otto and established EPPA for free
amalgamation classes in languages with relations and unary functions (e.g. the class of $k$-orientations arising from a Hrushovski construction~\cite{Evans2} or the class of all finite bowtie-free graphs~\cite{Evans2}). 

We give a combinatorial, elementary, and fully self-contained proof of a strengthening of all the aforementioned results~\cite{herwig1998,herwig2000,hodkinson2003,Evans3} and their coherent variants by Siniora and Solecki~\cite{Siniora, Siniora2}.
This has a number of applications which we list in Section~\ref{subsec:applications}. In particular, in Section~\ref{sec:orientations} we present a solution of a 
problem related to a class constructed by the Hrushovski predimension construction.

\subsection{$\GammaL$-structures}
Before presenting the summary of our results, let us introduce the structures we are dealing with.
With applications in mind, we generalise the standard notion of model-theoretic $L$-structures in two directions. We consider functions which go to subsets of the vertex set and we also equip the languages with a permutation group $\GammaL$. Our morphisms will consist of a map between vertices together with a permutation of the language: The standard notions of homomorphism, embedding etc. are generalised naturally, see Section~\ref{sec:background} for formal definitions.
If $\GammaL$ consists of the identity only and the ranges of all functions consist of singletons, one gets back the standard model-theoretic $L$-structures together with the standard mappings, the standard definition of EPPA etc.


\subsection{The main results}
We now state the principal results of this paper together with a short discussion.

A structure is \emph{irreducible} if it is not a free amalgamation of its proper substructures. If $\str B$ is an EPPA-witness for $\str A$, we say that it is \emph{irreducible structure faithful} if whenever $\str C$ is an irreducible substructure of $\str B$, then there is an automorphism $g\in \Aut(\str B)$ such that $g(C)\subseteq A$. A class has \emph{irreducible structure faithful EPPA} if it has EPPA and all EPPA-witnesses can be chosen to be irreducible structure faithful. (This is a natural generalization of the clique faithful EPPA introduced by Hodkinson and Otto~\cite{hodkinson2003} to structures with functions.)
Coherent EPPA is a ``functorial'' strengthening of EPPA introduced by Siniora and Solecki~\cite{Siniora, Siniora2} and is defined in Section~\ref{sec:coherence}. 

In this paper we prove two main theorems. The ``base'' unrestricted theorem, formulated as Theorem~\ref{thm:nreppa}, gives irreducible structure faithful coherent EPPA for the class of all finite $\GammaL$-structures, strengthening results of Herwig~\cite[Lemma 1]{herwig1998}, Hodkinson and Otto~\cite{hodkinson2003}, its coherent variant of Siniora and Solecki~\cite{Siniora}, and Evans, Hubi\v cka, and Ne\v set\v ril~\cite{Evans3}.
\begin{theorem}[Construction of an unrestricted EPPA-witness]
\label{thm:nreppa}
Let $L$ be a language consisting of relations and unary functions equipped with a permutation group $\GammaL$ and let $\str A$ be a finite $\GammaL$-structure. If $\str A$ lies in a finite orbit of the action of $\GammaL$ by relabelling, then there is a finite $\GammaL$-structure $\str B$, which is an irreducible structure faithful coherent EPPA-witness for $\str A$.

Consequently, the class of all finite $\GammaL$-structures has irreducible structure faithful coherent EPPA if every finite $\GammaL$-structure lies in a finite orbit of the action of $\GammaL$ by relabelling.
\end{theorem}
Here, the action of $\GammaL$ by relabelling is defined such that $g\in \GammaL$ sends a $\GammaL$-structure $\str A$ to a $\GammaL$-structure $\str A'$ on the same vertex set where the relations and functions are relabelled according to $g$ (see Definition~\ref{defn:relabelling}). Note that in particular if $\GammaL$ is finite (e.g. $\GammaL=\{\id_L\}$ or $L$ is finite), then every $\GammaL$-structure lies in a finite orbit of the action of $\GammaL$ by relabelling.

\medskip

After that, we provide a theorem which, given a finite irreducible structure faithful (coherent) EPPA-witness $\str B_0$ for $\str A$, produces a finite irreducible structure faithful (coherent) EPPA-witness $\str B$ for $\str A$ while providing extra control over the local structure of $\str B$:

\begin{theorem}[Construction of a restricted EPPA-witness]
\label{thm:maintree}
Let $L$ be a language consisting of relations and unary functions equipped with a permutation group $\GammaL$, let $\str A$ be a finite irreducible $\GammaL$-structure, let $\str{B}_0$ be a finite EPPA-witness for $\str{A}$ and let $n\geq 1$ be an integer. There is
a finite $\GammaL$-structure $\str{B}$ satisfying the following.
\begin{enumerate}
\item $\str B$ is an irreducible structure faithful EPPA-witness for $\str A$.
\item There is a homomorphism-embedding $\str B\to \str B_0$.
\item For every substructure $\str{C}$ of $\str{B}$ on at most
$n$ vertices there is a tree amalgamation $\str D$ of copies of $\str A$ and a homomorphism-embedding $f\colon \str C\to\str D$.
\item If $\str B_0$ is coherent then so is $\str B$.
\end{enumerate}
\end{theorem}
Here, a \emph{tree amalgamation of copies of $\str A$} is any structure which can be created by a series of free amalgamations of copies of $\str A$ over its substructures (see Definition~\ref{defn:tree-amalgamation}).

\medskip

Theorem~\ref{thm:nreppa} contains the condition that $\str A$ needs to lie in a finite orbit of the action of $\GammaL$ by relabelling. We prove that this is in fact necessary:
\begin{theorem}\label{thm:negative}
Let $L$ be a language equipped with a permutation group $\GammaL$ and let $\str A$ be a finite $\GammaL$-structure. If $\str A$ lies in an infinite orbit of the action of $\GammaL$ by relabelling then there is no finite $\GammaL$-structure $\str B$ which is an EPPA-witness for $\str A$.
\end{theorem}

We can combine Theorems~\ref{thm:nreppa} and~\ref{thm:negative} with an easy observation used earlier~\cite{hodkinson2003, Evans3, Siniora2} and characterize free amalgamation classes of finite $\GammaL$-structures which have (irreducible structure faithful coherent) EPPA, provided that all functions in the language are unary. The following theorem strengthens results of Hodkinson and Otto~\cite{hodkinson2003}, Evans, Hubi\v cka, and Ne\v set\v ril~\cite{Evans3}, and Siniora~\cite{Siniora2}, and in particular implies irreducible structure faithful coherent EPPA for the class of all graphs, $K_n$-free graphs or $k$-regular hypergraphs.

\begin{corollary}\label{cor:free}
Let $L$ be a language consisting of relations and unary functions equipped with a permutation group $\GammaL$ and let $\mathcal K$ be a free amalgamation class of finite $\GammaL$-structures. Then $\mathcal K$ has EPPA if and only if 
every $\str A\in \mathcal K$ lies in a finite orbit of the action of $\GammaL$ on $\GammaL$-structures by relabelling.
Moreover, if $\mathcal K$ has EPPA, then it has irreducible structure faithful coherent EPPA.
\end{corollary}

\medskip

We also provide two corollaries of Theorem~\ref{thm:maintree}, which might be easier to apply in some cases. The first corollary is a direct strengthening of the Herwig--Lascar theorem~\cite[Theorem~3.2]{herwig2000} and its coherent variant of Solecki and Siniora~\cite[Theorem~1.10]{Siniora}.
For a set $\mathcal F$ of $\GammaL$-structures, we denote by $\Forb(\mathcal F)$ the set of all finite and countable $\GammaL$ structures $\str A$ such that there is no $\str F\in\mathcal F$ with a homomorphism-embedding $\str F\to \str A$.

\begin{theorem}
\label{thm:main}
Let $L$ be a language consisting of relations and unary functions equipped with a permutation group $\GammaL$.
Let $\mathcal F$ be a finite family of finite $\GammaL$-structures and let $\str{A}\in \Forb(\mathcal F)$ be a finite $\GammaL$-structure which lies in a finite orbit of the action of $\GammaL$ by relabelling.  If there exists a (not necessarily finite) structure $\str{M}\in \Forb(\mathcal F)$ containing $\str{A}$ as a substructure such that each partial automorphism of $\str{A}$ extends to an automorphism of $\str{M}$, then there exists a finite structure $\str{B}\in \Forb(\mathcal F)$ which is an irreducible structure faithful coherent EPPA-witness for $\str{A}$.
\end{theorem}

Hubi\v cka and Ne\v set\v ril~\cite{Hubicka2016} gave a structural condition for a class to be Ramsey. It turns out that, in papers studying Ramsey expansions of various classes using their theorem, EPPA is sometimes an easy corollary of one of the intermediate steps, see e.g.~\cite{Aranda2017, Aranda2017a, Aranda2017c, Konecny2018b}. In this paper, we make this link explicit by proving a theorem on EPPA whose statement is very similar to~\cite[Theorem~2.18]{Hubicka2016}. For the definition of a locally finite automorphism-preserving subclass, see Section~\ref{sec:ramsey}.

\begin{theorem}
\label{thm:mainstrong}
Let $L$ be a language consisting of relations and unary functions equipped with a permutation group $\GammaL$ and let $\mathcal E$ be a class of finite $\GammaL$-structures which has EPPA.
Let $\K$ be a hereditary 
locally finite automorphism-preserving subclass of $\mathcal E$ with the strong amalgamation property which consists of irreducible structures. Then $\K$ has EPPA.  

Moreover, if EPPA-witnesses in $\mathcal E$ can be chosen to be coherent then EPPA-witnesses in $\K$ can be chosen to be coherent, too.
\end{theorem}

\subsection{Applications of our results}\label{subsec:applications}
When proving EPPA for (some of) the antipodal classes of metrically homogeneous graphs in~\cite{Aranda2017}, an additional ad hoc layer was added on top of an application of the Herwig--Lascar theorem to ensure that edges of length $\delta$ form a matching~\cite[Theorem 7.7]{Aranda2017}. Another ad hoc layer was needed in the same paper for the bipartite classes~\cite[Theorem 6.13]{Aranda2017}. These ad hoc constructions can be avoided using the main theorem of this paper, adding unary functions to represent edges of length $\delta$ for the antipodal classes, or adding two unary predicates which can be swapped by $\GammaL$ to describe the partition for the bipartite classes.

When proving EPPA for the antipodal classes of odd diameter and the bipartite antipodal classes of even diameter of metrically homogeneous graphs in~\cite{Konecny2019a}, the full strength of Theorem~\ref{thm:mainstrong} (from an early draft of this paper) was used: One needed unary functions to represent that edges of length $\delta$ form a matching, control over substructures to ensure that no non-metric cycles are created, and language permutations to generalise a construction from~\cite{eppatwographs}.

Section~\ref{sec:applications} of this paper is devoted to applications. In particular, we outline how to further strengthen our results to languages with constants or certain non-unary functions (see Theorems~\ref{thm:constants} and~\ref{thm:nonunary}), and we prove that a class connected to Hrushovski's predimension construction has EPPA (see Theorem~\ref{thm:DF}). The two latter proofs use a general method for dealing with higher-arity functions by chaining several applications of Theorem~\ref{thm:maintree} on top of each other using language permutations.

We are confident that there are many more applications of the main theorems of this paper to be discovered.

\subsection{EPPA and Ramsey}
The results and techniques of this paper are motivated by recent developments
of the structural Ramsey theory, particularly the efforts to characterise Ramsey classes
of finite structures. As this paper demonstrates, many techniques and proof strategies from structural Ramsey theory may serve
as a motivation for results about EPPA classes.  
We were inspired by the scheme of proofs of corresponding Ramsey results in~\cite{Hubicka2016},
by the construction of clique faithful EPPA-witnesses for relational structures given by Hodkinson and Otto~\cite{hodkinson2003}, by the treatment of unary function in~\cite{Evans3}, and by the recent proofs of EPPA for metric spaces~\cite{Hubicka2018metricEPPA} and for two-graphs~\cite{eppatwographs}.

In each section, we fix a $\GammaL$-structure $\str A$ and give an explicit construction of a $\GammaL$-structure $\str B$ and an embedding $\str A\to \str B$. Then, given a partial automorphism of $\str A$, we show how to construct an automorphism of $\str B$ extending it, that is, we prove that $\str B$ is an EPPA-witness for $\str A$. Finally, we prove that $\str B$ has the given special properties (e.g. irreducible structure faithfulness, or control over small substructures) and that the extension is coherent. Usually, the constructions are the interesting part and the proofs are just verification that a function is an automorphism and that it composes correctly.

While all this may be surprising on the first glance
and it is one of the novelties of this paper, we want to stress at this point some of the main differences between EPPA and Ramsey.
(Further open problems will be in Section~\ref{sec:conclusion}.)

Both EPPA and the Ramsey property imply the amalgamation property (see Observation~\ref{obs:eppaamalgamation}
and~\cite{Nevsetril2005}) and have strong consequences for the \Fraisse{} limits. Nonetheless,
not every amalgamation class has EPPA or the Ramsey property. While there is a meaningful
conjecture motivating the classification program of Ramsey classes (see \cite{Hubicka2016,Bodirsky2011a}),
for the classification of EPPA classes this is not yet the case.

The classification programme for EPPA classes was initiated in~\cite{Evans2,Evans3} by giving examples of classes with a non-trivial EPPA expansion. (See also the survey by the first author~\cite{hubika2020structural}.)
There exist many classes which have a non-trivial Ramsey expansion but fail to have a non-trivial EPPA expansion. Examples include the class of all finite linear orders or the class of all finite finite partially ordered sets. On the other hand, to the author's best knowledge, whenever a Ramsey expansion of an EPPA class is known, the expansion only adds a ``small amount of information'' (compared to what is promised for $\omega$-categorical structures by~\cite[Theorem 4.5]{Kechris2012}).

The correspondence between the structural conditions for EPPA and Ramsey classes then motivates the following conjecture.
\begin{conjecture}
Every strong amalgamation class with EPPA has a precompact Ramsey expansion.
\end{conjecture}
(See Section~\ref{sec:amalg} for a definition of strong amalgamation and for example~\cite{NVT14} for a definition of a precompact expansion.) Note that this conjecture implies every $\omega$-categorical structure with EPPA having a precompact Ramsey expansion.
Classes known to have EPPA where it is not known if they have a precompact Ramsey expansion
include the class of all finite groups~\cite{Siniora2,paolini2018automorphism} and the class of all finite skew-symmetric bilinear forms\footnote{David M. Evans, personal communication. See also~\cite{Cherlin2003}.}. More open problems are listed in Section~\ref{sec:conclusion}.

It is worth to mention a result of Jahel and Tsankov~\cite{jahel2020invariant} who prove that
for a large number of classes, EPPA implies the ordering property (which is closely related to the Ramsey property, see~\cite{Kechris2005}). In particular, this implies that while for Ramsey classes, there exists an ordering of \Fraisse{} limit which is compatible with the group of automorphisms, for EPPA classes satisfying the conditions of~\cite{jahel2020invariant} such a global
ordering cannot be definable. This in fact may be one of the main dividing lines.

Based on all this information and an analogous scheme in the Ramsey context~\cite{Hubicka2016}, this may be schematically depicted as follows.
\begin{center}
\begin{tikzpicture}[auto,
    box/.style ={rectangle, draw=black, thick, fill=white,
      text width=9em, text centered,
      minimum height=2em}]
     \tikzstyle{line} = [draw, thick, -latex',shorten >=2pt];
    \matrix [column sep=5mm,row sep=3mm] {
      \node [box] (Ramsey) {Ramsey\\ classes};
      &&\node [box] (amalg) {amalgamation classes};
      &&\node [box] (EPPA) {EPPA\\ classes};
      \\
\\
      \node [box] (lift) {more\\ special structures};
      &&\node [box] (lim) {ultrahomogeneous\\ structures};
      &&\node [box] (lift2) {special\\ structures};
      \\
    };
    \begin{scope}[every path/.style=line]
      \path (Ramsey)   -- (amalg);
      \path (EPPA)   -- (amalg);
      \path (amalg)   -- (lim);
      \path (lim)   -- (lift2);
      \path (lim)   -- (lift);
      \path (lift)   -- (Ramsey);
      \path (lift2)   -- (EPPA);
      \draw[->] (lift2.south west) to [bend left=20]  (lift.south east);
    \end{scope}
  \end{tikzpicture}
\end{center}


\medskip

This paper is organised as follows:
In Section~\ref{sec:background}, we give all the necessary notions and definitions. In Section~\ref{sec:graphs}, which is supposed to serve as a warm-up, we give a new proof of (a coherent strengthening of) Hrushovski's theorem~\cite{hrushovski1992}. Then, in Sections~\ref{sec:relstructures} and~\ref{sec:infinite_languages}, we show that this new construction generalises naturally to relational $\GammaL$-structures. In Section~\ref{sec:functions}, we add a new layer which allows the language to also contain unary functions. In Section~\ref{sec:faithful}, we combine this with techniques introduced earlier~\cite{hodkinson2003,Evans3} to obtain irreducible structure faithfulness, and in Section~\ref{sec:cycles}, we once again use a similar construction to deal with forbidden homomorphic images, which allows us to prove the main theorems of this paper in Sections~\ref{sec:maintree},~\ref{sec:main} and~\ref{sec:ramsey}. Finally, in Section~\ref{sec:applications}, we apply our results and prove EPPA for the class of $k$-orientations with $d$-closures, thereby confirming the first part of~\cite[Conjecture~7.5]{Evans2}. We also prove Corollary~\ref{cor:free}, illustrate the usage of Theorems~\ref{thm:nreppa} and~\ref{thm:mainstrong} on the example of integer-valued metric spaces with no large subspaces, where all vertices are in distance 1 from each other and prove EPPA for languages with constants or certain classes with non-unary functions.

\section{Background and notation}
\label{sec:background}

We find it convenient to work with model-theoretic structures generalised in two ways: We equip the language with a permutation group (giving a more systematic treatment to the concept of \emph{permomorphisms} introduced by Herwig~\cite{herwig1998}) and consider functions to the powerset (a further generalisation of~\cite{Evans3}). This is motivated by applications, see Section~\ref{sec:orientations}.

Let $L=L_\mathcal R\cup L_\mathcal F$ be a language consisting of relation symbols $\rel{}{}\in L_\mathcal R$ and function symbols $F\in L_\mathcal F$ each having its {\em arity} denoted by $\arity{}\geq 1$ for relations and $\arityf{}\geq 0$ for functions.

Let $\GammaL$ be a permutation group on $L$ which preserves types and arities of all symbols. We say that $L$ is a \emph{language equipped with a permutation group $\GammaL$}. 
Observe that when $\GammaL$ is trivial and the ranges of all functions consist of singletons, one obtains the usual notion of model-theoretic language (and structures). All results and
constructions in this paper presented on $\GammaL$-structures thus hold also for standard $L$-structures. By this we mean that given a class of standard $L$-structures, one can treat them as $\GammaL$-structures with $\GammaL$ trivial, use the results of this paper and then, perhaps after some straightforward adjustments, obtain the same results for the original class (see Observations~\ref{obs:relabelling} and~\ref{obs:translate_to_standard}).

Denote by $\powerset{A}$ the set of all subsets of $A$. A \emph{$\GammaL$-structure} $\str{A}$ is a structure with {\em vertex set} $A$, functions $\func{A}{}\colon A^{\arityf{}}\to \powerset{A}$ for every $\func{}{}\in L_\mathcal F$ and relations $\rel{A}{}\subseteq A^{\arity{}}$ for every $\rel{}{}\in L_\mathcal R$.
Notice that the domain of a function is a tuple while the range is a set, the reason for this is that it allows to explicitly represent algebraic closures by functions. If the set $A$ is finite, we call $\str A$ a \emph{finite structure}. We consider only structures with finitely or countably infinitely many vertices. 
If $L_\mathcal F = \emptyset$, we call $L$ a {\em relational language} and say that a $\GammaL$-structure is  a {\em relational $\GammaL$-structure}.
A function $\func{}{}$ such that $\arityf{}=1$ is a {\em unary function}.

In this paper, the language and its permutation group are often fixed and understood from the context (and they are in most cases denoted by $L$ and $\GammaL$ respectively), we also only consider unary functions.

\subsection{Maps between $\GammaL$-structures}
A \emph{homomorphism} $f\colon \str{A}\to \str{B}$ is a pair $f=(f_L,f_A)$ where $f_L\in \GammaL$ and $f_A$ is a mapping $A\to B$
 such that  for every $\rel{}{}\in L_\mathcal R$ and $\func{}{}\in L_\mathcal F$ we have:
\begin{enumerate}
\item[(a)] $(x_1,\ldots, x_{\arity{}})\in \rel{A}{}\implies (f_A(x_1),\ldots,f_A(x_{\arity{}}))\in \permrel{f_L}{B}{}$, and
\item[(b)] $f_A(\func{A}{}(x_1,\allowbreak \ldots, x_{\arityf{}}))\subseteq\permfunc{f_L}{B}{}(f_A(x_1),\ldots,f_A(x_{\arityf{}}))$.
\end{enumerate}
If $f = (f_L, f_A)\colon \str A\to \str B$ and $g=(g_L, g_B)\colon \str B\to \str C$ are homomorphisms, we denote by $gf = g\circ f = (g_L\circ f_L, g_B\circ f_A)$ the homomorphism $\str A\to \str C$ obtained by their composition. (It is straightforward to check that the composition is indeed a homomorphism $\str A\to\str C$.)

If $f_A$ is injective then $f$ is called a \emph{monomorphism}. A monomorphism $f=(f_L,f_A)$ is an \emph{embedding} if for every $\rel{}{}\in L_\mathcal R$ and $\func{}{}\in L_\mathcal F$:
\begin{enumerate}
\item[(a)] $(x_1,\ldots, x_{\arity{}})\in \rel{A}{}\iff (f_A(x_1),\ldots,f_A(x_{\arity{}}))\in \permrel{f_L}{B}{}$, and
\item[(b)] $f_A(\func{A}{}(x_1,\allowbreak \ldots, x_{\arityf{}}))=\permfunc{f_L}{B}{}(f_A(x_1),\ldots,f_A(x_{\arityf{}}))$.
\end{enumerate}
If $f$ is an embedding where $f_A$ is one-to-one then $f$ is an \emph{isomorphism}. An isomorphism from a structure to itself is called an \emph{automorphism}. If $f_A$ is an inclusion and $f_L$ is the identity then $\str{A}$ is a \emph{substructure} of $\str{B}$ and we may write $\str A\subseteq \str B$ to denote this fact.

Given a $\GammaL$-structure $\str{B}$ and $A\subseteq B$, the {\em closure of $A$ in $\str{B}$}, denoted by $\cl_\str{B}(A)$, is the smallest substructure of $\str{B}$ containing $A$.
For $x\in B$, we will also write $\cl_\str{B}(x)$ for $\cl_\str{B}(\{x\})$ and for a tuple $\bar{x} = (x_1,\ldots,x_n)\in B^n$ we will write $\cl_\str{B}(\bar{x})$ for $\cl_\str{B}(\{x_1,\ldots,x_n\})$.

Let $\str A$, $\str B$, $\str C$ and $\str C'$ be $\GammaL$-structures such that $\str C\subseteq \str A$ and $\str C'\subseteq \str B$. If $f\colon \str C\to \str C'$ is an isomorphism, we may also call it a \emph{partial isomorphism} between $\str A$ and $\str B$ (note that $f$ also includes a permutation $f_L\in\GammaL$).

Let $f = (f_L, f_A)\colon \str A\to \str B$ be a homomorphism. For brevity, we may write $f(x)$ for $f_A(x)$ in the context where $x\in A$, and $f(S)$ for $f_L(S)$ where $S\in L$. By $\dom(f)$ and $\range(f)$ we will always mean $\dom(f_A)$ and $\range(f_A)$ respectively. If $\bar{x} = (x_1,\ldots,x_n) \in A^n$ then by $f(\bar{x}) = f_A(\bar{x})$ we mean the tuple $(f_A(x_1),\ldots,f_A(x_n))$, and if $X\subseteq A^n$ then we put $f(X) = f_A(X) = \{f(\bar{x}) : \bar{x}\in X\}$.

Note that $f(\rel{A}{}) = f_A(\rel{A}{})$ is the image of a set of tuples, while $f(\rel{}{})_\str B = f_L(\rel{}{})_\str B$ is the realisation of the relation $f_L(\rel{}{})$ in $\str B$. These sets need not be equal in general (they will, however, be equal whenever $f$ is an embedding).

\medskip

If $f_L\in \GammaL$ and $f_A$ is a function from $A$ to some set $X$, we denote by $f(\str{A})$ the \emph{homomorphic image} of structure $\str{A}$, that is, the $\GammaL$-structure with vertex set $f_A(A)$ such that for every $\rel{}{}\in L_\mathcal R$ and $\func{}{}\in L_\mathcal F$ we have:
\begin{enumerate}
\item[(a)] $\nbrel{f(\str A)}{} = f_A(\permrel{f_L^{-1}}{A}{})$, and
\item[(b)] for every $\bar{x} \in f_A(A)^{\arityf{}}$ it holds that $$\nbfunc{f(\str A)}{}(\bar x) = \bigcup_{\bar{y}\in A^{\arityf{}}, \bar{x} = f(\bar{y})} f_A(\permfunc{f_L^{-1}}{A}{}(\bar{y})).$$
\end{enumerate}
Note that $f$ is a homomorphism $\str A\to f(\str A)$ and moreover all relations and functions of $f(\str A)$ are minimal possible for $f$ to be a homomorphism. Also observe that if $f_A$ is injective, then $f$ is an isomorphism $\str A\to f(\str A)$.

\medskip

\begin{definition}\label{defn:relabelling}
Let $L$ be a language equipped with a permutation group $\GammaL$. We define the \emph{action of $\GammaL$ on $\GammaL$-structures by relabelling}, such that for a $\GammaL$-structure $\str A$ and $g\in \GammaL$, we define $g\str A$ as $(g,\id_A)(\str A)$.
\end{definition}

\begin{observation}\label{obs:relabelling}
Let $L$ be a language equipped with a permutation group $\GammaL$. If $L$ is finite or $\GammaL=\{\id_L\}$ then every finite $\GammaL$-structure lies in a finite orbit of the action of $\GammaL$ by relabelling.
\end{observation}
\begin{proof}
If $L$ is finite, then there are only finitely many $\GammaL$-structures on any given finite set $A$ and the action of $\GammaL$ by relabelling preserves the vertex set. If $\GammaL=\{\id_L\}$, then the action is trivial and every orbit is a singleton.
\end{proof}

\subsection{$\GammaL$-structures as standard model-theoretic structures}\label{sec:gamma_to_nongamma}
Whenever there exists a structure $\str L$ such that $\Aut(\str L) = \GammaL$ (e.g. when $L$ is finite or more generally when $\GammaL$ is a closed subgroup of $\Sym(\mathbb N)$), there is a functorial correspondence between $\GammaL$-structures and structures in a bigger language with no permutation group. This makes it possible to extend many theorems about classical structures to (certain) $\GammaL$-structures without having to re-prove them.

\begin{definition}\label{defn:gamma_to_nongamma}
Let $L$ be a language equipped with a permutation group $\GammaL$. Let $X$ be the set of all symbols of $L$ which are not fixed by $\GammaL$ and assume that there is a language $L_0$ disjoint from $L$ and an $L_0$-structure $\str X$ such that $\Aut(\str X)$ is precisely the action of $\GammaL$ on $X$. Let $L^\circ$ be the language defined as follows:
\begin{enumerate}
	\item For every symbol from $L\setminus X$, we put the same symbol with the same arity into $L^\circ$.
	\item For every symbol from $L_0$, we put the same symbol with the same arity into $L^\circ$.
	\item For every $n$ such that $X$ contains a relation symbol of arity $n$, we put an $(n+1)$-ary relation symbol $\rel{}{n}$ into $L^\circ$ (without loss of generality $\rel{}{n}\notin L\cup L_0$).
	\item For every $n$ such that $X$ contains a function symbol of arity $n$, we put an $(n+1)$-ary function symbol $\func{}{n}$ into $L^\circ$ (without loss of generality $\func{}{n}\notin L\cup L_0$).
	\item There is a constant symbol $c\in L^\circ$ (without loss of generality $c\notin L\cup L_0$).
\end{enumerate}
Given a $\GammaL$-structure $\str A$, we define an $L^\circ$-structure $\str A^\circ$ as follows:
\begin{enumerate}
	\item The vertex set of $\str A^\circ$ is the disjoint union $A\cup X$ (without loss of generality we can assume that $A\cap X = \emptyset$).
	\item $c_{\str A^\circ} = X$.
	\item The substructure of $\str A^\circ$ induced on $X$ is isomorphic to $\str X$ (in particular, there are no relations or functions from $L$).
	\item For every symbol $S\in L\setminus X$ we have that $S_{\str A^\circ} = S_{\str A}$.
	\item For every $n$-ary relation symbol $S\in X$ and every tuple $(x_1,\ldots,x_n)\in A^n$ it holds that $(x_1,\ldots,x_n,S) \in \nbrel{\str A^\circ}{n}$ if and only if $(x_1,\ldots,x_n)\in S_\str{A}$.
	\item For every $n$-ary function symbol $S\in X$ and every tuple $(x_1,\ldots,x_n)\in A^n$ it holds that $(x_1,\ldots,x_n,S) \in \dom(\nbfunc{\str A^\circ}{n})$ if and only if $(x_1,\ldots,x_n)\in \dom(S_\str A)$, and in that case $\nbfunc{\str A^\circ}{n}(x_1,\ldots,x_n,S) = S_\str A(x_1,\ldots,x_n)$.
\end{enumerate}
\end{definition}

\begin{fact}
In the setting of Definition~\ref{defn:gamma_to_nongamma}, $f=(f_L,f_A)$ is an embedding of $\GammaL$-structures $\str A\to\str B$ if and only if $f_L\restriction_X \cup f_A$ is an embedding of $L^\circ$ structures $\str A^\circ \to \str B^\circ$.
\end{fact}
This implies that the map $\str A\mapsto \str A^\circ$ is an isomorphism of categories. Note that whenever $\GammaL$ is finite, we have that $\str A$ is finite if and only if $\str A^\circ$ is. This construction still gives structures where the images of functions need not consist of singletons. In order to deal with this, one can replace functions by relations and consider only algebraically closed substructures as is standard in the area, see for example~\cite{EvansBonn}.


\subsection{Amalgamation classes}\label{sec:amalg}
\begin{figure}
\centering
\includegraphics{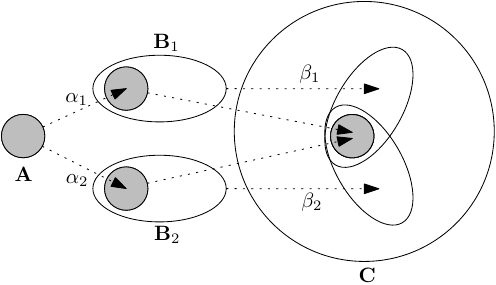}
\caption{An amalgamation of $\str{B}_1$ and $\str{B}_2$ over $\str{A}$.}
\label{amalgamfig}
\end{figure}
Let $\str{A}$, $\str{B}_1$, and $\str{B}_2$ be $\GammaL$-structures, and let $\alpha_1\colon\str{A}\to\str{B}_1$, $\alpha_2\colon\str{A}\to\str{B}_2$ be embeddings. A structure $\str{C}$
 with embeddings $\beta_1\colon\str{B}_1 \to \str{C}$ and
$\beta_2\colon\str{B}_2\to\str{C}$ such that $\beta_1\circ\alpha_1 =
\beta_2\circ\alpha_2$ (remember that this must also hold for the language part of $\alpha_i$'s and $\beta_i$'s) is called an \emph{amalgamation} of $\str{B}_1$ and $\str{B}_2$ over $\str{A}$ with respect to $\alpha_1$ and $\alpha_2$, see Figure~\ref{amalgamfig}.
We will often call $\str{C}$ simply an \emph{amalgamation} of $\str{B}_1$ and $\str{B}_2$ over $\str{A}$
(in most cases $\alpha_1$ and $\alpha_2$ can be chosen to be inclusion embeddings).

We say that the amalgamation is \emph{strong} if it holds that $\beta_1(x_1)=\beta_2(x_2)$ if and only if $x_1\in \alpha_1(A)$ and $x_2\in \alpha_2(A)$.
Strong amalgamation is \emph{free} if $C=\beta_1(B_1)\cup \beta_2(B_2)$, and whenever a tuple $\bar{x}$ of vertices of $\str{C}$ contains vertices of both
$\beta_1(B_1\setminus \alpha_1(A))$ and $\beta_2(B_2\setminus \alpha_2(A))$, then $\bar{x}$ is in no relation of $\str C$
and also for every function $\func{}{}\in L$ with $\arityf{} = |\bar{x}|$ it holds that $\func{C}{}(\bar{x})=\emptyset$.

\begin{definition}
\label{defn:amalg}
An \emph{amalgamation class} is a class $\K$ of finite $\GammaL$-structures which is closed for isomorphisms and satisfies the following three conditions:
\begin{enumerate}
\item {\em Hereditary property:} For every $\str{B}\in \K$ and every structure $\str{A}$ with an embedding $f\colon \str A \to \str{B}$ we have $\str{A}\in \K$;
\item {\em Joint embedding property:} For every $\str{A}, \str{B}\in \K$ there exists $\str{C}\in \K$ with an embeddings $f\colon \str A\to \str C$ and $g\colon \str B\to \str C$;
\item {\em Amalgamation property:} 
For $\str{A},\str{B}_1,\str{B}_2\in \K$ and embeddings $\alpha_1\colon\str{A}\to\str{B}_1$, $\alpha_2\colon\str{A}\to\str{B}_2$, there is $\str{C}\in \K$ which is an amalgamation of $\str{B}_1$ and $\str{B}_2$ over $\str{A}$ with respect to $\alpha_1$ and $\alpha_2$.
\end{enumerate}
If the $\str{C}$ in the amalgamation property can always be chosen to be a strong amalgamation then $\K$ is a {\em strong amalgamation class}, if it can always be chosen to be the free amalgamation then $\K$ is a {\em free amalgamation class}.
\end{definition}

By the \Fraisse{} theorem~\cite{Fraisse1953}, relational amalgamation classes in a countable language with trivial $\GammaL$ containing only countably many members up to isomorphism correspond to countable homogeneous structures. By Section~\ref{sec:gamma_to_nongamma}, one can extend this to various languages equipped with a permutation group using a variant of the \Fraisse{} theorem for languages with functions or for strong substructures, see for example~\cite{EvansBonn}.

\medskip

Generalising the notion of a graph clique, we say that a structure is
\emph{irreducible} if it is not a free amalgamation of its proper
substructures. A homomorphism $f\colon\str{A}\to \str{B}$ is
a \emph{homomorphism-embedding} if the restriction $f\restriction_{\str C}$ is an embedding whenever $\str C$ is an irreducible
substructure of $\str{A}$. Given a family $\mathcal F$ of $\GammaL$-structures, we denote by
$\Forb(\mathcal F)$ the class of all finite or countably infinite $\GammaL$-structures $\str{A}$
such that there is no $\str{F}\in \mathcal F$ with a homomorphism-embedding $\str{F}\to\str{A}$.

\subsection{EPPA for $\GammaL$-structures}\label{sec:gammaleppa}
A {\em partial automorphism} of a $\GammaL$-structure $\str{A}$ is a partial isomorphism between $\str A$ and $\str A$.
Let $\str A$ and $\str B$ be finite $\GammaL$-structures. We say that $\str B$
is an \emph{EPPA-witness for $\str A$} if there is an embedding $\psi\colon \str A\to \str B$
and every partial automorphism of $\psi(\str A)$ \emph{extends} to an automorphism of $\str B$, that is, for every partial automorphism $\varphi$ of $\psi(\str A)$ there is an automorphism $\widetilde{\varphi}\colon \str B\to \str B$ such that $\varphi\subseteq \widetilde{\varphi}$.

We say that a class of finite $\GammaL$-structures $\K$  has the {\em
extension property for partial automorphisms}  (shortly {\em EPPA}, sometimes called
the {\em Hrushovski property}) if for every $\str{A} \in \K$
there is $\str{B} \in \K$ which is an EPPA-witness for $\str A$.
Such a structure $\str{B}$ is \emph{irreducible structure faithful} (with respect to $\psi(\str A)$) if it has the property
that for every irreducible substructure $\str{C}$ of $\str{B}$ there exists an 
automorphism $g$ of $\str{B}$ such that $g(C)\subseteq \psi(A)$.

Note that the classes which we are interested in are closed under taking isomorphisms, and hence if there is an EPPA-witness
$\str B$ for $\str A$ in $\K$, then there is also an EPPA-witness $\str B'\in \K$ such that $\psi$ is just the inclusion $\str A\subseteq \str B'$. To simplify the arguments, we will often ignore this subtle technicality.

Homomorphism-embeddings were introduced in~\cite{Hubicka2016} and irreducible structure faithfulness was introduced in~\cite{Evans3} as a generalisation of clique faithfulness of Hodkinson and Otto~\cite{hodkinson2003}.
The following observation provides a link to the study of homogeneous structures.

\begin{observation}[\cite{herwig1998}]\label{obs:eppaamalgamation}
Every hereditary isomorphism-closed class of finite $\GammaL$-structures which has EPPA and the joint embedding property (see Definition~\ref{defn:amalg}) is an amalgamation class. 
\end{observation}
\begin{proof}
Let $\K$ be such a class and let $\str{A},\str{B}_1,\str{B}_2\in \K$, $\alpha_1\colon \str A \to \str B_1$, $\alpha_2\colon \str A\to \str B_2$ be as in Definition~\ref{defn:amalg}. Let $\str B$ be the joint embedding of $\str B_1$ and $\str B_2$ (that is, we have embeddings $\beta_1'\colon \str B_1\to \str B$ and $\beta_2'\colon \str B_2\to \str B$) and let $\str C$ be an EPPA-witness for $\str B$. Without loss of generality, we can assume that $\str B\subseteq \str C$.

Let $\varphi$ be a partial automorphism of $\str B$ sending $\beta_1'(\alpha_1(\str A))$ to $\beta_2'(\alpha_2(\str A))$ and let $\theta$ be its extension to an automorphism of $\str C$. Finally, put $\beta_1 = \theta\circ\beta_1'$ and $\beta_2 = \beta_2'$. It is easy to check that $\beta_1$ and $\beta_2$ certify that $\str C$ is an amalgamation of $\str B_1$ and $\str B_2$ over $\str A$ with respect to $\alpha_1$ and $\alpha_2$.
\end{proof}

We believe that the results of this paper may often be used in a more specialised setting such as for standard model-theoretic $L$-structures etc. In order to facilitate that, we state the following simple observation which allows translating the main theorems into this more specialised setting.

\begin{observation}\label{obs:translate_to_standard}
Let $\str A$ and $\str B$ be finite $\GammaL$-structures such that $\str B$ is an irreducible structure faithful EPPA-witness for $\str A$. If, in $\str A$, the range of every function consists of singletons then this also holds in $\str B$.
\end{observation}
\begin{proof}
Suppose for a contradiction that there is a function $\func{}{}\in L$ and a tuple $\bar{x}\in B^{\arityf{}}$ such that $\func{B}{}(\bar{x}) = X$ and $\lvert X\rvert > 1$. Since $\cl_\str B(\bar{x})$ is irreducible, by irreducible structure faithfulness there is an automorphism $f\colon \str B\to \str B$ sending $\cl_\str B(\bar{x})$ to $A$. In particular, $f(\bar{x})\subseteq A$ and $f(X)\subseteq A$. As $f$ is an automorphism, we know that $\permfunc{f_L}{A}{}(f_B(\bar{x})) = f_B(X)$ implying that $\lvert f_B(X)\rvert \leq 1$ which is a contradiction.
\end{proof}

\subsection{EPPA and automorphism groups}
As was mentioned in the introduction, the significance of EPPA comes from the fact that, while being a property of a class of finite structures, it is closely connected with topological properties of the automorphism group of an infinite structure, namely the \Fraisse{} limit of the class. We do not aim for this section to be self-contained nor complete (and refer the reader for example to~\cite{Siniora2}), we only outline some of these connections and discuss how they extend to $\GammaL$-structures. In contrast to the rest of this paper, in this section we are mostly going to be interested in (countably) infinite structures. For simplicity, we will assume that $\GammaL$ is finite (so that Section~\ref{sec:gamma_to_nongamma} can be fully applied), the case of infinite $\GammaL$ can be more complicated and deserves a study on its own.

Let $\str M$ be a countably infinite $\GammaL$-structure. Assume without loss of generality that the vertex set of $\str M$ are the natural numbers and put $G = \Aut(\str M)$. Remember that members of $G$ are pairs $f = (f_L, f_M)$ with $f_L\in \GammaL$. This means that $G$ can be understood as a subgroup of $\GammaL\times\Sym(\mathbb N)$ and as such inherits the topology of this product, where $\Sym(\mathbb N)$ is equipped with the pointwise-convergence topology and $\GammaL$, being finite, is equipped with the discrete topology. Note that this precisely corresponds to what one gets by using Section~\ref{sec:gamma_to_nongamma} and recalling the standard definitions.

\medskip

Let $\str M$ be a homogeneous $\GammaL$-structure. We say that $\str M$ is \emph{locally finite} if for every finite $X\subseteq M$ it holds that $\cl_\str M(X)$ is also finite. By $\age(\str M)$ we denote the class of all finite $\GammaL$-structures which embed into $\str M$. The following theorem has been proved by Kechris and Rosendal~\cite{Kechris2007} for classical structures and by Section~\ref{sec:gamma_to_nongamma} extends naturally to $\GammaL$-structures:

\begin{theorem}[\cite{Kechris2007}]
Let $L$ be a language equipped with a finite permutation group $\GammaL$ and let $\str M$ be a countable locally finite homogeneous $\GammaL$-structure. Then $\age(\str M)$ has EPPA if and only if $\Aut(\str M)$ can be written as the closure of a chain of compact subgroups. Moreover, if $\age(\str M)$ has EPPA, then $\Aut(\str M)$ is amenable.
\end{theorem}

\begin{definition}[\cite{hodges1993b,Kechris2007}]
Let $L$ be a language equipped with a finite permutation group $\GammaL$, let $\str M$ be a countable locally finite homogeneous $\GammaL$-structure and let $n\geq 1$ be an integer. We say that $\str M$ has \emph{$n$-generic automorphisms} if $G$ has a comeagre orbit on $G^n$ in its action by diagonal conjugation. We say that $\str M$ has \emph{ample generics} if it has $n$-generic automorphisms for every $n\geq 1$.
\end{definition}
Here, the action by diagonal conjugation is defined by $$g\cdot (h_1,\ldots,h_n) = (gh_1g^{-1},\ldots,gh_ng^{-1}).$$ The existence of ample generics has many consequences for the automorphism group such as the small index property. From the point of view of this paper, ample generics are relevant, because EPPA is very often a key ingredient in proving them. We will outline this connection in the rest of this section.

\begin{definition}
Let $L$ be a language equipped with a permutation group $\GammaL$, let $\mathcal C$ be a class of finite $\GammaL$-structures and let $n\geq 1$ be an integer. An \emph{$n$-system over $\mathcal C$} is a tuple $(\str A, p_1,\ldots,p_n)$, where $\str A\in \mathcal C$ and $p_1,\ldots,p_n$ are partial automorphisms of $\str A$. We denote by $\mathcal C^n$ the class of all $n$-systems over $\mathcal C$.

If $P=(\str A,p_1,\ldots,p_n)$ and $Q=(\str B,q_1,\ldots,q_n)$ are both $n$-systems over $\mathcal C$ and $f\colon \str A\to\str B$ is an embedding of $\GammaL$-structures, we say that $f$ is an \emph{embedding of $n$-systems $P\to Q$} if for every $1\leq i \leq n$ it holds that $f\circ p_i \subseteq q_i\circ f$ (in particular, $f(\dom(p_i))\subseteq \dom(q_i)$ and $f(\range(p_i))\subseteq \range(q_i)$).
\end{definition}

\begin{definition}
Let $L$ be a language equipped with a permutation group $\GammaL$, let $\mathcal C$ be a class of finite $\GammaL$-structures and let $n\geq 1$ be an integer. We say that $\mathcal C^n$ has the \emph{joint embedding property} if for every $P,Q\in\mathcal C^n$ there exists $S\in \mathcal C^n$ with embeddings of $n$-systems $f\colon P\to S$ and $g\colon Q\to S$. We say that $\mathcal C^n$ has the \emph{weak amalgamation property} if for every $T\in \mathcal C^n$ there exists $\hat{T}\in \mathcal C^n$ and an embedding of $n$-systems $\iota\colon T\to\hat{T}$ such that for every pair of $n$-systems $P,Q\in \mathcal C^n$ and embeddings of $n$-systems $\alpha_1\colon \hat{T}\to P$ and $\alpha_2\colon \hat{T}\to Q$ there exists $S\in \mathcal C^n$ with embeddings on $n$-systems $\beta_1\colon P\to S$ and $\beta_2\colon Q\to S$ such that $\beta_1\alpha_1\iota = \beta_2\alpha_2\iota$.
\end{definition}

We only state the following theorem for $\GammaL=\{\id\}$, as $\mathcal C^n$ does not have the joint embedding property for any $n$ if $\mathcal C$ is a class of finite $\GammaL$-structures with $2\leq \lvert\GammaL\rvert < \infty$ (see Example~\ref{ex:non_jep}).
\begin{theorem}[\cite{Kechris2007}]\label{thm:ample}
Let $L$ be a language, let $\str M$ be a countable locally finite homogeneous $L$-structure, put $\mathcal C = \age(\str M)$ and fix $n\geq 1$. Then $\str M$ has $n$-generic automorphisms if and only if $\mathcal C^n$ has the joint embedding property and the weak amalgamation property.
\end{theorem}

In order to explain the connection between EPPA and ample generics, we need one more definition

\begin{definition}
Let $L$ be a language equipped with a permutation group $\GammaL$ and let $\mathcal C$ be a class of finite $\GammaL$-structures. We say that $\mathcal C$ has the \emph{amalgamation property with automorphisms} (abbreviated as \emph{APA}) if for every $\str{A},\str{B}_1,\str{B}_2\in \mathcal C$ and embeddings $\alpha_1\colon\str{A}\to\str{B}_1$, $\alpha_2\colon\str{A}\to\str{B}_2$ there exists $\str{C}\in \mathcal C$ with embeddings $\beta_1\colon\str{B}_1 \to \str{C}$ and $\beta_2\colon\str{B}_2\to\str{C}$ such that $\beta_1\circ\alpha_1 = \beta_2\circ\alpha_2$ (i.e. $\str C$ is an amalgamation of $\str B_1$ and $\str B_2$ over $\str A$ with respect to $\alpha_1$ and $\alpha_2$) and moreover whenever we have $f\in \Aut(\str B_1)$ and $g\in \Aut(\str B_2)$ such that $f(\alpha_1(A)) = \alpha_1(A)$, $g(\alpha_2(A)) = \alpha_2(A)$ and for every $a\in A$ it holds that $\alpha_1^{-1}(f(\alpha_1(a))) = \alpha_2^{-1}(g(\alpha_2(a)))$ (that is, $f$ and $g$ agree on the copy of $\str A$ we are amalgamating over), then there is $h\in \Aut(\str C)$ which extends $\beta_1 f \beta_1^{-1} \cup\beta_2 g \beta_2^{-1}$. We call such $\str C$ with embeddings $\beta_1$ and $\beta_2$ an \emph{APA-witness} for $\str B_1$ and $\str B_2$ over $\str A$ with respect to $\alpha_1$ and $\alpha_2$
\end{definition}

\begin{prop}[\cite{Kechris2007}]\label{prop:eppa_apa_weak}
Let $L$ be a language equipped with a permutation group $\GammaL$ and let $\mathcal C$ be a class of finite $\GammaL$-structures. If $\mathcal C$ has EPPA and APA then $\mathcal C^n$ has the weak amalgamation property for every $n\geq 1$.
\end{prop}
\begin{proof}
Fix $n\geq 1$. If $S = (\str S,s_1,\ldots,s_n) \in \mathcal C^n$ is an $n$-system, we denote by $\hat{S} = (\hat{\str S}, \hat{s}_1,\ldots,\hat{s}_n)\in \mathcal C^n$ the $n$-system where $\hat{\str S}$ is an EPPA-witness for $\str S$ (with respect to the inclusion embedding) and for every $1\leq i\leq n$ it holds that $\hat{s}_i$ is an automorphism of $\hat{\str S}$ extending $s_i$.

We now prove that $\mathcal C^n$ has the weak amalgamation property. Towards that, fix some $T = (\str T, t_1,\ldots,t_n)\in \mathcal C^n$. Let $P=(\str P,p_1,\ldots,p_n),Q=(\str Q,q_1,\ldots,q_n)\in \mathcal C^n$ be arbitrary $n$-systems with embeddings $\alpha_1\colon \hat{T}\to P$ and $\alpha_2\colon \hat{T}\to Q$.

Use APA for $\mathcal C$ to get $\str S\in \mathcal C$ and embeddings $\beta_1\colon\hat{\str P} \to \str S$ and $\beta_2\colon\hat{\str Q}\to\str S$ such that $\str S$ with $\beta_1$ and $\beta_2$ form an APA-witness for $\hat{\str P}$ and $\hat{\str Q}$ over $\hat{\str{T}}$ with respect to $\alpha_1$ and $\alpha_2$. Let $S=(\str S,s_1,\ldots,s_n)\in\mathcal C^n$ be some $n$-system such that for every $1\leq i \leq n$ we have that $s_i$ extends $\beta_1 \hat{p}_i \beta_1^{-1} \cup\beta_2 \hat{q}_i \beta_2^{-1}$. It is straightforward to verify that $S$ is the desired $n$-system witnessing the weak amalgamation property for $P$, $Q$ and $T$.
\end{proof}

\begin{example}
Consider the class $\mathcal C$ of all finite graphs. By a theorem of Hrushov\-ski~\cite{hrushovski1992} (or by Section~\ref{sec:graphs}) we know that $\mathcal C$ has EPPA. APA for $\mathcal C$ is an easy exercise (in general, APA for free amalgamation classes is always true). Hence, by Proposition~\ref{prop:eppa_apa_weak}, $\mathcal C^n$ has the weak amalgamation property for every $n\geq 1$. To prove ample generics for the countable random graph it thus remains to prove the joint embedding property for $\mathcal C^n$. However, it is again an easy exercise (simply take the disjoint union of the graphs and the partial automorphisms).
\end{example}

\begin{example}\label{ex:non_jep}
Let $L$ be a language consisting of two unary relations $U$ and $V$, put $\GammaL=\Sym(L)$ and let $\mathcal C$ be the class of all finite $\GammaL$-structures where every vertex is in precisely one of the two unary relations. Clearly, $\mathcal C$ can equivalently be seen as the class of all finite structures with one equivalence relation with two equivalence classes. Since $\mathcal C$ is a free amalgamation class, Corollary~\ref{cor:free} gives us EPPA for $\mathcal C$, APA for $\mathcal C$ is straightforward. Hence, by Proposition~\ref{prop:eppa_apa_weak}, $\mathcal C^n$ has the weak amalgamation property for every $n\geq 1$.

However, $\mathcal C^n$ fails to have the joint embedding property for any $n\geq 1$ and this is already visible on $n$-systems with the empty structure. Let $\str E$ be the $\GammaL$-structure with no vertices and assume that $\GammaL$ is enumerated as $\{\id, t\}$ where $\id$ is the identity and $t$ is the transposition $U\leftrightarrow V$. Put $P = (\str E, (\id, \emptyset))$ and $Q = (\str E, (t,\emptyset))$. Clearly, there is no 1-system which embeds both $P$ and $Q$.

However, this kind of obstacle is the only reason why $\mathcal C^n$ does not have the joint embedding property (in general for free amalgamation classes): One can define an equivalence relation $\sim_n$ on $\mathcal C^n$ for every $n$ and for every pair of $n$-systems $P = (\str P, (p^1_L, p^1_P),\ldots,(p^n_L, p^n_P))\in \mathcal C^n$ and $Q = (\str Q, (q^1_L, q^1_Q),\ldots,(q^n_L, q^n_Q))\in \mathcal C^n$ by putting $P\sim_n Q$ if and only if there is $f\in \GammaL$ such that for every $1\leq i\leq n$ we have that $f\circ p^i_L = q^i_L\circ f$. Then $P,Q\in \mathcal C^n$ have a joint embedding if and only if $P\sim_n Q$. In fact, it is then possible to prove a relativised version of Theorem~\ref{thm:ample} and obtain generic automorphisms for every equivalence class of $\sim_n$. For example, if $\mathcal C$ is the class from this example, $n=1$ and the language part is the identity, then the generic automorphism is just a pair of permutations of vertices of each unary such that both of them have no infinite cycles and infinitely many $k$-cycles for every finite $k\geq 1$.
\end{example}

\subsection{Coherence of EPPA-witnesses}\label{sec:coherence}
Siniora and Solecki~\cite{solecki2009,Siniora} strengthened the notion of EPPA in order to get a dense locally finite subgroup of the automorphism group of the corresponding \Fraisse{} limit.
In order to state their definitions, we need to define how partial maps compose. Let $L$ be a language equipped with a permutation group $\GammaL$, let $\str A$, $\str B$ and $\str C$ be $\GammaL$-structures, let $f$ be a partial isomorphism between $\str A$ and $\str B$ and let $g$ be a partial isomorphism between $\str B$ and $\str C$ such that $\dom(g_B) = \range(f_A)$. We define their composition $gf$ (also denoted by $g\circ f$) to be the partial isomorphism between $\str A$ and $\str C$ such that $(gf)_L = g_Lf_L$, $(gf)_A(x)$ is defined if and only if $x\in \dom(f_A)$ and $f_A(x)\in \dom(g_B)$, and in this case we put $(gf)_A(x)=g_A(f_B(x))$.

\begin{definition}[Coherent maps]
Let $L$ be a language equipped with a permutation group $\GammaL$, let $\str A$ be a $\GammaL$-structure and let $\mathcal P$ be a family of partial automorphisms of $\str A$. A triple $f,g,h\in \mathcal P$ is called a \emph{coherent triple} if $\range(f_A)=\dom(g_A)$ and $h=gf$. A pair $f,g\in \mathcal P$ is called a \emph{coherent pair} if there is $h\in\mathcal P$ such that $f,g,h$ is a coherent triple.

Let $\str A$ and $\str B$ be $\GammaL$-structures, and let $\mathcal P$ and $\mathcal Q$ be families of partial automorphisms of $\str A$ and $\str B$, respectively. A function $\varphi\colon \mathcal P \to \mathcal Q$ is said to be a
{\em coherent map} if for each coherent triple $(f, g, h)$ from $\mathcal P$, its image ($\varphi(f), \varphi(g), \varphi(h)$) in $\mathcal Q$ is also coherent.
\end{definition}
\begin{definition}[Coherent EPPA]
\label{defn:coherent}
A class $\K$ of finite $\GammaL$-structures is said to have {\em coherent EPPA} if $\K$ has EPPA and moreover the extension of partial automorphisms
is coherent. That is, for every $\str{A} \in \K$, there exists $\str{B} \in \K$ and an embedding $\psi\colon \str A\to \str B$ such that every
partial automorphism $f$ of $\psi(\str{A})$ extends to some $\hat{f} \in \Aut(\str{B})$ with the property that the map $f\mapsto \hat{f}$ from partial automorphisms of $\psi(\str{A})$ to $\Aut(\str{B})$ is coherent. We say that $\str B$ is a \emph{coherent EPPA-witness} for $\str A$.
\end{definition}

The following easy proposition will be used several times. We include its proof to make this paper self-contained.
\begin{prop}[Lemma~2.1 in~\cite{Siniora}]\label{prop:setcoherence}
Every finite set is a coherent EPPA-witness for itself. Consequently, the class of all finite sets has coherent EPPA.

Explicitly, for every finite set $A$ there is a map assigning to every partial injective function $\varphi\colon A\to A$ a permutation $\widehat{\varphi}$ of $A$ such that $\varphi\subseteq \widehat{\varphi}$ and moreover for every coherent pair $\varphi_1,\varphi_2\colon A\to A$ it holds that $\widehat{\varphi_2}\widehat{\varphi_1}=\widehat{\varphi_2\varphi_1}$.
\end{prop}
\begin{proof}
Fix a set $A$. Without loss of generality we can assume that $A=\{1,\ldots, n\}$. Let $\varphi$ be a partial automorphism of $A$, in other words, a partial injective function $A\to A$. We construct a permutation $\widehat{\varphi}\colon A\to A$ extending $\varphi$ in the following way: 

Put $X = A\setminus \dom(\varphi)$ and $Y=A\setminus \range(\varphi)$ and enumerate $X=\{x_1,\ldots, x_k\}$ and $Y=\{y_1,\ldots,y_k\}$ such that $x_1<\cdots<x_k$ and $y_1<\cdots<y_k$. Define $\widehat{\varphi}$ by
$$\widehat{\varphi}(x) = 
\begin{cases}
\varphi(x) & \text{if }x\in \dom(\varphi)\\
y_i & \text{if }x = x_i.
\end{cases}$$
It is obvious that $\widehat{\varphi}$ is a permutation of $A$ which extends $\varphi$. Thus it only remains to prove coherence.

Consider $x\in A$. If $x\in \dom(\varphi_1)$, then we have $\varphi_1(x)\in \dom(\varphi_2)$ and hence $\widehat{\varphi_2\varphi_1}(x) = \widehat{\varphi_2}(\widehat{\varphi_1}(x))$. Put $X = A\setminus\dom(\varphi_1)$, $Y=A\setminus\range(\varphi_1)$ ($=A\setminus\dom(\varphi_2)$) and $Z=A\setminus\range(\varphi_2)$ and again enumerate them in an ascending order. If $x=x_i$, we have $\widehat{\varphi_1}(x_i)=y_i$, $\widehat{\varphi_2}(y_i)=z_i$ and $\widehat{\varphi_2\varphi_1}(x_i)=z_i$, therefore indeed $\widehat{\varphi_2\varphi_1}(x_i) =\widehat{\varphi_2}(\widehat{\varphi_1}(x_i))$.
\end{proof}
When using this result, we will often simply say that we extend a partial permutation in an \emph{order-preserving way} or \emph{coherently}.

\section{Warm-up: new proof of EPPA for graphs}
\label{sec:graphs}
We start with a simple proof of the theorem of Hrushovski~\cite{hrushovski1992}. (Our proof is different from another simple proof given
by Herwig and Lascar~\cite[Section 4.1]{herwig2000}.)
This is the simplest case where the construction of coherent EPPA-witnesses is non-trivial and we encourage the reader to spend enough time on this section, as it can provide a very useful intuition for the subsequent sections.
We consider graphs to be (relational) structures in a language with a single
binary relation $E$ which is symmetric and irreflexive.

Fix a graph $\str{A}$ with vertex set $A=\{1,\ldots, n\}$.

\addtocontents{toc}{\SkipTocEntry}
\subsection*{Witness construction}
We give a construction of a coherent EPPA-witness $\str{B}$. It will be constructed as follows:

\begin{enumerate}
  \item The vertices of $\str{B}$ are all pairs $(x,\chi)$ where $x\in A$ and
    $\chi$ is a function from $A\setminus \{x\}$ to $\{0,1\}$ (called a
    \emph{valuation function for $x$}).
  \item Vertices $(x,\chi)$ and $(x',\chi')$ form an edge of $\str B$ if and only if
    $x\neq x'$ and $\chi(x')\neq \chi'(x)$.
\end{enumerate}
We now introduce
a \emph{generic copy} $\str{A}'$ of $\str{A}$ in $\str{B}$ using an embedding $\psi\colon \str{A}\to \str{B}$ defined by $\psi(x)=(x,\chi_x)$, where $\chi_x(y)=1$ if $x > y$ and $\{x,y\}\in E_\str{A}$ and $\chi_x(y)=0$ otherwise (remember that we enumerated $A=\{1,\ldots, n\}$).
We put $\str{A}'$ to be the graph induced by $\str{B}$ on $\psi(A)$.
It follows directly that $\psi$ is indeed an embedding of $\str{A}$ into $\str{B}$.

\begin{remark}
Note that the functions $\chi_x$ from the definition of $\psi$ are in fact the rows of an asymmetric variant of the adjacency matrix of $\str A$.
\end{remark}

Let $\pi\colon B \to A$ be the \emph{projection} mapping $(x,\chi) \mapsto x$. Note that $\pi(\psi(x)) = x$ for every $x\in \str A$. This means that $\str A'$ is \emph{transversal}, that is, $\pi$ is injective on $A'$.

\addtocontents{toc}{\SkipTocEntry}
\subsection*{Constructing the extension}
\begin{figure}
\centering
\includegraphics{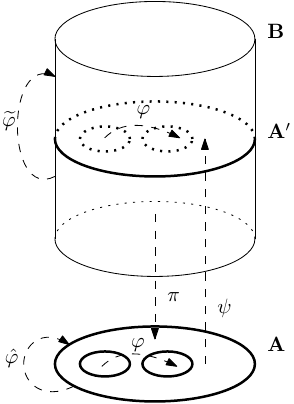}
\caption{Scheme of the construction of $\widetilde{\varphi}$.}
\label{prooffig}
\end{figure}
The construction from the following paragraphs is schematically depicted in Figure~\ref{prooffig}.

Let $\varphi$ be a partial automorphism of $\str{A}'$. Using $\pi$ we get a partial permutation of (the set) $A$ and we denote by $\hat\varphi$ its order-preserving extension to a permutation of $A$ (cf. Proposition~\ref{prop:setcoherence}). 

We now construct a set $F\subseteq {A\choose 2}$ of \emph{flipped pairs} by putting $\{x,y\}\in F$ if $x\neq y$, $(x,\chi_x)\in \dom(\varphi)$ and $\chi_x(y)\neq \chi'(\hat\varphi(y))$, where $\varphi ((x,\chi_x))=(\hat\varphi(x),\chi')$. Note that if we also have that $(y,\chi_y)\in\dom(\varphi)$, then $\{x,y\}\in F$ if and only if $\chi_y(x)\neq \chi''(\hat\varphi(x))$, where $\varphi((y,\chi_y)) = (\hat\varphi(y),\chi'')$. This follows from the fact that $\varphi$ is a partial automorphism.

For every $x\in A$ we define a function $f_x$ on valuation functions for $x$ putting
$$f_x(\chi)(\hat\varphi(y)) = \begin{cases}
  \chi(y) &  \text{if }\{x,y\}\notin F \\
  1-\chi(y) & \text{if }\{x,y\}\in F.
\end{cases}$$
Finally, we define a function $\widetilde{\varphi}\colon B\to B$ by putting $\widetilde{\varphi}((x,\chi))=(\hat\varphi(x),f_x(\chi))$. This function will be the coherent extension of $\varphi$.

\addtocontents{toc}{\SkipTocEntry}
\subsection*{Proofs}
Both proofs in this section are only an explicit verification that our constructions work as expected.
\begin{lemma}\label{lem:graphs:auto}
$\widetilde{\varphi}$ is an automorphism of $\str B$ extending $\varphi$. In other words, $\str B$ is an EPPA-witness for $\str A$.
\end{lemma}
\begin{proof}
Clearly, $\hat\varphi$ has an inverse $\hat\varphi^{-1}$. Observe also that the function $f^{-1}_x$ defined as
$$
f^{-1}_x(\chi)(\hat\varphi^{-1}(y))=
\begin{cases}
  \chi(y) &  \text{if }\{\hat\varphi^{-1}(x),\hat\varphi^{-1}(y)\}\notin F \\
  1-\chi(y) & \text{if }\{\hat\varphi^{-1}(x),\hat\varphi^{-1}(y)\}\in F
\end{cases}
$$
is an inverse of $f_x$. It follows that $\widetilde{\varphi}$ is a bijection $B\to B$.


Let $(x,\chi)$ and $(y,\xi)$ be vertices of $\str B$.
If $x=y$, then by the definition of $\str B$ neither of $(x,\chi),(y,\xi)$ and $\widetilde{\varphi}((x,\chi)),\widetilde{\varphi}((y,\xi))$
form an edge. If $x\neq y$, then we have $f_x(\chi)(\hat\varphi(y)) \neq f_y(\xi)(\hat\varphi(x))$ if and only if $\chi(y)\neq \xi(x)$ (by the definition of $f_x$ and $f_y$), hence $\widetilde{\varphi}$ preserves both edges and non-edges, that is, it is an automorphism of $\str B$.

Let $(x,\chi_x)\in \dom(\varphi)$ with $\varphi((x,\chi_x))=(z,\chi_z)$. We have for every $y\in A$ that $\{x,y\}\in F$ if and only if $\chi_x(y) \neq \chi_z(\hat\varphi(y))$. By the definition of $f_x$, it follows that $\chi_x(y) \neq f_x(\chi)(\hat\varphi(y))$ if and only if $\{x,y\}\in F$, therefore $\chi_z = f_x(\chi)$. This means that $\widetilde{\varphi}$ indeed extends $\varphi$.
\end{proof}

\begin{lemma}\label{lem:graphs:coherence}
Let $\varphi_1$, $\varphi_2$ and $\varphi$ be partial automorphisms of $\str A$ such that $\varphi = \varphi_2\circ\varphi_1$ and $\widetilde{\varphi_1}$, $\widetilde{\varphi_2}$ and $\widetilde{\varphi}$ their corresponding extensions as above. Then $\widetilde{\varphi} = \widetilde{\varphi_2}\circ \widetilde{\varphi_1}$.
\end{lemma}
\begin{proof}
Denote by $\hat\varphi_1$, $\hat\varphi_2$ and $\hat\varphi$ the
corresponding permutations of $A$ constructed above, and by $F_1$, $F_2$ and $F$ the corresponding sets of flipped pairs.

By Proposition~\ref{prop:setcoherence} we get that $\hat\varphi=\hat\varphi_2\circ\hat\varphi_1$. To see
that $\widetilde{\varphi}$ is a composition of $\widetilde{\varphi_1}$ and $\widetilde{\varphi_2}$ it remains to
verify that pairs flipped by $\widetilde{\varphi}$ are precisely
those pairs that are flipped by the composition of $\widetilde{\varphi_1}$ and $\widetilde{\varphi_2}$.

This follows from the construction of $F$. 
Only pairs with at least one vertex in the domain of $\varphi_1$ are put into sets $F$ and $F_1$
and again only pairs with at least one vertex in the domain of $\varphi_2$ (which is the same
as the value range of $\varphi_1$) are put into $F_2$.

Consider $\{x,y\}\in F$. This means that at least one of them (without loss of generality $x$) is in $\pi(\dom(\varphi)) = \pi(\dom(\varphi_1))$. Furthermore we know that $f_x(\chi_x)(\hat\varphi(y))\neq \chi_x(y)$. Because $\varphi = \varphi_2\circ\varphi_1$, we get that either $\{x,y\}\in F_1$, or $\{\hat\varphi_1(x), \hat\varphi_1(y)\}\in F_2$ (and precisely one of these happens). And this means that both $\widetilde{\varphi}$ and $\widetilde{\varphi_2}\circ\widetilde{\varphi_1}$ flip $\{x,y\}$.

On the other hand, if $\{x,y\}\notin F$, then either $\{x,y\}$ is in both $F_1$ and $F_2$ or in neither of them and then, again, neither $\widetilde{\varphi}$ nor $\widetilde{\varphi_2}\circ\widetilde{\varphi_1}$ flip $\{x,y\}$. This implies that indeed $\widetilde{\varphi} = \widetilde{\varphi_2}\circ \widetilde{\varphi_1}$.
\end{proof}

The previous lemmas immediately imply the following proposition.
\begin{prop}
\label{prop:graphs}
The graph $\str{B}$ is a coherent EPPA-witness for $\str{A}'$.
\end{prop}

\begin{remark}
Note that $\str B$ only depends on the number of vertices of $\str A$ and, as such, is a coherent EPPA-witness for all graphs with at most $|A|$ vertices. 
\end{remark}

\begin{remark}
There is a simple generalisation of the ideas of this section which gives coherent EPPA for (not only) $k$-uniform hypergraphs directly, producing EPPA-witnesses on fewer vertices than Corollary~\ref{cor:free} (see the next paragraph). For example, when $k=3$, vertices of $\str B$ are pairs $(x,\chi)$, where $\chi$ is a function from ${A\setminus\{x\} \choose 2}$ (the set of all (unordered) pairs of vertices of $\str A$ different from $x$) to $\{0,1\}$ and we say that $(x,\chi),(x',\chi'),(x'',\chi'')$ form a hyperedge of $\str B$ if and only if $x$, $x'$ and $x''$ are distinct and $\chi(\{x',x''\})+\chi'(\{x,x''\})+\chi''(\{x,x'\})$ is odd. The rest of the construction is generalised in the same way, see also Section~\ref{sec:relstructures}

Corollary~\ref{cor:free} also implies coherent EPPA for the class of all $k$-uniform hypergraphs and other such classes. However, its proof takes a detour by first constructing EPPA-witnesses where the $k$-ary relation is not symmetric and contains tuples with repeated occurrences of the same vertices (a generalisation of loops), and then relying on a construction of irreducible structure faithful EPPA-witnesses to get a $k$-uniform hypergraph.
\end{remark}

\begin{remark}
A minor change to the construction makes it possible to prove the extension property for partial switching automorphisms (which is a strengthening of standard EPPA), and hence also EPPA for two-graphs and antipodal metric spaces. This was done by Evans and the authors in~\cite{eppatwographs}.
\end{remark}

\begin{remark}
Hrushovski's construction gives EPPA-witnesses on at most $(2n2^n)!$ vertices (where $\vert A\vert=n$) and he asks if this can be improved~\cite[Section~3]{hrushovski1992}.\footnote{We would like to thank H.~Andr\'eka and I.~N\'emeti for bringing this question to our attention.} The combinatorial construction of Herwig and Lascar~\cite{herwig2000} provides EPPA-witnesses on roughly $(kn)^k$ vertices, where $\vert A\vert=n$ and $k$ is the maximum degree of a vertex of $\str A$. Our construction gives EPPA-witnesses on $n2^{n-1}$ vertices, thereby providing the best uniform bound over all graphs on $n$ vertices. The same construction was found independently by Andr\'eka and N\'emeti~\cite{Andreka2019}.

Hrushovski also proves a lower bound $\vert B\vert\geq 2^m+m$ for $\vert A\vert = 2m$. It remains open to improve either of the bounds. Some partial progress on obtaining EPPA-witnesses of small size for some special classes of graphs has been made by Bradley-Williams and Cameron~\cite{BradleyEPPA}. We believe that studying bounds on the number of vertices of EPPA-witnesses is an interesting and meaningful project which can deepen our understanding of symmetries of graphs.
\end{remark}

\medskip

What now follows is a series of strengthenings of the main ideas from this section.
Each of the constructions will proceed in several steps:
\begin{enumerate}
 \item Define a structure $\str{B}$ using a suitable variant of valuations.
 \item Give a construction of a generic copy $\str{A}'$ of $\str{A}$ in $\str B$.
 \item For a partial automorphism $\varphi$ of $\str A'$, give a construction of its coherent extension $\widetilde{\varphi}\colon \str B\to \str B$.
 \item Prove that $\widetilde{\varphi}$ is indeed a coherent extension of $\varphi$ and that $\str B$ and $\widetilde{\varphi}$ have the extra properties required in the respective section.
\end{enumerate}
We believe that the constructions are what is interesting. However, the proofs often contain some steps which are conceptually straightforward but slightly technical due to the nature of the constructions. We decided to state these technicalities as Claims and prove them at the very end of each section. We believe that this helps separate the key parts of the arguments from technical verifications.

\section{Coherent EPPA for relational structures}
\label{sec:relstructures}

In this section we generalise the construction from the previous section to prove the following proposition:
\begin{prop}
\label{prop:relstructures}
Let $L$ be a finite relational language equipped with a permutation group $\GammaL$ and let $\str A$ be a finite $\GammaL$-structure. There exists a finite $\GammaL$-structure $\str B$ which is a coherent EPPA-witness for $\str A$.
\end{prop}

Fix a finite relational language $L$ equipped with a permutation group $\GammaL$ and a finite $\GammaL$-structure $\str{A}$
with $A=\{1,\ldots, k\}$. We will construct a $\GammaL$-structure $\str B$ and give an embedding $\psi\colon\str A\to \str B$ such that $\str B$ is a coherent EPPA-witness for $\str A$ (with respect to $\psi$). Proposition~\ref{prop:relstructures} then immediately follows.

\addtocontents{toc}{\SkipTocEntry}
\subsection*{Witness construction}
Given a vertex $x\in A$ and an integer $n$, we denote by $U^A_n(x)$ the set of all $n$-tuples (i.e. $n$-element sequences) of elements of $A$ containing $x$. Note that $U^A_n(x)$ also includes $n$-tuples with repeated occurrences of vertices.

Given a relation $R\in L$ of arity $n$ and a vertex $x\in A$, we say that a function $\xi\colon U^A_n(x)\to \{0,1\}$ is an \emph{$R$-valuation function} for $x$. An \emph{$L$-valuation function} for a vertex $x\in A$ is a function $\chi$
assigning to every $R\in L$ an $R$-valuation function $\chi(R)$ for $x$.

\medskip

Now we are ready to give the definition of $\str{B}$:

\begin{enumerate}
  \item The vertices of $\str{B}$ are all pairs $(x,\chi)$, where $x\in A$ and
    $\chi$ is an $L$-valuation function for $x$.
  \item For every relation symbol $R$ of arity $n$ we put $$((x_1,\chi_1),\allowbreak \ldots, (x_n,\chi_n))\in \rel{B}{}$$ if and only if for every $1\leq i<j\leq n$ such that $x_i=x_j$ it also holds that $\chi_i = \chi_j$ and furthermore
$$\sum_{\chi\in \{\chi_i:1\leq i\leq n\}} \chi(R)((x_1,\ldots, x_n))\text{ is odd}$$ (summing over $\chi\in \{\chi_i:1\leq i\leq n\}$ ensures that possible multiple occurrences of $(x_i,\chi_i)$ are only counted once).
\end{enumerate}

Next we give an embedding $\psi\colon \str{A}\to \str{B}$ by putting $\psi_L$ to be the identity, and
$$\psi_A(x)=(x,\chi_x),$$ where $\chi_x$ is an $L$-valuation function for $x$ such that for every $R\in L$ we have
$$\chi_x(R)((y_1,\ldots, y_n))=\begin{cases}
1&\text{ if }(y_1,\ldots, y_n)\in \rel{A}{}\text{ and }x=y_1,\\
0&\text{ otherwise}.
\end{cases}$$
The following claim follows from the construction:
\begin{claim}\label{c:relstructures:emb}
$\psi$ is an embedding $\str A\to \str B$.
\end{claim}
\begin{proof}
Fix an $n$-ary relation $\rel{}{}\in L$. Recall that for $x\in A$, we have $\psi(x) = (x,\chi_x)$ with $\chi_x(R)(\bar{y})=1$ if and only if $\bar{y}_1 = x$ and $\bar{y}\in \rel{A}{}$. In particular, if $\bar{y}\notin\rel{A}{}$, then $\psi(\bar{y})\notin\rel{B}{}$, as for every $i$ we have that $\chi_{\bar{y}_i}(R)(\bar{y}) = 0$.

Suppose now that $\bar{y}\in\rel{A}{}$. For every $i$ we have $\chi_{\bar{y}_i}(R)(\bar{y}) = 1$ if and only if $\bar{y}_i = \bar{y}_1$. Hence
$$\sum_{\chi\in \{\chi_{\bar{y}_i} : 1\leq i\leq n\}} \chi(R)(\bar{y}) = 1,$$
so it is odd and thus $\psi(\bar{y})\in\rel{B}{}$.
\end{proof}

Put $\str{A}' = \psi(\str{A})$. This is the copy whose partial automorphisms we will later extend. Let $\pi\colon B\to A$ defined as $\pi((x,\chi))=x$ be the \emph{projection}.

\addtocontents{toc}{\SkipTocEntry}
\subsection*{Constructing the extension}
As in Section~\ref{sec:graphs}, we fix a partial automorphism $\varphi\colon \str{A}'\to\str{A}'$ and extend the projection of $\varphi$ to a permutation $\hat\varphi$ of $A$ in an order-preserving way. Note that $\varphi$ already contains a permutation of the language, therefore we will focus on extending the structural part.

For every relation symbol $R\in L$ of arity $n$, we construct a function $F_R\colon A^n\to \{0,1\}^n$.
These functions will play a similar role as the set $F$ in Section~\ref{sec:graphs} 
(i.e., they will control the \emph{flips}) and are constructed as follows:

For an $n$-tuple $\bar{x}=(x_1,\ldots, x_n)$ and $1\leq i\leq n$, we put $F_R(\bar{x})_i=1$ if and only if one of the following two cases is true:
\begin{enumerate}
\item\label{case:1} $x_i \in \pi(\dom(\varphi))$ and $\chi_{x_i}(R)(\bar{x})\neq \chi_{x_j}(\varphi(R))(\hat\varphi(\bar{x}))$, where $\varphi((x_i, \chi_{x_i})) = (x_j,\chi_{x_j})$.
\item\label{case:2} $\left\lvert\{x_j : x_j\in \pi(\dom(\varphi))\text{ and }F_R(\bar{x})_j=1\}\right\rvert$ is odd and $x_i = x_m$, where $1\leq m \leq n$ is the smallest index such that $x_m\notin \pi(\dom(\varphi))$ (note that $m$ might not exist, but then all entries are covered by case~\ref{case:1}).
\end{enumerate}
All the other entries of $F_R(\bar{x})$ are equal to 0. Note that case~\ref{case:2} ensures that the there is an even number of distinct vertices of $\bar{x}$ whose corresponding entry in $F_R(\bar{x})$ is equal to 1. 


For every $x\in A$ we define a function $f_x$ on $L$-valuation functions for $x$, putting
$$f_x(\chi)(\varphi(R))(\hat\varphi(\bar{y})) = \begin{cases}
 \chi(R)(\bar{y}) &  \text{if } F_R(\bar{y}) \text{ has 0 on entry corresponding to $x$} \\
  1-\chi(R)(\bar{y}) & \text{if } F_R(\bar{y}) \text{ has 1 on entry corresponding to $x$}.
\end{cases}$$

Finally, we define $\widetilde{\varphi}\colon \str B\to \str B$ by putting $\widetilde{\varphi}_L = \varphi_L$ and $\widetilde{\varphi}_B((x,\chi))=(\hat\varphi(x),f_x(\chi))$.

\addtocontents{toc}{\SkipTocEntry}
\subsection*{Proofs}
In the rest we proceed analogously to Section~\ref{sec:graphs}.

\begin{lemma}
$\widetilde{\varphi}$ is an automorphism of $\str B$ extending $\varphi$.
\end{lemma}
\begin{proof}
In the same way as in Lemma~\ref{lem:graphs:auto} one can see that $\widetilde{\varphi}$ is a bijection. Observe that by the construction we get that for every $\rel{}{}\in L$ of arity $n$ and every $n$-tuple $\bar{x}=(x_1,\ldots,x_n)\in A^n$ we have that $F_R(\bar{x})_i=F_R(\bar{x})_j$ whenever $x_i=x_j$ and that 
$$\left\lvert\{x_i : F_R(\bar{x})_i = 1\}\right\rvert\text{ is even}$$
(where taking the size of the set means that each distinct vertex is counted only once even if it has repeated occurrences in $\bar x$): Indeed, if $\bar{x}$ contains vertices from $A\setminus\pi(\dom(\varphi))$, this follows directly. Otherwise all vertices of $\bar{x}$ are from $\pi(\dom(\varphi))$, but then (as $\varphi$ is a partial automorphism), we get that $$((x_1,\chi_{x_1}),\ldots,(x_n,\chi_{x_n})) \in \rel{B}{}\iff (\varphi((x_1,\chi_{x_1})),\ldots,\varphi((x_n,\chi_{x_n})))\in\permrel{\varphi}{B}{},$$
and using the definition of relations in $\str B$ we see that an even number of distinct vertices from $\psi(\bar{x})$ changed how they valuate $\bar{x}$ with respect to $\rel{}{}$ and $\hat{\varphi}(\bar{x})$ with respect to $\permrel{\varphi}{}{}$ respectively.

Pick $\vec{x} = ((x_1,\chi_1),\ldots,(x_n,\chi_n))\in B^n$ and put $\bar{x}=(x_1,\ldots,x_n)$. Recall that $\vec{x}\in\rel{B}{}$ if and only if for every $1\leq i<j\leq n$ such that $x_i=x_j$ it also holds that $\chi_i = \chi_j$ and furthermore
$$\sum_{\chi\in \{\chi_i:1\leq i\leq n\}} \chi(R)(\bar{x})\text{ is odd}.$$
Note that if $x_i = x_j$ then $f_{x_i}(\chi_i) = f_{x_j}(\chi_j)$ if and only if $\chi_i = \chi_j$.
To get that $\widetilde{\varphi}$ is an automorphism, it remains to show that 
$$\sum_{\chi\in \{f_{x_i}(\chi_i):1\leq i\leq n\}} \chi(\varphi(R))(\hat\varphi(\bar{x}))\text{ is odd}$$
if and only if
$$\sum_{\chi\in \{\chi_i:1\leq i\leq n\}} \chi(R)(\bar{x})\text{ is odd}.$$

By the construction of $f_x$, we have for every $i$ that $$f_{x_i}(\chi_i)(\varphi(R))(\hat\varphi(\bar{x})) = \chi_i(R)(\bar{x})$$ if and only if $F_R(\bar{x})_i = 0$. This means that
$$\lvert\{x_i : f_{x_i}(\chi_i)(\varphi(R))(\hat\varphi(\bar{x})) \neq \chi_i(R)(\bar{x})\}\rvert$$
is equal to $\lvert\{x_i : F_R(\bar{x})_i = 1\}\rvert$, which is an even number. This concludes the proof that $\widetilde{\varphi}$ is an automorphism of $\str B$.

To see that $\widetilde{\varphi}$ extends $\varphi$, pick a tuple $\bar{x} = (x_1,\ldots,x_n)\in A^n$, an index $1\leq i\leq n$ such that $(x_i,\chi_{x_i})\in \dom(\varphi)$, and an arbitrary $\rel{}{}\in L$. Recall that $F_R(\bar{x})_i=1$ if and only if $\chi_{x_i}(R)(\bar{x})\neq \chi_{x_j}(\varphi(R))(\hat\varphi(\bar{x}))$, where $\varphi((x_i, \chi_{x_i})) = (x_j,\chi_{x_j})$. Since $f_{x_i}(\chi_{x_i})(\varphi(R))(\hat\varphi(\bar{x})) \neq \chi_{x_i}(R)(\bar{x})$ if and only if $F_R(\bar{x})_i=1$, we get that $\chi_{x_j} = f_{x_i}(\chi_{x_i})$ and hence $\varphi\subseteq \widetilde{\varphi}$.
\end{proof}

\begin{lemma}\label{lem:rel:coherence}
Let $\varphi_1$, $\varphi_2$ and $\varphi$ be partial automorphisms of $\str A$ such that $\varphi = \varphi_2\circ\varphi_1$ and let $\widetilde{\varphi_1}$, $\widetilde{\varphi_2}$ and $\widetilde{\varphi}$ be their corresponding extensions as above. Then $\widetilde{\varphi} = \widetilde{\varphi_2}\circ \widetilde{\varphi_1}$.
\end{lemma}
\begin{proof}
Let $\hat\varphi_1$, $\hat\varphi_2$ and $\hat\varphi$ be the permutations of $A$ constructed in the previous section for $\varphi_1$, $\varphi_2$ and $\varphi$ respectively, similarly define $F_R^1$, $F_R^2$ and $F_R$ for every $\rel{}{}\in L$. Since $\hat\varphi_1$, $\hat\varphi_2$ and $\hat\varphi$ were chosen in an order-preserving way, by Proposition~\ref{prop:setcoherence} we get that $\hat\varphi = \hat\varphi_2\circ\hat\varphi_1$. 

Hence, by an argument analogous to the proof of Lemma~\ref{lem:graphs:coherence}, we can see that it suffices to show that for every $\rel{}{}\in L$, for every $\bar{x}\in A^{\arity{}}$ and for every $1\leq i\leq \arity{}$, it holds that $F_R(\bar{x})_i=F_R^1(\bar{x})_i+F_{\varphi_1(R)}^2(\hat\varphi_1(\bar{x}))_i\mod 2$.

Fix such $\rel{}{}$, $\bar{x}$ and $i$. Put $y=\bar{x}_i$ and $(y,\chi) = \psi(y)$. First suppose that $(y,\chi)\in \dom(\varphi)=\dom(\varphi_1)$ and denote $(y',\chi')=\varphi_1((y,\chi))$ and $(y'',\chi'')=\varphi((y,\chi))=\varphi_2((y',\chi'))$. By the construction, we have the following:
\begin{align*}
F_R^1(\bar{x})_i = 1 &\iff \chi(\bar{x})\neq \chi'(\hat\varphi_1(\bar{x})),\\
F_{\varphi_1(R)}^2(\hat\varphi_1(\bar{x}))_i = 1 &\iff \chi'(\hat\varphi_1(\bar{x}))\neq \chi''(\hat\varphi(\bar{x})),\\
F_R(\bar{x})_i = 1 &\iff \chi(\bar{x})\neq \chi''(\hat\varphi(\bar{x})).
\end{align*}
It immediately follows that $F_R(\bar{x})_i = 1$ if and only if exactly one of $F_R^1(\bar{x})_i$ and $F_{\varphi_1(R)}^2(\hat\varphi_1(\bar{x}))_i$ is equal to one and we are done.

Otherwise $(y,\chi)\notin\dom(\varphi)$. Let $m$, $m_1$ and $m_2$ be the indices from case~\ref{case:2} of the definition of $F_R$ for $\bar{x}$, $F^1_R$ for $\bar{x}$, and $F^2_{\varphi_1(\rel{}{})}$ for $\hat\varphi_1(\bar{x})$, respectively. Since $(\varphi_1,\varphi_2,\varphi)$ is a coherent triple, we get that $m=m_1=m_2$.

If $y=\bar{x}_i \neq \bar{x}_m$ is not in $\pi(\dom(\varphi))$, it follows that $F_R^1(\bar{x})_i = F_{\varphi_1(R)}^2(\hat\varphi_1(\bar{x}))_i = F_R(\bar{x})_i=0$. Now we will assume that $\bar{x}_i=\bar{x}_m$. Define 
\begin{align*}
I &= \{1\leq j\leq n : \bar{x}_j\in\pi(\dom(\varphi))\text{ and }F_R(\bar{x})_j=1\},\\
I_1 &= \{1\leq j\leq n : \bar{x}_j\in\pi(\dom(\varphi_1))\text{ and }F_R^1(\bar{x})_j=1\},\text{ and}\\
I_2 &= \{1\leq j\leq n : \hat\varphi_1(\bar{x}))_j\in\pi(\dom(\varphi_2))\text{ and }F_{\varphi_1(R)}^2(\hat\varphi_1(\bar{x}))_j=1\}.
\end{align*}
Observe that by the previous paragraphs we have that $I$ is the symmetric difference of $I_1$ and $I_2$, so in particular
$$\vert I\vert = \vert I_1\vert + \vert I_2\vert - 2\vert I_1\cap I_2\vert.$$
Also note that $\bar{x}_j=\bar{x}_k$ if and only if $\hat\varphi_1(\bar{x})_j = \hat\varphi_1(\bar x)_k$. It follows that
$$\vert \{x_j : j\in I\}\vert = \vert \{x_j : j\in I_1\}\vert + \vert \{x_j : j\in I_2\}\vert - 2\vert \{x_j : j\in I_1\cap I_2\}\vert,$$
where $\vert \{x_j : j\in I\}\vert$ is the number of distinct vertices of $\str A$ such that their corresponding entry in $F_R(\bar{x})$ is equal to one. Looking at this equation modulo 2, we get that $\vert \{x_j : j\in I\}\vert$ is odd if and only if precisely one of $\vert \{x_j : j\in I_1\}\vert$ and $\vert \{x_j : j\in I_2\}$ is odd.

This implies (comparing with case~\ref{case:2} of the definitions of $F_R$, $F^1_R$ and $F^2_{\varphi_1(\rel{}{})}$) that even for $\bar{x}_i = \bar{x}_m$, we have that $F_R(\bar{x})_i=F_R^1(\bar{x})_i+F_{\varphi_1(R)}^2(\hat\varphi_1(\bar{x}))_i\mod 2$, which finishes the proof.
\end{proof}
This finishes the proof of Proposition~\ref{prop:relstructures}.

\begin{remark}\label{rem:rel_bound}
The EPPA-witness $\str B$ constructed in this section has at most
$$\mathcal O\left(\lvert A\rvert2^{\lvert L \rvert\lvert A\rvert^m}\right)$$
vertices, where $m$ is the largest arity of a relation in $L$. Consequently, the size of a coherent EPPA-witness for $\str A$ only depends on the language and on the number of vertices of $\str A$.
\end{remark}

\section{Infinite languages}\label{sec:infinite_languages}
When a non-trivial permutation group is present it is not true that for every finite structure there is a finite EPPA-witness. Consider, for example, the language $L$ consisting of infinitely many unary relations, where $\GammaL$ is the symmetric group. Let $\str A$ be a structure with a single vertex which is in exactly one relation. Then every EPPA-witness for $\str A$ needs to, in particular, extend all partial automorphisms of $\str A$ of type $(g,\emptyset)$, where $g\in \GammaL$ and $\emptyset$ is the empty map. This implies that every EPPA-witness for $\str A$ must contain a vertex in precisely one unary relation $U$ for every $U\in L$, hence infinitely many vertices.

First, we generalise this argument and prove Theorem~\ref{thm:negative}:
\begin{proof}[Proof of Theorem~\ref{thm:negative}]
$\str A$ being in an infinite orbit means that there is an infinite sequence $g_1, g_2,\ldots \in \GammaL$ such that the sequence $(g_1,\id_A)(\str A), (g_2,\id_A)(\str A), \ldots$ consists of pairwise distinct structures. For a contradiction, assume that there is a (finite) EPPA-witness $\str B$ for $\str A$.

In particular, $\str B$ needs to extend all partial automorphisms $(g_i,\emptyset)$, $i \geq 1$, which means that for every $i\geq 1$, there is an embedding of $(g_i,\id_A)(\str A)$ into $\str B$. In other words, for every $i\geq 1$ we get a tuple $\bar x_i \in B^{\lvert A\rvert}$, and by the assumption, all these tuples are pairwise distinct. This implies that the set $B^{\lvert A\rvert}$ is infinite, and since $\lvert A\rvert$ is finite, it follows that $B$ is infinite, a contradiction.
\end{proof}

\medskip

On the positive side, we prove the following proposition, thereby characterising relational languages $L$ equipped with a permutation group $\GammaL$ for which the class of all finite $\GammaL$-structures has EPPA.
\begin{prop}\label{prop:infinite_languages}
Let $L$ be a relational language equipped with a permutation group $\GammaL$ and let $\str A$ be a finite $\GammaL$-structure such that $\str A$ lies in a finite orbit of the action of $\GammaL$ by relabelling. There is a finite $\GammaL$-structure $\str B$ which is a coherent EPPA-witness for $\str A$.
\end{prop}

In order to prove Proposition~\ref{prop:infinite_languages}, we will need the following lemma.

\begin{lemma}\label{lem:redundant_groups}
Let $M$ be a relational language equipped with a permutation group $\GammaM$ and let $\mathcal C$ be a class of finite $\GammaM$-structures. Suppose that there is a set $N\subseteq M$ such that for every $\str A\in \mathcal C$ and every $\rel{}{}\in M\setminus N$, it holds that $\rel{A}{}=\emptyset$ and for every $\pi\in \GammaM$ we have $\pi(N)=N$ (i.e. $\pi$ fixes $N$ setwise). Put
$$\GammaN = \{\pi\restriction_N : \pi\in \GammaM\},$$
where for a permutation $\pi\colon M\to M$, $\pi\restriction_N$ is its restriction to $N$.

Then $\GammaN$ is a permutation group on $N$ and $\mathcal C$ has coherent EPPA if and only if $\mathcal D$ does, where $\mathcal D$ is the class consisting of the same structures as $\mathcal C$, but understood as $\GammaN$-structures.
\end{lemma}
\begin{proof}
Since every $\pi\in \GammaM$ fixes $N$ setwise, we immediately get that $\GammaN$ is a permutation group on $N$. If $\mathcal C$ has coherent EPPA, then clearly $\mathcal D$ does, too, because for each $\pi'\in \GammaN$, we can simply pick an arbitrary $\pi\in \GammaM$ such that $\pi' = \pi\restriction_N$ and use coherent EPPA for $\mathcal C$. It thus remains to prove the other direction.

In the following, for $\str A\in \mathcal C$, we denote by $\str A\!^N$ its corresponding $\GammaN$-structure from $\mathcal D$.

Fix $\str A\in \mathcal C$. By the assumption that $\mathcal D$ has coherent EPPA, we get $\str B\in \mathcal C$ such that $\str B\!^N$ is a coherent EPPA-witness for $\str A\!^N$. Let $f = (f_L, f_A)$ be a partial automorphism of $\str A$. Then $(f_L\restriction_N, f_A)$ is a partial automorphism of $\str A\!^N$ and it extends to an automorphism $(f_L\restriction_N, \theta)$ of $\str B\!^N$. It is straightforward to check that $(f_L, \theta)$ is an automorphism of $\str B$ extending $f$ (it clearly extends $f$, and it is an automorphism of $\str B$, because $\str B$ contains no relations from $M\setminus N$ and $f_L(N)=N$). Coherence follows by coherence in $\mathcal D$.
\end{proof}

\addtocontents{toc}{\SkipTocEntry}
\subsection{Proof of Proposition~\ref{prop:infinite_languages}}
First we will define some auxiliary notions. Given an $n$-tuple $\bar x=(x_1,\ldots,x_n)$ and a function $\omega\colon \{1,\ldots,m\}\to \{1,\ldots,n\}$, we define an $m$-tuple
$$\bar x \circ \omega = (x_{\omega(1)},\ldots,x_{\omega(m)}).$$

For a $\GammaL$-structure $\str B$ and an $n$-tuple $\bar x \in B$ containing no repeated vertices (i.e. if $\bar x_i = \bar x_j$, then $i=j$), we define
$\sigma(\bar x, \str B)$ to be the set of all pairs $(\rel{}{},\omega)$, where $\rel{}{}\in L$ is an $m$-ary relation and $\omega\colon \{1,\ldots,m\}\to \{1,\ldots,n\}$ is a surjective function, such that $\bar x \circ \omega \in \rel{B}{}$.

Next, we define sets $M_1, M_2, \ldots$, such that $M_n$ consists of all pairs $(\rel{}{},\omega)$, where $\rel{}{}\in L$ is an $m$-ary relation and $\omega$ is a surjection $\{1,\ldots,m\}\to \{1,\ldots,n\}$. For every $n$ and for every $X\subseteq M_n$ we assume that $L$ does not contain the symbol $\rel{}{X}$ and we let $\rel{}{X}$ be an $n$-ary relation. We define a language $M'$ to consist of all these relations $\rel{}{X}$.

We put $M = M' \cup L$ (remember that $M'\cap L = \emptyset$ by our assumption). For every $g\in \GammaL$, we define a permutation $\pi_g$ on $M$ such that
$$\pi_g(R) = \begin{cases}
g(R) &\text{ if }R\in L,\\
\rel{}{Y} &\text{ if }R = \rel{}{X}\in M',
\end{cases}$$
where $Y = \{(g(S),\omega) : (S,\omega)\in X\}$.

Finally, we put $\GammaM = \{\pi_g : g\in \GammaL\}$. It is easy to verify that the map $g\mapsto \pi_g$ is a group isomorphism $\GammaL\to\GammaM$ (this is the only place in the proof where we use that $L\subseteq M$).

\medskip

Given a $\GammaL$-structure $\str B$, we define a $\GammaM$-structure $T(\str B)$ such that the vertex set of $T(\str B)$ is $B$ and for every $\bar x\in B^n$ containing no repeated vertices, we put $\bar x\in\nbrel{T(\str B)}{\sigma(\bar x, \str B)}$. There are no other tuples in any relations of $T(\str B)$.

In the other direction, given a $\GammaM$-structure $\str B$ such that $\rel{B}{} = \emptyset$ for every $\rel{}{}\in L$, we define a $\GammaL$-structure $U(\str B)$ such that the vertex set of $U(\str B)$ is $B$, and whenever $\bar x\in \rel{B}{X}$, we put $\bar x\circ \omega \in \nbrel{U(\str B)}{S}$ for every $(S,\omega)\in X$. There are no other tuples in any relations of $U(\str B)$. It is easy to verify that $T$ and $U$ are mutually inverse, that is $UT(\str B) = \str B$ for every $\GammaL$-structure $\str B$, and $TU(\str B) = \str B$ for every $\GammaM$-structure $\str B$ such that $\rel{B}{} = \emptyset$ for every $\rel{}{}\in L$.

In fact, these maps are functorial in the sense of the following lemma.

\begin{lemma}\label{lem:functors}
Let $\str B,\str C$ be $\GammaL$-structures. Let $(g,f)$ be an embedding $\str B\to \str C$ ($g\in \GammaL$, $f\colon B\to C$). Then $(\pi_g,f)$ is an embedding $T(\str B)\to T(\str C)$.

Let $\str B,\str C$ be $\GammaM$-structures such that $\rel{B}{} = \rel{C}{} = \emptyset$ for every $\rel{}{}\in L$. Let $(\pi_g,f)$ be an embedding $\str B\to \str C$ ($\pi_g\in \GammaM$, $f\colon B\to C$). Then $(g,f)$ is an embedding $U(\str B)\to U(\str C)$.
\end{lemma}
\begin{proof}
We only need to verify the definition of an embedding. For the first part, we know that for every $S\in L$, every $n$-tuple $\bar x\in B$ containing no repeated vertices and for every surjection $\omega\colon \{1,\ldots,m\}\to \{1,\ldots,n\}$, we have $\bar x\circ \omega \in S_\str B$ if and only if $f(\bar x)\circ \omega \in g(S)_\str B$, which implies that
$$\sigma(f(\bar x), (g,f)(\str B)) = \left\{(g(S),\omega) : (S,\omega)\in \sigma(\bar x, \str B)\right\},$$
from which the claim follows. The second part can be proved in a complete analogy.
\end{proof}

\medskip

Continuing with the proof of Proposition~\ref{prop:infinite_languages} we define $N$ to be the subset of $M'$ consisting of all $\pi_g(\rel{}{\sigma(\bar x,\str A)})$, where $\pi_g\in \GammaM$ and $\bar x$ is a tuple of vertices of $\str A$ containing no repeated ones (remember that $\str A$ is the $\GammaL$-structure fixed in the statement of Proposition~\ref{prop:infinite_languages}).

We claim that $N$ is finite: Whenever $\rel{}{X}\in N$, then there is $\bar x\in A$ and $\pi_g\in \GammaM$ such that $\rel{}{X} = \pi_g(\rel{}{\sigma(\bar x,\str A)})$, however, this is equivalent to saying that $X = \sigma(\bar x,(g,\id_A)(\str A))$. In other words, every $n$-ary relation $\rel{}{X}\in N$ corresponds to at least one pair $(\bar x, (g,\id_A)(\str A))$, where $\bar x$ is an $n$-tuple of vertices of $\str A$ with no repeated occurrences and $g\in \GammaL$. Since $\str A$ lies in a finite orbit of the action of $\GammaL$ by relabelling, it follows that there are only finitely many different choices for $(g,\id_a)(\str A)$. By definition, all relations in $N$ have arity at most $\lvert A\rvert$, hence there are finitely many choices for $\bar x$ and thus $N$ is indeed finite. Observe also that for every $\pi_g\in \GammaM$ we have $\pi_g(N) = N$.

Let $\mathcal C$ be the class consisting of all $\GammaM$-structures $\str B$ such that whenever $\rel{}{}\in M\setminus N$, then $\rel{B}{}=\emptyset$. Observe that $T(\str A)\in \mathcal C$ and $U(\str B)$ is defined for every $\str B\in \mathcal C$.

\medskip

We have verified that $\mathcal C$ satisfies the conditions of Lemma~\ref{lem:redundant_groups}. Hence, we get a permutation group $\GammaN$ on $N$ and a class $\mathcal D$, which in this case is simply the class of all finite $\GammaN$-structures and hence has coherent EPPA by Proposition~\ref{prop:relstructures}. By Lemma~\ref{lem:redundant_groups} we then get that $\mathcal C$ also has coherent EPPA.

In particular, we get $\str C\in \mathcal C$ which is a coherent EPPA-witness for $T(\str A)$. Putting $\str B = U(\str C)$, we have a $\GammaL$-structure $\str B$ such that $T(\str B)$ is a coherent EPPA-witness for $T(\str A)$. In the last paragraph, we shall prove that $\str B$ is a coherent EPPA-witness for $\str A$.

\medskip

Let $(g,f)$ be a partial automorphism of $\str A$. From the construction it follows that $(\pi_g, f)$ is a partial automorphism of $T(\str A)$ (equivalently, it follows from Lemma~\ref{lem:functors} and the observation that a partial automorphism of $\str A$ can be understood as a pair of embeddings of the same structure into $\str A$), which extends to an automorphism $(\pi_g,\widetilde{f})$ of $T(\str B)$ (that is, $f\subseteq \widetilde{f}$). By Lemma~\ref{lem:functors} again, we get that $(g,\widetilde{f})$ is an automorphism of $\str B$, and since $f\subseteq \widetilde{f}$, it extends $(g,f)$. Coherence follows from coherence of $T(\str B)$, because $\pi_{gf} = \pi_g\pi_f$. This finishes the proof of Proposition~\ref{prop:infinite_languages}.

\section{EPPA for structures with unary functions}\label{sec:functions}
We are now ready to introduce unary functions into the language. In order to do so, we will use valuation structures instead of valuation functions, which was first done in~\cite{Evans3}. Otherwise we follow the general scheme as above and prove the following proposition.

\begin{prop}
\label{prop:eppafunctions}
Let $L$ be a language consisting of relation and unary function symbols equipped with a permutation group $\GammaL$ and let $\str A$ be a finite $\GammaL$-structure which lies in a finite orbit of the action of $\GammaL$ by relabelling. Then there is a coherent EPPA-witness for $\str A$.
\end{prop}

Fix a language $L$ consisting of relation and unary function symbols equipped with a permutation group $\GammaL$, and a finite $\GammaL$-structure  $\str A$ which lies in a finite orbit of the action of $\GammaL$ by relabelling.

Denote by $L_\mathcal R\subseteq L$ the language consisting of all relation
symbols of $L$ and let $\Gamma\!_{L_\mathcal R}$ be the group obtained by restricting permutations from $\GammaL$ to $L_\mathcal R$.
For a $\GammaL$-structure $\str{D}$, we will denote by $\str{D}^{-}$ the $\Gamma\!_{L_\mathcal R}$-reduct
of $\str{D}$ (that is, the $\Gamma\!_{L_\mathcal R}$-structure on the same vertex set as $\str{D}$ with
$\nbrel{\str{D}^-}{}=\rel{D}{}$ for every $R\in L_\mathcal R$)

\addtocontents{toc}{\SkipTocEntry}
\subsection*{Witness construction}
Let $\str{B}_0$ be a finite $\Gamma\!_{L_\mathcal R}$-structure which is a coherent EPPA-witness for $\str{A}^-$ ($\str B_0$ exists by
Proposition~\ref{prop:infinite_languages}). We furthermore, for convenience, assume that $\str A^-\subseteq \str B_0$. Let $x\in B_0$ be a vertex of $\str B_0$ and let $\str V$ be a $\GammaL$-structure. We say that $\str{V}$ is 
a \emph{valuation structure for $x$} if the following hold:
\begin{enumerate} 
\item $x\in V$,
\item there exists $y\in A$ and an isomorphism $\iota\colon\str{V}\to \cl_\str{A}(y)$ satisfying $\iota(x)=y$ (note that $\iota$ can permute the language), 
\item $\str{V}^{-}$ is a substructure of $\str{B}_0$.
\end{enumerate}
Note that if $L$ contains no functions, then there is exactly one valuation structure for every $x\in \str B_0$, namely the substructure of $\str B_0$ induced on $\{x\}$. In this case, the rest of this construction simply describes the identity.

We construct $\str{B}$ as follows:
\begin{enumerate}
  \item The vertices of $\str{B}$ are all pairs $(x,\str V)$ where $x\in B_0$ and $\str V$ is a valuation structure for $x$,
  \item for every relation symbol $\rel{}{}\in L$ of arity $n$, we put $((x_1, \str V_{1}), \ldots,\allowbreak (x_n, \str V_{n}))\in \rel{B}{}$ if and only if $(x_1,\ldots,x_n)\in \nbrel{\str{B}^-_0}{}$,
  \item for every (unary) function symbol $\func{}{}\in L$ we put 
    $$\func{B}{}((x,\str{V}))=\{(y,\cl_{\str{V}}(y)) : \allowbreak y\in \nbfunc{\str V}{}(x)\}.$$
    Since $\str V$ is a valuation structure for $x$, it follows that $\cl_{\str{V}}(y)$ is a valuation structure for $y$.
\end{enumerate}

Next we define an embedding $\psi\colon \str A\to \str B$, putting $\psi_A(x) = (x, \cl_\str A(x))$ and $\psi_L=\id_L$. Note that $\cl_\str A(x)$ (a substructure of $\str A$) is indeed a valuation structure for $\iota$ being the identity, because we assumed that $\str A^-\subseteq \str B_0$. We put $\str A'=\psi(\str A)$ to be the copy of $\str A$ in $\str B$ whose partial automorphisms we will extend.

\begin{claim}\label{c:func:emb}
$\psi$ is an embedding $\str A\to \str B$.
\end{claim}
\begin{proof}
It follows directly from the construction that $\psi$ is injective and that for every relation $\rel{}{}\in L$ we have $\bar{x}\in \rel{A}{}$ if and only if $\psi(\bar{x})\in\rel{B}{}$. It remains to verify that for every $x\in \str A$ and for every function $\func{}{}\in L$ we have $\psi(\func{A}{}(x)) = \func{B}{}(\psi(x))$.

By the construction of $\str B$ we have
$$\func{B}{}(\psi(x))=\func{B}{}((x,\cl_\str A(x)))=\{(y,\cl_{\cl_\str A(x)}(y)) : \allowbreak y\in \nbfunc{\cl_\str A(x)}{}(x)\}.$$
Since $\cl_{\cl_\str A(x)}(y) = \cl_{\str A}(y)$ and $\nbfunc{\cl_\str A(x)}{}(x) = \func{A}{}(x)$, we have
$$\func{B}{}((x,\cl_\str A(x)))=\{(y,\cl_\str A(y)) : \allowbreak y\in \func{A}{}(x)\},$$
which is exactly $\psi(\func{A}{}(x))$.
\end{proof}

Observe that the vertex set of $\str B$ is finite: Assume for a contradiction that it is infinite. Since there are finitely many vertices in $B_0$, this implies that there is $x\in B_0$ for which there are infinitely many valuation structures. Moreover, by the definition of a valuation structure, this implies that in fact there is a vertex $y\in A$, a structure $\str W\subseteq \str B_0$, an injection $\iota_W\colon W\to \cl_\str{A}(y)$, a sequence of permutations $g_1,g_2,\ldots$ and a sequence of structures $\str V_1,\str V_2,\ldots$, such that the following hold:
\begin{enumerate}
\item $\str V_i^- = \str W$ for every $i\geq 1$ (so, in particular, they have the same vertex set),
\item the structures $\str V_1,\str V_2,\ldots$ are pairwise distinct, and
\item $(g_i,\iota_W)$ is an embedding $\str V_i\to \cl_\str{A}(y)$ for every $i\geq 1$.
\end{enumerate}
Taking the inverse, we get that there is a substructure $(g_1,\iota_W)(\str V_1) = \str X\subseteq \cl_\str{A}(y)$ such that the structures $(g_i^{-1},\iota_W^{-1})(\str X)=\str V_i$, $i\geq 1$ are pairwise distinct, which gives a contradiction with $\str A$ lying in a finite orbit of the action of $\GammaL$ by relabelling.

\addtocontents{toc}{\SkipTocEntry}
\subsection*{Constructing the extension}
Let $\pi\colon B\to B_0$, defined by $\pi((x,\str V)) = x$, be the projection. Note that $\pi(\str A')=\str A^-$. Fix a partial automorphism $\varphi$ of $\str{A}'$. It induces (by $\pi$ and restriction to $\Gamma\!_{L_\mathcal R}$) a partial automorphism $\varphi_0$ of $\str{A}^-$. Denote by $\hat\varphi$ the
extension of $\varphi_0$ to an automorphism of $\str{B}_0$. Put $\widetilde{\varphi}_L = \varphi_L$ and
$$\widetilde{\varphi}_B((x, \str V)) = (\hat\varphi(x), (\varphi_L,\hat{\varphi})(\str V)).$$

\addtocontents{toc}{\SkipTocEntry}
\subsection*{Proofs}
We again proceed analogously to Section~\ref{sec:graphs}.
\begin{lemma}
$\widetilde{\varphi}$ is an automorphism of $\str B$ extending $\varphi$.
\end{lemma}
\begin{proof}
Since $\hat\varphi$ is a bijection $B_0\to B_0$, it follows that for every $x\in B_0$ the function $(\varphi_L,\hat\varphi)$ is a bijection of valuation structures for $x$. Hence $\widetilde{\varphi}_B$ is a bijection $B\to B$. The relations on $\str B$ only depend on the projection, and since $\hat\varphi$ is an automorphism, we get that $\widetilde{\varphi}$ respects the relations. It remains to prove that for every $\func{}{}\in \GammaL$ and every $(x,\str V)\in B$ we have that $\widetilde{\varphi}(\func{B}{}((x,\str V))) = \permfunc{\widetilde{\varphi}}{B}{}(\widetilde{\varphi}((x,\str V)))$. To make the notation more readable, define function $h$ mapping every valuation structure $\str V$ to $(\varphi_L,\hat\varphi)(\str{V})$.

By the definition of $\str B$ we know that
$$\func{B}{}((x,\str V)) = \{(y,\cl_{\str{V}}(y)) : \allowbreak y\in \nbfunc{\str V}{}(x)\},$$
hence
$$\widetilde{\varphi}(\func{B}{}((x,\str V))) = \{(\hat\varphi(y),h(\cl_{\str{V}}(y))) : \allowbreak y\in \nbfunc{\str V}{}(x)\}.$$
Denote $X = \permfunc{\widetilde{\varphi}}{B}{}(\widetilde{\varphi}((x,\str V)))$. By the definition of $\str B$, we have
$$X = \{(y,\cl_{h(\str V)}(y)) : \allowbreak y\in \permnbfunc{\widetilde{\varphi}}{h(\str V)}{}(\hat\varphi(x))\}.$$
Since $\hat\varphi$ is a bijection $B\to B$, we can write
$$X = \{(\hat\varphi(y),\cl_{h(\str V)}(\hat\varphi(y))) : \allowbreak \hat\varphi(y)\in \permnbfunc{\widetilde{\varphi}}{h(\str V)}{}(\hat\varphi(x))\}.$$
Note that $\permnbfunc{\widetilde{\varphi}}{h(\str V)}{}(\hat\varphi(x)) = \hat\varphi(\nbfunc{\str V}{}(x))$, hence 
$\hat\varphi(y)\in \permnbfunc{\widetilde{\varphi}}{h(\str V)}{}(\hat\varphi(x))$ if and only if $y\in \nbfunc{\str V}{}(x)$, and so we can write
$$X = \{(\hat\varphi(y),\cl_{h(\str V)}(\hat\varphi(y))) : \allowbreak y\in \nbfunc{\str V}{}(x)\}.$$
Finally, $\cl_{h(\str V)}(\hat\varphi(y)) = h(\cl_{\str{V}}(y))$, hence indeed $$X = \widetilde{\varphi}(\func{B}{}((x,\str V))).$$
\end{proof}

\begin{lemma}
Assume that $\str{B}_0$ is a coherent EPPA-witness for $\str{A}^-$ and thus $\hat\varphi$ can be chosen to be coherent. Let $\varphi_1$, $\varphi_2$ and $\varphi$ be partial automorphisms of $\str A$ such that $\varphi = \varphi_2\circ\varphi_1$, and let $\widetilde{\varphi_1}$, $\widetilde{\varphi_2}$ and $\widetilde{\varphi}$ be their corresponding extensions as above. Then $\widetilde{\varphi} = \widetilde{\varphi_2}\circ \widetilde{\varphi_1}$.
\end{lemma}
\begin{proof}
Pick an arbitrary $(x,\str V)\in B$. We know that
$$\widetilde{\varphi}((x,\str V)) = (\hat\varphi(x), (\varphi_L,\hat\varphi)(\str V)).$$
Similarly,
$$\widetilde{\varphi_2}\circ \widetilde{\varphi_1}((x,\str V)) = (\hat\varphi_2(\hat\varphi_1(x)), (\varphi_L,\id_V)((\id_L,\hat\varphi_2)((\id_L,\hat\varphi_1)(\str V))))$$
(as by the assumption, the language parts of $\varphi_1$ and $\varphi_2$ compose to $\varphi_L$ and moreover the language permutation commutes with applying $\hat\varphi_i$ on the vertex set of $\str V$). By coherence of $\str B_0$ we know that $\hat\varphi_2(\hat\varphi_1(x)) = \hat\varphi(x)$, and so $$(\varphi_L,\id_V)((\id_L,\hat\varphi_2)((\id_L,\hat\varphi_1)(\str V)) = (\varphi_L,\hat\varphi)(\str V),$$ hence indeed $\widetilde{\varphi} = \widetilde{\varphi_2}\circ \widetilde{\varphi_1}$.
\end{proof}
This proves Proposition~\ref{prop:eppafunctions}.

\begin{observation}\label{obs:functions_bound}
The number of vertices of the EPPA-witness $\str B$ constructed in this section can be bounded from above by a function which depends only on the number of vertices of $\str B_0$, the number of vertices of $\str A$ and the size of the orbit of the action of $\GammaL$ by relabelling in which $\str A$ lies.
\end{observation}
\begin{proof}
We know that the vertex set of $\str B$ consists of pairs $(x,\str V)$, where $x\in B_0$ and $\str V$ is a valuation structure for $x$. Thus, it is enough to bound the number of valuation structures for any vertex of $\str B_0$. The vertex set of any valuation structure is a subset of $B_0$, hence there are at most $2^{\lvert B_0\rvert}$ different vertex sets of valuation structures. Hence it remains to bound the number of different valuation structures on a given subset $V\subseteq B_0$.

Let $\str A_1,\ldots,\str A_o$ be an enumeration of the orbit of the action of $\GammaL$ by relabeling in which $\str A$ lies and let $g_1,\ldots,g_o\in \GammaL$ such that $\str A_i = (g_i,\id_A)(\str A)$. Note that for every $\GammaL$-structure $\str U$ and every embedding $(f_L,f_U)\colon \str U\to \str A$, there is $i$ such that $(f_L,f_U)(\str U) = (g_i^{-1},f_U)(\str U)$. Indeed, by the way embeddings compose, we have that $(f_L,f_U) = (f_L,\id_A)\circ (\id_L,f_U)$. Composing with $(f_L^{-1},\id_A)$ on the left, we get that $(\id_L,f_U)$ is an embedding $\str U\to (f_L^{-1},\id_A)(\str A)$. This means that there is $i$ such that $(f_L^{-1},\id_A)(\str A) = \str A_i = (g_i,\id_A)(\str A)$ and hence indeed $(f_L,f_U)(\str U) = (g_i^{-1},f_U)(\str U)$.

Fix $V\subseteq B_0$. For every valuation structure $\str V$ on this vertex set, there is an isomorphism $\iota=(\iota_L,\iota_V)\colon V\to \cl_\str A(y)$, where $y\in A$. Since $\iota$ is in particular an embedding $\str V\to \str A$, by the previous paragraph we get $1\leq i\leq o$ such that $(\iota_L,\iota_V)(\str V) = (g_i^{-1},\iota_V)(\str V)$. Hence there are at most as many different structures $\str V$ as there are pairs $(g_i^{-1},\iota_V)$. Since there are at most $o$ choices for $g_i^{-1}$ and at most $\lvert A\rvert^{\lvert V\rvert}\leq \lvert A\rvert^{\lvert A\rvert}$ choices for $\iota_V$, the claim is proved.
\end{proof}

\medskip

From now on, our structures may contain unary functions. To some extent, the unary functions do not interfere too much with the properties which we are going to ensure and thus it is possible to treat them ``separately''. Namely, we will always first introduce a notion of a valuation function (in order to get the desired property) and then wrap the valuation functions in a variant of the valuation structures.

\section{Irreducible structure faithful EPPA}
\label{sec:faithful}
In this section we prove the following proposition, which is a strengthening of~\cite[Theorem
1.7]{Evans3}, which in turn extends~\cite[Theorem 9]{hodkinson2003}.
\begin{prop}
\label{prop:faithful}
Let $L$ be a language consisting of relation and unary function symbols equipped with a permutation group $\GammaL$ and let $\str A$ be a finite $\GammaL$-structure. Let $\str B_0$ be a finite $\GammaL$-structure which is an EPPA-witness for $\str A$. Then there is a finite $\GammaL$-structure $\str B$ which is an irreducible structure faithful EPPA-witness for $\str A$, and a homomorphism-embedding $\str B\to \str B_0$.

Moreover, if $\str{B}_0$ is coherent then $\str{B}$ is coherent, too.
\end{prop}

\begin{remark}
Note that up to this point, the permutation group $\GammaL$ was not very relevant. In Section~\ref{sec:relstructures}, the constructed EPPA-witnesses worked for $\GammaL$ being the symmetric group, and in Section~\ref{sec:functions}, it didn't play an important role either.

However, in this section $\GammaL$ plays a central role because it restricts which irreducible substructures can be sent to $\str A$ by an automorphism. For example, let $L$ be the language consisting of $n$ unary relation $\rel{}{1},\ldots,\rel{}{n}$ and let $\str A$ be an $L$-structure consisting of one vertex which is in $\rel{A}{1}$ and in no other relations.

For this $\str A$, Section~\ref{sec:relstructures} will produce an EPPA-witness $\str B_0$, where $B_0 = \{v_S : S\subseteq \{1,\ldots,n\}\}$ such that $v_S\in \nbrel{\str B_0}{i}$ if and only if $i\in S$. Fix now a permutation group $\GammaL$ on the language $L$ and consider $\str A$ as a $\GammaL$-structure. An irreducible structure faithful EPPA-witness $\str B$ for $\str A$ will contain only vertices, which are in precisely one unary relation $\rel{B}{i}$ such that moreover $\rel{}{1}$ and $\rel{}{i}$ are in the same orbit of $\GammaL$. In particular, if $\GammaL = \{\id_L\}$, then $\str A$ is an irreducible structure faithful EPPA-witness for itself. If $\GammaL = \Sym(L)$, then a possible irreducible structure faithful EPPA-witness for $\str A$ has $n$ vertices $v_1,\ldots,v_n$ such that $v_i$ has precisely one unary mark $\rel{}{i}$.
\end{remark}

Fix a language $L$ consisting of relation symbols and unary function symbols equipped with a permutation group $\GammaL$. Fix also a finite $\GammaL$-structure $\str{A}$
and its EPPA-witness $\str{B}_0$. Without loss of generality, assume that $\str A\subseteq \str B_0$.
We now present a construction of an irreducible structure faithful EPPA-witness $\str{B}$ with
a homomorphism-embedding (projection) to $\str{B}_0$, such that every
extension of a partial automorphism in $\str{B}$ is induced by the extension of its projection to $\str{B}_0$.

\addtocontents{toc}{\SkipTocEntry}
\subsection*{Witness construction}
Let $\str I$ be an irreducible substructure of $\str{B}_0$. We say that $\str I$ is \emph{bad} if there is no
automorphism $f\colon \str{B}_0\to\str{B}_0$ such that $f(I)\subseteq A$. 
 Given a vertex $x\in B_0$,
we denote by $U(x)$ the set of all bad irreducible substructures of $\str B_0$
containing $x$.

For a vertex $x\in B_0$, we say that a function assigning to every $\str{I}\in U(x)$ a value from $\{1,\ldots, |I|-1\}$ is a \emph{valuation function for $x$}.  
Given vertices $x,y\in B_0$ and their valuation functions $\chi$ and $\chi'$ respectively, we say
that the pairs $(x,\chi)$ and $(y,\chi')$ are {\em generic}, if either
$(x,\chi)=(y,\chi')$, or $x\neq y$ and for every $\str{I}\in U(x)\cap U(y)$ it holds that $\chi(\str{I})\neq
\chi'(\str{I})$. We say that a set $S$ is \emph{generic} if it consists of pairs $(x,\chi)$ where $x\in B_0$ and $\chi$ is a valuation function for $x$, and every pair $(x,\chi),(y,\chi')\in S$ is generic. In particular, the projection to the first coordinate is injective on every generic set.

A \emph{valuation structure} for a vertex $x\in B_0$ is a $\GammaL$-structure $\str{V}$
such that:
\begin{enumerate}
\item The vertex set of $\str V$ is a generic set of pairs $(y,\chi)$ with $y\in \cl_{\str{B}_0}(x)$ and $\chi$ being a valuation function for $y$, and
\item the pair $\iota=(\id_L, \iota_V)$, where $\iota_V((y,\chi))=y$, is an isomorphism of $\str{V}$ and $\cl_{\str{B}_0}(x)$.
\end{enumerate}
For a pair $(x, \str V)$, where $x\in B_0$ and $\str V$ is a valuation structure for $x$, 
we denote by $\chi(x, \str{V})$ the (unique) valuation function for $x$ such that $(x,\chi(x, \str{V}))\in V$ and
we put $\pi(x, \str V)=x$ ($\pi$ is again the \emph{projection} from $\str{B}$ to $\str{B}_0$). We say that a set $S$ of pairs $(x, \str V)$, such that $x\in B_0$ and $\str V$ is a valuation structure for $x$, is \emph{generic}, if the union $\bigcup_{(x,\str V)\in S} V$ is generic. Note that this implies that in particular $\{(x, \chi(x, \str V)) : (x,\str V)\in S\}$ is generic and thus $\pi$ is injective on every generic set. 

Observe that if $L$ contains no functions then every valuation structure $\str V$ for $x\in B_0$ contains exactly one vertex $(x,\chi(x,\str V))$, and conversely, for every valuation function $\chi$ for $x$ there is exactly one valuation structure $\str V$ for $x$ such that $\chi(x,\str V) = \chi$.

\medskip

Now we construct a $\GammaL$-structure $\str{B}$: 
\begin{enumerate}
\item The vertices of $\str{B}$ are all pairs $(x,\str{V})$, where $x\in B_0$ and $\str V$ is a valuation structure for $x$.

\item For every relation symbol $\rel{}{}\in L_\mathcal R$, we put $$((x_1,\str{V}_{1}),\ldots,\allowbreak (x_{\arity{}{}},\str{V}_{\arity{}{}}))\in \rel{B}{},$$ if and only if $(x_1,\ldots, x_{\arity{}{}})\in \nbrel{\str{B}_0}{}$, and $\{(x_1,\str{V}_{1}),\ldots,\allowbreak (x_{\arity{}{}},\str{V}_{\arity{}{}})\}$ is generic.

\item For every (unary) function symbol $\func{}{}\in L_\mathcal F$, we put
$$\func{B}{}((x,\str{V}))=\left\{(y,\cl_{\str{V}}((y,\chi))) : (y,\chi)\in \nbfunc{\str{V}}{}((x,\chi(x,\str V)))\right\}.$$
\end{enumerate}

Note that in the definition of $\func{B}{}((x,\str{V}))$ it holds that $\cl_{\str{V}}((y,\chi))$ is isomorphic to $\cl_{\str B_0}(y)$, so it is indeed a valuation structure for $y$. Also observe that $\str B$ is finite, because $B_0$ is finite, $U(x)$ is finite for every $x\in B_0$ (hence there are only finitely many candidate vertex sets for valuations structures), and there is at most one valuation structure on any candidate vertex set.

The following claim (whose proof is quite technical and will be given at the end of this section) justifies our definition of genericity and the construction of $\str B$.
\begin{claim}\label{c:faithful:irreducible}
Let $\str D\subseteq \str B$ be irreducible. Then $D$ is generic.
\end{claim}

We also have this complementary fact to Claim~\ref{c:faithful:irreducible}, which will be useful several times in this section.
\begin{claim}\label{c:faithful:generic}
Let $\str D\subseteq \str B$ be such that $D$ is a generic set. Then the restriction of $(\id_L,\pi)$ to $\str D$ is an embedding $\str D\to\str B_0$.
\end{claim}

Next we define an embedding $\psi\colon\str{A}\to\str{B}$ with $\psi_L=\id$. For every bad irreducible $\str{I}\subseteq \str B_0$, we fix an arbitrary injective function $u_\str{I}\colon I\cap A\to \{1,2,\ldots, \vert I\vert - 1\}$. Such a function exists, because $A\cap I$ is a proper subset of $I$ (otherwise $\str I$ would not be bad). For every $x\in A$ we define a valuation function $\chi_x$ for $x$ such that $\chi_x(\str I) = u_\str I(x)$.

Given $x\in A$, we also define a valuation structure $\str V_x$ for $x$ such that $V_x = \{(y, \chi_y):y\in \cl_\str{A}(x)\}$, and the structure on $V_x$ is chosen such that the pair $(\id_L,(y,\chi_y)\mapsto y)$ is an isomorphism $\str V_x\to\cl_\str{A}(x)$. We put $\psi(x) = (x, \str V_x)$.

\begin{claim}\label{c:faithful:emb}
$\psi$ is an embedding $\str A \to \str B$ and $\str A'=\psi(\str A)$ is generic.
\end{claim}

\addtocontents{toc}{\SkipTocEntry}
\subsection*{Constructing the extension}
At some point, we will also need to prove irreducible structure faithfulness. And in that proof, we are going to need to construct some automorphisms of $\str B$ based on some automorphisms of $\str B_0$ and partial automorphisms of $\str B$. Because of it, we will prove a more general statement

\begin{lemma}\label{lem:faithful-extension}
Let $\varphi$ be a partial automorphism of $\str B$ satisfying the following conditions:
\begin{enumerate}
\item Both the domain and the range of $\varphi$ are generic, and
\item\label{cond2} there is an automorphism $\hat\varphi$ of $\str B_0$ which extends the projection of $\varphi$ via $\pi$.
\end{enumerate}
Then there is an automorphism $\widetilde\varphi$ of $\str B$ extending $\varphi$.
\end{lemma}

Note that if $\varphi$ is a partial automorphism of $\str A'$, then it satisfies both conditions (as $\str A'$ is generic) and therefore it can be extended to an automorphism of $\str B$. Note also that in condition~\ref{cond2}, the projection of $\varphi$ via $\pi$ is a partial automorphism of $\str B_0$ by Claim~\ref{c:faithful:generic}.



\begin{proof}[Proof of Lemma~\ref{lem:faithful-extension}]
Put $$D=\{(x,\chi(x,\str V)):(x,\str V)\in \dom(\varphi)\}$$ and $$R=\{(x,\chi(x,\str V)):(x,\str V)\in \range(\varphi)\}.$$ Because both $\dom(\varphi)$ and $\range(\varphi)$ are generic, we get that $|D|=|R|=|\pi(D)|=|\pi(R)|$, so in particular no $x\in B_0$ appears in $D$ or $R$ with more than one valuation structure. Therefore, $\varphi$ defines a bijection $q\colon D\to R$.

For a bad irreducible substructure $\str I\subseteq \str B_0$, we can define a partial permutation $\tau_{\str I}^{\varphi}$ of $\{1,\ldots, |I|-1\}$, such that for every $(y,\chi)\in D$ with $y\in I$ and for $q((y, \chi))=(\hat\varphi(y), \chi')$, we put
$$\tau_{\str I}^{\varphi}(\chi(\str I)) = \chi'(\hat\varphi(\str I)).$$
This is indeed a partial permutation of $\{1,\ldots, |I|-1\}$, because both $D$ and $R$ are generic. Let $\hat\tau_{\str I}^\varphi$ be the order-preserving extension of $\tau_{\str I}^{\varphi}$.

Put $$\mathcal V = \bigcup_{(x,\str V)\in \str B} V.$$
Having $\hat\tau_{\str I}^\varphi$ for every bad $\str I$, we can define $\hat q\colon \mathcal V \to \mathcal V$ as
$$\hat q((x, \chi)) = (\hat\varphi(x), \chi'),$$
where $\chi'(\hat\varphi(\str I)) = \hat\tau_{\str I}^\varphi(\chi(\str I))$. Since $\hat\varphi$ is an automorphism of $\str B_0$ and each $\hat\tau_{\str I}^\varphi$ is a permutation of $\{1,\ldots, |I|-1\}$, it follows that $\hat q$ is a permutation of $\mathcal V$. It is easy to check that $\hat q$ extends $q$.

Finally, we define $\widetilde\varphi\colon B\to B$ by putting $\widetilde\varphi_L = \varphi_L$ and
$$\widetilde{\varphi}((x, \str V)) = (\hat\varphi(x), (\varphi_L,\hat{q})(\str V)).$$

The proof of the following claim, which will be given at the end of this section, is simply a mechanical verification that our constructions are well-defined.
\begin{claim}\label{c:faithful:auto}
$\widetilde\varphi$ is an automorphism of $\str B$ extending $\varphi$.
\end{claim}
\end{proof}

\addtocontents{toc}{\SkipTocEntry}
\subsection*{Proofs}
We again proceed analogously to Section~\ref{sec:graphs}.
\begin{lemma}
$\str B$ is an EPPA-witness for $\str A$. Moreover, if $\str B_0$ is coherent then so is $\str B$.
If $\str B_0$ is a coherent EPPA-witness for $\str A$, then $\str B$ is a coherent EPPA-witness for $\str A'$.
\end{lemma}
\begin{proof}
Lemma~\ref{lem:faithful-extension} implies that $\str B$ is indeed an EPPA-witness for $\str A$, because $A'$ is generic. We thus focus on proving coherence. Let $\varphi_1$, $\varphi_2$ and $\varphi$ be a coherent triple of partial automorphisms of $\str A'$ and let $\hat\varphi_1$, $\hat\varphi_2$, $\hat\varphi$ be automorphisms of $\str B_0$ which are the coherent extensions of the projections of $\varphi_1$, $\varphi_2$ and $\varphi$ by $\pi$.

Denote by $\hat q_1$, $\hat q_2$, $\hat q$ and  $\tau_{\str I}^\varphi$, $\tau_{\str I}^{\varphi_1}$ and $\tau_{\str I}^{\varphi_2}$ for every $\str I$ the corresponding functions from proof of Lemma~\ref{lem:faithful-extension}. Coherence on the first coordinate follows from coherence of $\hat\varphi_1$, $\hat\varphi_2$, $\hat\varphi$, to get coherence on the second coordinate, we need to prove that $\hat q = \hat q_2 \circ \hat q_1$. To see that, one only needs to prove that $\tau_{\str I}^\varphi = \tau_{\str I}^{\varphi_2}\circ \tau_{\str I}^{\varphi_1}$, which is true as all of them are extended in an order-preserving way.
\end{proof}

\begin{proof}[Proof of Proposition~\ref{prop:faithful}]
First we prove that $(\id_L,\pi)$ is a homomorphism-embedding from $\str B$ to $\str B_0$. From the construction it directly follows that it is a homomorphism which preserves functions. Let $\str I$ be an irreducible substructure of $\str B$. By Claim~\ref{c:faithful:irreducible} we get that $\str I$ is generic and hence by Claim~\ref{c:faithful:generic} we get that $(\id_L,\pi)$ is an embedding on $\str I$.

It only remains to prove irreducible structure faithfulness of $\str B$. Let $\str D$ be an irreducible substructure of $\str B$. By Claim~\ref{c:faithful:irreducible} we get that $\str D$ is generic and hence $\pi(\str{D})$ is not a bad substructure of $\str{B}_0$. It is, however, irreducible, because $\pi$ is a homomorphism-embedding, and thus
there is $\hat\varphi\in \Aut(\str{B}_0)$ such that $\hat\varphi(\pi(\str{D}))\subseteq A$.
Define $\varphi\colon \str{D} \to \str{A}'$ by $\varphi_D((x,\str{V})) =  \psi(\hat\varphi(x))$ and $\varphi_L = \hat\varphi_L$. This is a partial automorphism of $\str{B}$ with generic domain and range, whose projection extends to $\hat\varphi$. Lemma~\ref{lem:faithful-extension} then gives an automorphism of $\str B$ sending $D$ to $A'$, which is what we wanted.
\end{proof}

We are now ready to prove Theorem~\ref{thm:nreppa}.
\begin{proof}[Proof of Theorem~\ref{thm:nreppa}]
Section~\ref{sec:functions} gives a finite coherent EPPA-witness $\str B_0$ for $\str A$, Proposition~\ref{prop:faithful} ensures irreducible structure faithfulness and preserves coherence. The ``consequently'' part is immediate.
\end{proof}

\begin{remark}\label{rem:irreducible_symmetric}
Note that Theorem~\ref{thm:nreppa} can be used to prove EPPA for classes where the relations are, for example, symmetric (because a non-symmetric relation is witnessed on a tuple which is irreducible), and similarly it implies EPPA for classes with unary functions whose range has a given size (for example, size 1, which means that we can prove EPPA for the standard model-theoretic unary functions).
\end{remark}

\begin{remark}
Note that our partial automorphism extension actually has some functorial properties. Taking isomorphic copies, we can assume that $\str A\subseteq \str B_0$ and $\str A\subseteq \str B$ (then in particular $\pi\restriction_A = \id_A$). Given a partial automorphism $\varphi$ of $\str A$ let $\hat{\varphi}$ be its (coherent) extension to an automorphism of $\str B_0$ and let $\widetilde{\varphi}$ be the constructed extension to an automorphism of $\str B$ using $\hat{\varphi}$. Let $\sim$ be an equivalence relation on $B$ given by $x\sim y \iff \pi(x) = \pi(y)$. Then $\sim$ is a congruence with respect to $\widetilde{\varphi}$ and the natural actions of $\hat{\varphi}$ and $\widetilde{\varphi}$ on the equivalence classes of $\sim$ coincide.
\end{remark}

\begin{observation}\label{obs:faithful_bound}
The number of vertices of the EPPA-witness $\str B$ constructed in this section can be bounded from above by a function which depends only on the number of vertices of $\str B_0$ and the number of vertices of $\str A$.
\end{observation}
\begin{proof}
Put $m=\lvert B_0\rvert$ and $n=\lvert A\rvert$. There are at most $2^m$ bad substructures of $\str B_0$ and hence at most $(m-1)^{2^m}$ valuation functions for a given $x\in B_0$ (this is a very rough estimate). Let $P$ be the set of all pairs $(x,\chi)$, where $x\in B_0$ and $\chi$ is a valuation function for $x$. We get that $\lvert P\rvert \leq m(m-1)^{2^m}$.

The vertices of $\str B$ are pairs $(x,\str V)$, where $x\in B_0$ and $\str V$ is a valuation structure for $x$. The vertex set of every valuation structure is a subset of $P$ (and hence there are at most $2^{\lvert P\rvert}$ of them) and the structure on $\str V$ is determined by an embedding $(\id_L,\iota_V)\colon \str V\to \str B_0$. There are at most $m^{\lvert V\rvert} \leq m^m$ such embeddings. This finishes the proof.
\end{proof}

\begin{proof}[Proof of Claim~\ref{c:faithful:irreducible}]
For a contradiction, suppose that it is not the case, that is, there are $(u,\chi),(u',\chi')\in \bigcup_{(x,\str V)\in D} V$ which form a non-generic pair. This implies that there are
$(x,\str{X}),(y,\str{Y})\in D$ such that $(u,\chi)\in X$ and $(u',\chi')\in Y$ (they cannot both lie in the same valuation structure, because the vertex sets of valuation structures are generic), and hence the set $\{(x,\str{X}),(y,\str{Y})\}$ is non-generic. Put $\str{E}_x = \{(a,\str{U}) \in D : (x,\str{X}) \not\in \cl_{\str{D}}((a,\str{U}))\}$ and similarly define $\str E_y$. Since closures are unary, these are substructures of $\str{D}$. Note that as closures in $\str B$ are generic, we also know that $(y, \str Y)\in \str E_x$ and $(x,  \str X)\in \str E_y$. This means that $\str E_x, \str E_y$ are both non-empty and neither is a substructure of the other.

We first prove $\str{E}_x\cup \str{E}_y = \str{D}$. Suppose for a contradiction that there is $(z,\str{Z}) \in \str{D}$ with $\{(x,\str{X}), (y,\str{Y})\} \subseteq \cl_{\str{D}}((z,\str{Z}))$. Then (by the construction of $\str B$) we have $\str{X}, \str{Y} \subseteq \str{Z}$, which is a contradiction with $(x,\str{X}),(y,\str{Y})$ forming a non-generic pair.

Fix $(a,\str{U}) \in \str{E}_x\setminus \str{E}_y$ and $(b,\str{W}) \in \str{E}_y\setminus \str{E}_x$. Because we know that $\str{Y} \subseteq \str{U}$ and $\str{X} \subseteq \str{W}$, we get that $(a,\str{U})$, $(b,\str{W})$ is not a generic pair and therefore no relation of $\str D$ contains both $(a,\str{U})$ and $(b,\str{W})$. Thus $\str{D}$ is a free amalgam of $\str{E}_x$ and $\str{E}_y$ over their intersection, which is a contradiction with its irreducibility. Therefore $D$ is indeed generic.
\end{proof}

\begin{proof}[Proof of Claim~\ref{c:faithful:generic}]
We have already observed that $\pi$ is injective. The fact that $\pi$ preserves relations and non-relations follows directly from the construction of $\str B$. Let $\func{}{}\in L$ be a function and fix $(x,\str V)\in D$. We need to prove that $\nbfunc{\str B_0}{}(x)$ is equal to $\pi\left(\func{B}{}((x,\str{V}))\right)$.

By definition,
$$\pi\left(\func{B}{}((x,\str{V}))=\left\{y : (y,\chi)\in \nbfunc{\str{V}}{}((x,\chi(x,\str V)))\right\}\right).$$
Moreover, we know that the pair $\iota=(\id_L, \iota_V)$, where $\iota_V((y,\chi))=y$, is an isomorphism of $\str{V}$ and $\cl_{\str{B}_0}(x)$. Hence if $(y,\chi)\in \nbfunc{\str{V}}{}((x,\chi(x,\str V)))$ then $y\in\nbfunc{\str B_0}{}(x)$ and conversely, whenever $y\in\nbfunc{\str B_0}{}(x)$ then there is $\chi$ such that $(y,\chi)\in\nbfunc{\str{V}}{}((x,\chi(x,\str V)))$ (and since $\str V$ is generic, such $\chi$ is uniquely determined). So
$$\pi\left(\func{B}{}((x,\str{V}))\right)=\left\{y : y\in\nbfunc{\str B_0}{}(x)\right\} = \nbfunc{\str B_0}{}(x).$$
This concludes the proof.
\end{proof}

\begin{proof}[Proof of Claim~\ref{c:faithful:emb}]
First we will prove that $A'$ is a generic set. By definition this happens if $V = \bigcup_{(x,\str V_x)\in A'} V_x = \bigcup_{x\in A} V_x$ is a generic set. Note that $V = \{(x,\chi_x) : x\in A\}$. Let $x\neq y\in A$ be arbitrary and pick $\str I\in U(x)\cap U(y)$. By the construction we have $\chi_x(\str I) = u_\str I(x) \neq u_\str I(y) = \chi_y(\str I)$, hence $(x,\chi_x)$ and $(y,\chi_y)$ form a generic pair which implies that $V$ is indeed a generic set. From this it follows that in particular every $V_x$ is a generic set and hence $\psi$ is a function $A\to B$.

Next fix a relation $\rel{}{}\in L$ and a tuple $\bar{x} \in A^n$. We will prove that $\psi(\bar{x})\in\rel{B}{}$ if and only if $\bar{x}\in\rel{A}{}$. By the definition of $\str B$ the ``only if'' part is immediate, to prove the other implication, we need to prove that $\psi(\bar{x})$ (understood as a set) is generic, but clearly $\psi(\bar{x})$ is a subset of a generic set $V$, which concludes the proof.

Finally we prove that for every (unary) function $\func{}{}\in L$ and every vertex $x\in \str A$ we have $\psi(\func{A}{}(x)) = \func{B}{}(\psi(x))$ (remember that $\psi_L = \id_L$). Clearly
$$\psi(\func{A}{}(x)) = \{(y, \str V_y) : y\in \func{A}{}(x)\}.$$
Put $X = \func{B}{}(\psi(x))$. By the construction we have 
$$X = \func{B}{}((x,\str{V}_x))=\left\{(y,\cl_{\str{V}_x}((y,\chi))) : (y,\chi)\in \nbfunc{\str{V}_x}{}((x,\chi(x,\str V_x)))\right\}.$$
Note that $\chi(x,\str V_x)) = \chi_x$ and that if $(y,\chi)\in \nbfunc{\str{V}_x}{}((x,\chi(x,\str V_x)))$, then $\chi = \chi_y$. Hence, in particular, $\pi$ is injective on $\nbfunc{\str{V}_x}{}((x,\chi_x))$. So we can write
$$X = \left\{(y,\cl_{\str{V}_x}((y,\chi_y))) : (y,\chi_y)\in \nbfunc{\str{V}_x}{}((x,\chi_x))\right\}.$$
Since $(\id_L,(y,\chi_y)\mapsto y)$ is an isomorphism $\str V_x\to\cl_\str{A}(x)$, we get that $(y,\chi_y)\in \nbfunc{\str{V}_x}{}((x,\chi_x))$ if and only if $y \in \nbfunc{\cl_\str{A}(x)}{}(x) = \func{A}{}(x)$. For the same reason, $\cl_{\str{V}_x}((y,\allowbreak\chi_y))$ is isomorphic to $\cl_\str A(y)$ by projecting to the first coordinate and hence in fact $\cl_{\str{V}_x}((y,\chi_y)) = \str V_y$.
Putting this together, we get that
$$\func{B}{}(\psi(x)) = X = \left\{(y,\str V_y) : y\in \func{A}{}(x)\right\} = \psi(\func{A}{}(x)).$$
\end{proof}

\begin{proof}[Proof of Claim~\ref{c:faithful:auto}]
It is easy to see that $\widetilde\varphi$ is a bijection which maps generic sets to generic sets. Fix a relation $\rel{}{}\in L$ and a tuple $((x_1, \str V_1), \ldots, (x_n,\str V_n))\in B^n$. Note that $$(x_1,\ldots,x_n)\in \nbrel{\str B_0}{} \iff (\hat\varphi(x_1),\ldots,\hat\varphi(x_n)) \in \permnbrel{\widetilde\varphi}{\str B_0}{},$$ because $\hat\varphi$ is an automorphism of $\str B_0$. Together with the fact that $\widetilde\varphi$ maps generic sets to generic sets it follows that $((x_1, \str V_1), \ldots, (x_n,\str V_n))\in \rel{B}{}$ if and only if $(\widetilde\varphi((x_1, \str V_1)), \ldots, \widetilde\varphi((x_n,\str V_n)))\in \permrel{\widetilde\varphi}{B}{}$.

It remains to prove that for every function $\func{}{}\in L$ and every $(x,\str V)\in B$ we have $\widetilde\varphi(\func{B}{}((x,\str V))) = \permfunc{\varphi_L}{B}{}(\widetilde\varphi((x,\str V)))$. To simplify the notation we put $h(\str V) = (\varphi_L,\hat q)(\str V)$ for every valuation structure $\str V$. Fix an arbitrary $(x,\str V)\in \str B$ and put $\chi_0 = \chi(x,\str V)$.
By the definition of $\str B$ we know that
$$\widetilde\varphi(\func{B}{}((x,\str V))) = \left\{(\hat\varphi(y),h(\cl_{\str{V}}((y,\chi)))) : (y,\chi)\in \nbfunc{\str{V}}{}((x,\chi_0))\right\}.$$
Denote $X = \permfunc{\widetilde{\varphi}}{B}{}(\widetilde{\varphi}((x,\str V)))$. By the definition of $\str B$, we have
$$X = \{(y,\cl_{h(\str V)}((y,\chi))) : \allowbreak (y,\chi)\in \permnbfunc{\widetilde{\varphi}}{h(\str V)}{}(\hat q((x,\chi_0)))\}.$$
Since $\hat q$ is a bijection $\mathcal V\to \mathcal V$ which agrees with $\hat\varphi$ on the first coordinate, we can write
$$X = \{(\hat\varphi(y),\cl_{h(\str V)}(\hat q((y,\chi)))) : \allowbreak \hat q((y,\chi))\in \permnbfunc{\widetilde{\varphi}}{h(\str V)}{}(\hat q((x,\chi_0)))\}.$$
Note that $$\permnbfunc{\widetilde{\varphi}}{h(\str V)}{}(\hat q((x,\chi_0))) = \hat q(\nbfunc{\str V}{}((x,\chi_0))),$$ hence 
$\hat q((y,\chi))\in \permnbfunc{\widetilde{\varphi}}{h(\str V)}{}(\hat q((x,\chi_0)))$ if and only if $(y,\chi)\in \nbfunc{\str V}{}((x,\chi_0))$, and so we have
$$X = \{(\hat\varphi(y),\cl_{h(\str V)}(\hat q((y,\chi)))) : \allowbreak (y,\chi)\in \nbfunc{\str V}{}((x,\chi_0))\}.$$
Finally, $\cl_{h(\str V)}(\hat q((y,\chi))) = h(\cl_{\str{V}}((y,\chi)))$, hence indeed $$X = \widetilde{\varphi}(\func{B}{}((x,\str V))).$$
\end{proof}

\section{Unwinding induced cycles}
\label{sec:cycles}
In this section we give a key ingredient for proving Theorem~\ref{thm:maintree}:

\begin{lemma}
\label{lem:sparsen}
Let $L$ be a language consisting of relations and unary functions equipped with a permutation group $\GammaL$. Let $\str A$ be a finite irreducible $\GammaL$-structure and let $\str B_0$ be its (finite) irreducible structure faithful EPPA-witness. Assume that $L$ contains a binary relation $E$ which is fixed by every permutation in $\GammaL$ and assume that $E_\str A$ is a complete graph. (Note that by irreducible structure faithfulness $E_{\str B_0}$ is an undirected graph without loops.)

There is a finite $\GammaL$-structure $\str{B}$ which
is an irreducible structure faithful EPPA-witness for $\str{A}$ satisfying the following:
\begin{enumerate}
\item There is a homomorphism-embedding $f\colon \str{B}\to \str{B}_0$.
\item Let $C$ be a subset of $B$. Then at least one of the following holds:
\begin{enumerate}
	\item\label{lem:sparsen:1} $E_\str{B}\cap C^2$ contains no (induced) cycle of length $\geq 4$,
	\item\label{lem:sparsen:2} $|f(C)| < |C|$, or
	\item\label{lem:sparsen:3} $|E_{\str B_0}\cap f(C)^2| > |E_\str{B}\cap C^2|$.
\end{enumerate}
\end{enumerate}

Moreover, if $\str{B}_0$ is a coherent EPPA-witness for $\str{A}$, then $\str{B}$ is also coherent.
\end{lemma}
Note that since $f$ is a homomorphism-embedding, we get that if  $|f(C)| = |C|$, then $|E_{\str B_0}\cap f(C)^2| \geq |E_\str{B}\cap C^2|$ and $f$ induces an injective mapping from $E_\str{B}\cap C^2$ to $E_{\str B_0}\cap f(C)^2$.

In the rest of the section, we will prove Lemma~\ref{lem:sparsen}. The construction is inspired by a similar construction for EPPA for metric spaces by the authors~\cite{Hubicka2018metricEPPA}.

For the rest of the section, fix $L$, $\GammaL$, $\str A$ and $\str B_0$ as in the statement of Lemma~\ref{lem:sparsen}. Assume without loss of generality that $\str A \subseteq \str B_0$. 

\addtocontents{toc}{\SkipTocEntry}
\subsection*{Valuations}
A sequence $(c_1,\ldots, c_k)$ of distinct vertices of $\str{B}_0$ is a \emph{bad cycle
sequence} if $k\geq 4$ and the structure induced by $E_{\str{B}_0}$ on
$\{c_1,\ldots, c_k\}$ is a graph cycle containing precisely the edges connecting $c_i$ and
$c_{i+1}$ for every $1\leq i\leq k$ (where we identify $c_{k+1}=c_1$).

Given a vertex $x\in B_0$, we denote by $U(x)$ the set of all bad cycle sequences
containing $x$. We call functions $U(x)\to \{0,1\}$ \emph{valuation functions for $x$}.
Given vertices $x,y\in B_0$ and their valuation functions $\chi$ and $\chi'$, we say
that the pairs $(x,\chi)$ and $(y,\chi')$ are {\em generic}, if either
$(x,\chi)=(y,\chi')$, or $x\neq y$ and for every bad cycle sequence
$\vec{c}=(c_1,\ldots, c_k)\in U(x)\cap U(y)$, one of the following holds:
\begin{enumerate}
 \item There is $1\leq i < k$ such that $\{c_i, c_{i+1}\}=\{x,y\}$ and $\chi(\vec{c})=\chi'(\vec{c})$, or
 \item $\{c_1,c_k\}=\{x,y\}$ and $\chi(\vec{c})\neq \chi'(\vec{c})$.
\end{enumerate}
A set $S$ of pairs $(x,\chi)$ is \emph{generic} if every pair $(x,\chi),(y,\chi')\in S$ is generic.

Let $x\in B_0$ be a vertex of $\str B_0$. A \emph{valuation structure for $x$} is a $\GammaL$-structure $\str{V}$
such that:
\begin{enumerate}
\item The vertex set $V$ is a generic set of pairs $(y,\chi)$ where $y\in \cl_{\str{B}_0}(x)$ and $\chi$ is a valuation function for $y$.
\item The pair $(\id_L, (y,\chi)\mapsto y)$ is an isomorphism of $\str V$ and $\cl_{\str{B}_0}(x)$.
\end{enumerate}
Let $\str V$ be a valuation structure for $x$. We denote by $\chi(x, \str{V})$ the valuation function for $x$ such that $(x,\chi(x, \str{V}))\in V$. Similarly as in Section~\ref{sec:faithful}, if $L$ contains no functions then every valuation structure $\str V$ for $x\in B_0$ contains exactly one vertex $(x,\chi(x,\str V))$ and conversely, for every valuation function $\chi$ for $x$ there is exactly one valuation structure $\str V$ for $x$ such that $\chi(x,\str V) = \chi$.

A set $S$ of pairs $(x,\str V)$, where $\str V$ is a valuation structure for $x$, is \emph{generic}, if the union $\bigcup_{(x,\str V)\in S} V$ is generic. 

\addtocontents{toc}{\SkipTocEntry}
\subsection*{Witness construction}
Now we construct a $\GammaL$-structure $\str{B}$: 
\begin{enumerate}
\item The vertices of $\str{B}$ are all pairs $(x, \str V)$ where $x\in B_0$ and $\str V$ is a valuation structure for $x$.
\item For every relation symbol $\rel{}{}\in L_\mathcal R$, we put
$$((x_1,\str{V}_{1}),\ldots, (x_{\arity{}{}},\str{V}_{\arity{}{}}))\in \rel{B}{},$$
if and only if $(x_1,\ldots, x_{\arity{}{}})\in \nbrel{\str{B}_0}{}$, and $\{(x_1,\str{V}_{1}),\ldots, (x_{\arity{}{}},\str{V}_{\arity{}{}})\}$ is generic.

\item for every (unary) function symbol $F\in L_\mathcal F$ and every vertex $(x,\str{V})\in B$, we put 
$$\func{B}{}((x,\str{V}))=\{(y,\cl_{\str{V}}((y,\chi))):(y,\chi)\in \nbfunc{\str V}{}((x, \chi(x,\str V)))\}.$$
\end{enumerate}
\begin{claim}\label{c:cycles:irreducible}
If $\str D$ is an irreducible substructure of $\str B$, then $D$ is generic.
\end{claim}
Define $\pi_B(x, \str{V})=x$ and $\pi_L = \id_L$. We then have the following:
\begin{claim}\label{c:cycles:b}
$\str B$ is a finite $\GammaL$-structure and $\pi$ is a homomorphism-embedding from $\str B$ to $\str B_0$ which is an embedding on every generic $D\subseteq B$.
\end{claim}

\medskip

Observe that since every pair of distinct vertices $x,y\in A$ is in $\rel{A}{E}$,
it follows that every bad cycle sequence contains at most two vertices of $\str{A}$, and if it contains precisely two, then they are adjacent in $E_\str{B_0}$.
For every bad cycle sequence $\vec{c}=(c_1,\ldots, c_k)$ containing at least one
vertex of $\str{A}$, we define a function $\chi_{\vec{c}}\colon A\cap \{c_1,\ldots, c_k\}\to \{0,1\}$ as follows.
$$\chi_{\vec{c}}(x)=\begin{cases}
1 &\text{ if } x=c_1 \text{ and } c_k\in A,\\
0 &\text{ otherwise}.
\end{cases}$$

Next we give an embedding $\psi\colon\str{A}\to\str{B}$.
Given a vertex $x\in A$, we define a valuation function $\chi_x$ for $x$, putting $\chi_x(\vec{c})=\chi_{\vec{c}}(x)$ for every $\vec{c}\in U(x)$, and we define a valuation structure $\str V_x$ for $x$ with $V_x = \{(y, \chi_y):y\in \cl_\str{A}(x)\}$ such that $\pi$ restricted to $V_x$ is an isomorphism $\str V_x\to\cl_\str{A}(x)$ (which is a substructure of $\str A$). We put $\psi_A(x) = (x, \str V_x)$ and $\psi_L=\id$.

\begin{claim}\label{c:cycles:emb}
$\psi$ is an embedding $\str A\to \str B$ and $\str A'=\psi(\str A)$ is generic.
\end{claim}

\addtocontents{toc}{\SkipTocEntry}
\subsection*{Constructing the extension}
Similarly as in the last section, we will need to prove irreducible structure faithfulness and the following slightly more general extension lemma will be useful in proving that.

\begin{lemma}\label{lem:cycles-extension}
Let $\varphi$ be a partial automorphism of $\str B$ satisfying the following conditions:
\begin{enumerate}
\item Both the domain and the range of $\varphi$ are generic, and
\item there is an automorphism $\hat\varphi$ of $\str B_0$ which extends the projection of $\varphi$ via $\pi$.
\end{enumerate}
Then there is an automorphism $\widetilde\varphi$ of $\str B$ extending $\varphi$.
\end{lemma}
Note that by Claim~\ref{c:cycles:b}, $\pi$ is an embedding on generic sets, hence the projection of $\varphi$ via $\pi$ is a partial automorphism of $\str B_0$.
\begin{proof}

Let $F$ be the set consisting of all bad cycle sequences $\vec{c}$ for which there is a vertex $(x,\str{V})\in \dom(\varphi)$ such that that $\chi(x,\str{V})(\vec{c})\neq \chi(\varphi((x,\str{V})))(\hat\varphi(\vec{c}))$.

We define a function $f$ such that if $\chi$ is a valuation function for $x$ then $f(\chi)$ is a valuation for $x$ satisfying 
$$
f(\chi)(\hat\varphi(\vec{c}))=
\begin{cases}
  \chi(\vec{c}) &  \text{if }\vec{c}\notin F \\
  1-\chi(\vec{c}) & \text{if }\vec{c}\in F.
\end{cases}
$$
Put $\mathcal V = \bigcup_{(x,\str V)\in \str B} V$. Next we define a function $\hat q\colon \mathcal V\to \mathcal V$ putting $$\hat q((x,\chi))=(\hat\varphi(x),f(\chi)),$$ and using it we construct the extension $\widetilde{\varphi}$ such that $\widetilde{\varphi}_L = \varphi_L$ and $\widetilde{\varphi}((x, \str V)) = (\hat\varphi(x), (\varphi_L,\hat q)(\str V))$.

The proof of the following claim, which will be given at the end of this section, is simply a mechanical verification that our constructions are well-defined.
\begin{claim}\label{c:cycles:auto}
$\widetilde\varphi$ is an automorphism of $\str B$ extending $\varphi$.
\end{claim}
\end{proof}

\addtocontents{toc}{\SkipTocEntry}
\subsection*{Proofs}

\begin{lemma}The following statements about $\str B$ are true:
\begin{enumerate}
	\item $\str B$ is irreducible structure faithful.
	\item If $\str B_0$ is a coherent EPPA-witness for $\str A$, then $\str B$ is a coherent EPPA-witness for $\str A'$.
\end{enumerate}
\end{lemma}
\begin{proof}
To prove irreducible structure faithfulness, let $\str I$ be an irreducible substructure of $\str B$. By Claim~\ref{c:cycles:irreducible}, $I$ is generic and hence $\pi$ is an embedding on $\str I$ (Claim~\ref{c:cycles:b}). This means that $\pi(\str I)$ is an irreducible substructure of $\str B_0$ and thus there is an automorphism $\hat\varphi$ of $\str B_0$ sending $\pi(\str I)$ to $A$. Put $\varphi$ to be the partial automorphism of $\str B$ sending $\str I$ to $\psi(\hat\varphi(\pi(\str I)))$ with $\varphi_L = \hat\varphi_L$. This is a partial automorphism of $\str B$ with generic domain and range (using Claim~\ref{c:cycles:emb}) and $\hat\varphi$ extends $\varphi$. By Lemma~\ref{lem:cycles-extension} we get an automorphism $\widetilde\varphi$ of $\str B$ extending $\varphi$, that is, $\widetilde\varphi(I)\subseteq A'$. Therefore $\str B$ is indeed irreducible structure faithful.

To prove coherence, let $\varphi_1$, $\varphi_2$ and $\varphi$ be a coherent triple of partial automorphisms of $\str A'$ and let $\hat\varphi_1$, $\hat\varphi_2$, $\hat\varphi$ be automorphisms of $\str B_0$ which are the coherent extensions of their projections by $\pi$.

Denote by $\hat q_1$, $\hat q_2$, $\hat q$ and  $F$, $F_1$ and $F_2$ the corresponding functions and sets from proof of Lemma~\ref{lem:faithful-extension}. Coherence on the first coordinate follows from coherence of $\hat\varphi_1$, $\hat\varphi_2$, $\hat\varphi$. To get coherence on the second coordinate, we need to prove that $F$ is the symmetric difference of $F_1$ and $F_2$. This follows by the same argument as in the proof of Lemma~\ref{lem:rel:coherence}: Since $\varphi = \varphi_2\circ\varphi_1$ and $\hat\varphi = \hat\varphi_2\circ\hat\varphi_1$, we have that
$$\chi(x,\str{V})(\vec{c})\neq \chi(\varphi((x,\str{V})))(\hat\varphi(\vec{c}))$$
if and only if exactly one of
$$\chi(x,\str{V})(\vec{c})\neq \chi(\varphi_1((x,\str{V})))(\hat\varphi_1(\vec{c}))$$
and
$$\chi(\varphi_1((x,\str{V})))(\hat\varphi_1(\vec{c}))\neq \chi(\varphi_2(\varphi_1((x,\str{V}))))(\hat\varphi_2(\hat\varphi_1(\vec{c})))$$
happens.
\end{proof}

To finish the proof of Lemma~\ref{lem:sparsen}, we now prove that for every $C\subseteq B$ such that $E_\str{B}\cap C^2$ is a cycle of length $\geq 4$, it holds that $\pi|_C$ is not an embedding (of the reducts to relation $E$). This would imply that whenever $C\subseteq B$ contains an induced graph cycle of length $\geq 4$, one of~(\ref{lem:sparsen:2}) and~(\ref{lem:sparsen:3}) holds.

Fix a set $C\subseteq B$ such that $E_\str{B}|_C$ is an induced graph cycle of length $\geq 4$ and, for a contradiction,
assume that its projection $\pi(C)$ is again an induced graph cycle of the same length in the relation $E_{\str{B}_0}$.
This means that we can enumerate $C$ as $(x_1,\str{V}_{1}),\ldots, (x_k,\str{V}_{k})$ such that $\vec{c}=(x_1, \ldots, x_k)$ is bad cycle sequence.
For every $1\leq i\leq k$, we have $\{(x_i,\str{V}_{i}),(x_{i+1},\str{V}_{i+1})\}\in E_\str{B}$ (identifying $(x_{k+1},\str V_{k+1})=(x_1,\str V_1)$), so in particular the set $\{(x_i,\str{V}_{i}),(x_{i+1},\str{V}_{i+1})\}$ is generic. By definition, this implies that for every $1\leq i < k$, we have
$$\chi(x_i,\str{V}_{i})(\vec{c})=\chi(x_{i+1},\str{V}_{i+1})(\vec{c}),$$
but
$$\chi(x_1,\str{V}_{1})(\vec{c})\neq\chi(x_{k},\str{V}_{k})(\vec{c}),$$
which is a contradiction.

\begin{remark}
Exactly as in the previous section, our partial automorphism extension has some functorial properties. Taking isomorphic copies, we can assume that $\str A\subseteq \str B_0$ and $\str A\subseteq \str B$ (then in particular $\pi\restriction_A = \id_A$). Given a partial automorphism $\varphi$ of $\str A$ let $\hat{\varphi}$ be its (coherent) extension to an automorphism of $\str B_0$ and let $\widetilde{\varphi}$ be the constructed extension to an automorphism of $\str B$ using $\hat{\varphi}$. Let $\sim$ be an equivalence relation on $B$ given by $x\sim y \iff \pi(x) = \pi(y)$. Then $\sim$ is a congruence with respect to $\widetilde{\varphi}$ and the natural actions of $\hat{\varphi}$ and $\widetilde{\varphi}$ on the equivalence classes of $\sim$ coincide.
\end{remark}

\begin{observation}\label{obs:unwind_bound}
The number of vertices of the EPPA-witness $\str B$ constructed in this section can be bounded from above by a function which depends only on the number of vertices of $\str B_0$.
\end{observation}
\begin{proof}
Let $m$ be the number of vertices of $\str B_0$. There are at most $(m+1)^m$ (rough estimate) bad cycle sequences and hence at most $2^{(m+1)^m}$ valuation functions for any given vertex $x\in B_0$. This means that there are at most $m2^{(m+1)^m}$ different pairs $(x,\chi)$, where $x\in B_0$ and $\chi$ is a valuation function for $x$, and thus at most $2^{m2^{(m+1)^m}}$ different generic sets. Given a generic set $V$ and a vertex $x\in B_0$, there is at most one valuation structure for $x$ with vertex set $V$. Since the vertex set of $\str B$ consists of all pairs $(x,\str V)$, where $x\in B_0$ and $\str V$ is a valuation structure for $x$, the claim then follows.
\end{proof}

\begin{proof}[Proof of Claim~\ref{c:cycles:irreducible}]
This is a word-to-word copy of the proof of Claim~\ref{c:faithful:irreducible}.

For a contradiction, suppose that it is not the case, that is, there are $(u,\chi),(u',\chi')\in \bigcup_{(x,\str V)\in D} V$ which form a non-generic pair. This implies that there are
$(x,\str{X}),(y,\str{Y})\in D$ such that $(u,\chi)\in X$ and $(u',\chi')\in Y$ (they cannot both lie in the same valuation structure, because the vertex sets of valuation structures are generic), and hence the set $\{(x,\str{X}),(y,\str{Y})\}$ is non-generic. Put $\str{E}_x = \{(a,\str{U}) \in D : (x,\str{X}) \not\in \cl_{\str{D}}((a,\str{U}))\}$ and similarly define $\str E_y$. Since closures are unary, these are substructures of $\str{D}$. Note that as closures in $\str B$ are generic, we also know that $(y, \str Y)\in \str E_x$ and $(x,  \str X)\in \str E_y$. This means that $\str E_x, \str E_y$ are both non-empty and neither is a substructure of the other.

We first prove $\str{E}_x\cup \str{E}_y = \str{D}$. Suppose for a contradiction that there is $(z,\str{Z}) \in \str{D}$ with $\{(x,\str{X}), (y,\str{Y})\} \subseteq \cl_{\str{D}}((z,\str{Z}))$. Then (by the construction of $\str B$) we have $\str{X}, \str{Y} \subseteq \str{Z}$, which is a contradiction with $(x,\str{X}),(y,\str{Y})$ forming a non-generic pair.

Fix $(a,\str{U}) \in \str{E}_x\setminus \str{E}_y$ and $(b,\str{W}) \in \str{E}_y\setminus \str{E}_x$. Because we know that $\str{Y} \subseteq \str{U}$ and $\str{X} \subseteq \str{W}$, we get that $(a,\str{U})$, $(b,\str{W})$ is not a generic pair and therefore no relation of $\str D$ contains both $(a,\str{U})$ and $(b,\str{W})$. Thus $\str{D}$ is a free amalgam of $\str{E}_x$ and $\str{E}_y$ over their intersection, which is a contradiction with its irreducibility. Therefore $D$ is indeed generic.


\end{proof}

\begin{proof}[Proof of Claim~\ref{c:cycles:b}]
Finiteness of $\str B$ follows from Observation~\ref{obs:unwind_bound}. Given $(x,\str V)\in B$, we have that $\pi(\str V) = \cl_{\str B_0}(x)$ (since $\str V$ is a valuation structure). By definition of $\str B$,
$$\pi(\func{B}{}((x,\str{V})))=\{y:(y,\chi)\in \nbfunc{\str V}{}((x, \chi(x,\str V)))\}.$$
As $V$ is generic, we have that for every $y$ there is at most one $\chi$ such that $(y,\chi) \in V$. Moreover, since $\pi$ is an isomorphism $\str V\to \cl_{\str B_0}(x)$, we can write
$$\pi(\func{B}{}((x,\str{V})))=\{y:y\in \nbfunc{\cl_{\str B_0}}{}(x)\},$$
and because $\nbfunc{\cl_{\str B_0}}{}(x) = \nbfunc{\str B_0}{}(x)$, we indeed have
$$\pi(\func{B}{}((x,\str{V}))) = \nbfunc{\str B_0}{}(x),$$
that is, $\pi$ preserves functions.

Let $\rel{}{}\in L$ be a relation and let $((x_1,\str V_1),\ldots,(x_n,\str V_n))$ be a tuple of vertices of $\str B$. Clearly, if $((x_1,\str V_1),\ldots,(x_n,\str V_n))\in \rel{B}{}$ then $\pi(((x_1,\str V_1),\ldots,(x_n,\str V_n))) = (x_1,\ldots,x_n)\in\nbrel{\str B_0}{}$, hence $\pi$ is a homomorphism. If $\{(x_1,\str V_1),\ldots,(x_n,\str V_n)\}$ is generic, the definition of $\str B$ gives us that $((x_1,\str V_1),\ldots,(x_n,\str V_n))\in \str B$ if and only if $\pi(((x_1,\str V_1),\ldots,(x_n,\str V_n))) = (x_1,\ldots,x_n)\in\nbrel{\str B_0}{}$, hence $\pi$ is an embedding on every generic set.

The fact that $\pi$ is a homomorphism-embedding now follows from Claim~\ref{c:cycles:irreducible}.
\end{proof}

\begin{proof}[Proof of Claim~\ref{c:cycles:emb}]
First we will prove that $A'$ is a generic set. By definition this happens if $V = \bigcup_{(x,\str V_x)\in A'} V_x = \bigcup_{x\in A} V_x$ is a generic set. Note that $V = \{(x,\chi_x) : x\in A\}$. Let $x\neq y\in A$ be arbitrary and pick $\vec{c}\in U(x)\cap U(y)$. Remember that since $E_\str A$ is a complete graph, there are at most two vertices of $\str A$ in every bad cycle sequence, and if there are two, then they are connected by an edge of the cycle. Hence we have $\chi_x(\vec{c}) = \chi_{\vec{c}}(x)$ and $\chi_{\vec{c}}(y) = \chi_y(\vec{c})$. By the choice of $\chi_{\vec{c}}$ we get that $(x,\chi_x)$ and $(y,\chi_y)$ form a generic pair which implies that $V$ is indeed a generic set. From this it follows that in particular every $V_x$ is a generic set and hence $\psi$ is a function $A\to B$. What follows is a word-to-word copy of the proof of Claim~\ref{c:faithful:emb}.

Next fix a relation $\rel{}{}\in L$ and a tuple $\bar{x} \in A^n$. We will prove that $\psi(\bar{x})\in\rel{B}{}$ if and only if $\bar{x}\in\rel{A}{}$. By the definition of $\str B$ the ``only if'' part is immediate, to prove the other implication, we need to prove that $\psi(\bar{x})$ (understood as a set) is generic, but clearly $\psi(\bar{x})$ is a subset of a generic set $V$, which concludes the proof.

Finally we prove that for every (unary) function $\func{}{}\in L$ and every vertex $x\in \str A$ we have $\psi(\func{A}{}(x)) = \func{B}{}(\psi(x))$ (remember that $\psi_L = \id_L$). Clearly
$$\psi(\func{A}{}(x)) = \{(y, \str V_y) : y\in \func{A}{}(x)\}.$$
Put $X = \func{B}{}(\psi(x))$. By the construction we have 
$$X = \func{B}{}((x,\str{V}_x))=\left\{(y,\cl_{\str{V}_x}((y,\chi))) : (y,\chi)\in \nbfunc{\str{V}_x}{}((x,\chi(x,\str V_x)))\right\}.$$
Note that $\chi(x,\str V_x)) = \chi_x$ and that if $(y,\chi)\in \nbfunc{\str{V}_x}{}((x,\chi(x,\str V_x)))$, then $\chi = \chi_y$. Hence, in particular, $\pi$ is injective on $\nbfunc{\str{V}_x}{}((x,\chi_x))$. So we can write
$$X = \left\{(y,\cl_{\str{V}_x}((y,\chi_y))) : (y,\chi_y)\in \nbfunc{\str{V}_x}{}((x,\chi_x))\right\}.$$
Since $(\id_L,(y,\chi_y)\mapsto y)$ is an isomorphism $\str V_x\to\cl_\str{A}(x)$, we get that $(y,\chi_y)\in \nbfunc{\str{V}_x}{}((x,\chi_x))$ if and only if $y \in \nbfunc{\cl_\str{A}(x)}{}(x) = \func{A}{}(x)$. For the same reason, $\cl_{\str{V}_x}((y,\allowbreak\chi_y))$ is isomorphic to $\cl_\str A(y)$ by projecting to the first coordinate and hence in fact $\cl_{\str{V}_x}((y,\chi_y)) = \str V_y$.
Putting this together, we get that
$$\func{B}{}(\psi(x)) = X = \left\{(y,\str V_y) : y\in \func{A}{}(x)\right\} = \psi(\func{A}{}(x)).$$
\end{proof}

\begin{proof}[Proof of Claim~\ref{c:cycles:auto}]
It is easy to see that $\widetilde\varphi$ is a bijection which maps generic sets to generic sets. Since the domain of $\varphi$ is generic, it follows that for every bad cycle sequence $\vec{c}$ there are at most two vertices from $\vec{c}$ in $\pi(\dom(\varphi))$, and if there are two of them, they are connected by an edge of the cycle. The same holds for $\range(\varphi)$.

Fix a bad cycle sequence $\vec{c}$ of length $k$ and suppose that there are distinct vertices $x,y\in\vec{c}$ and valuation structures $\str U$, $\str V$ such that $(x,\str U),(y,\str V)\in \dom(\varphi)$ and denote $\varphi((x,\str U)) = (\hat\varphi(x),\str U')$ and $\varphi((y,\str V)) = (\hat\varphi(y),\str V')$. We know that there is $i$ such that (without loss of generality) $x = \vec{c}_i$ and $y=\vec{c}_{i+1}$ (with $\vec{c}_{k+1} = \vec{c}_1$). By genericity of $\dom(\varphi)$ we know that $\chi(x,\str U)(\vec{c}) = \chi(y,\str V)(\vec{c})$ if and only if $i\neq k$. Because $\hat\varphi$ is a bijection $B_0\to B_0$ we have that $\hat\varphi(x) = \hat\varphi(\vec{c})_i$ and $\hat\varphi(x) = \hat\varphi(\vec{c})_{i+1}$. And as $\range(\varphi)$ is also generic, we get that $\chi(\hat\varphi(x),\str U')(\hat\varphi(\vec{c})) = \chi(\hat\varphi(y),\str V')(\hat\varphi(\vec{c}))$ if and only if $i\neq k$. Hence 
$$\chi(x,\str U)(\vec{c}) = \chi(y,\str V)(\vec{c}) \iff \chi(\hat\varphi(x),\str U')(\hat\varphi(\vec{c})) = \chi(\hat\varphi(y),\str V')(\hat\varphi(\vec{c})).$$
Moreover, this happens if and only if $\vec{c}\in F$, and thus $\widetilde\varphi$ extends $\varphi$. In the remaining paragraphs we prove that $\widetilde\varphi$ is an automorphism of $\str B$. The proof is in fact a word-to-word copy of the analogous argument from Claim~\ref{c:faithful:auto}.

Fix a relation $\rel{}{}\in L$ and a tuple $((x_1, \str V_1), \ldots, (x_n,\str V_n))\in B^n$. Note that $$(x_1,\ldots,x_n)\in \nbrel{\str B_0}{} \iff (\hat\varphi(x_1),\ldots,\hat\varphi(x_n)) \in \permnbrel{\widetilde\varphi}{\str B_0}{},$$ because $\hat\varphi$ is an automorphism of $\str B_0$. Together with the fact that $\widetilde\varphi$ maps generic sets to generic sets it follows that $((x_1, \str V_1), \ldots, (x_n,\str V_n))\in \rel{B}{}$ if and only if $(\widetilde\varphi((x_1, \str V_1)), \ldots, \widetilde\varphi((x_n,\str V_n)))\in \permrel{\widetilde\varphi}{B}{}$.

It remains to prove that for every function $\func{}{}\in L$ and every $(x,\str V)\in B$ we have $\widetilde\varphi(\func{B}{}((x,\str V))) = \permfunc{\varphi_L}{B}{}(\widetilde\varphi((x,\str V)))$. To simplify the notation we put $h(\str V) = (\varphi_L,\hat q)(\str V)$ for every valuation structure $\str V$. Fix an arbitrary $(x,\str V)\in \str B$ and put $\chi_0 = \chi(x,\str V)$.
By the definition of $\str B$ we know that
$$\widetilde\varphi(\func{B}{}((x,\str V))) = \left\{(\hat\varphi(y),h(\cl_{\str{V}}((y,\chi)))) : (y,\chi)\in \nbfunc{\str{V}}{}((x,\chi_0))\right\}.$$
Denote $X = \permfunc{\widetilde{\varphi}}{B}{}(\widetilde{\varphi}((x,\str V)))$. By the definition of $\str B$, we have
$$X = \{(y,\cl_{h(\str V)}((y,\chi))) : \allowbreak (y,\chi)\in \permnbfunc{\widetilde{\varphi}}{h(\str V)}{}(\hat q((x,\chi_0)))\}.$$
Since $\hat q$ is a bijection $\mathcal V\to \mathcal V$ which agrees with $\hat\varphi$ on the first coordinate, we can write
$$X = \{(\hat\varphi(y),\cl_{h(\str V)}(\hat q((y,\chi)))) : \allowbreak \hat q((y,\chi))\in \permnbfunc{\widetilde{\varphi}}{h(\str V)}{}(\hat q((x,\chi_0)))\}.$$
Note that $$\permnbfunc{\widetilde{\varphi}}{h(\str V)}{}(\hat q((x,\chi_0))) = \hat q(\nbfunc{\str V}{}((x,\chi_0))),$$ hence 
$\hat q((y,\chi))\in \permnbfunc{\widetilde{\varphi}}{h(\str V)}{}(\hat q((x,\chi_0)))$ if and only if $(y,\chi)\in \nbfunc{\str V}{}((x,\chi_0))$, and so we have
$$X = \{(\hat\varphi(y),\cl_{h(\str V)}(\hat q((y,\chi)))) : \allowbreak (y,\chi)\in \nbfunc{\str V}{}((x,\chi_0))\}.$$
Finally, $\cl_{h(\str V)}(\hat q((y,\chi))) = h(\cl_{\str{V}}((y,\chi)))$, hence indeed $$X = \widetilde{\varphi}(\func{B}{}((x,\str V))).$$
\end{proof}

\section{Locally tree-like EPPA-witnesses: Proof of Theorem~\ref{thm:maintree}}\label{sec:maintree}
The goal of this section is to prove Theorem~\ref{thm:maintree} using Lemma~\ref{lem:sparsen}.

\begin{definition}\label{defn:tree-amalgamation}
Let $L$ be a language equipped with a permutation group $\GammaL$ and let $\str A$ be a finite $\GammaL$-structure. We recursively define what a \emph{tree amalgamation of copies of $\str A$} is.
\begin{enumerate}
\item If $\str D$ is isomorphic to $\str A$ then $\str D$ is a tree amalgamation of copies of $\str A$.
\item\label{tree:2} If $\str B_1$ and $\str B_2$ are tree amalgamations of copies of $\str A$, $\str D$ is a $\GammaL$-structure and $\alpha_1\colon\str A\to \str B_1$, $\alpha_2\colon\str A\to \str B_2$ and $\delta_1,\delta_2\colon \str D\to \str A$ are embeddings then the free amalgamation of $\str B_1$ and $\str B_2$ over $\str D$ with respect to $\alpha_1\circ\delta_1$ and $\alpha_2\circ\delta_2$ is also a tree amalgamation of copies of $\str A$.
\end{enumerate}
\end{definition}

The following proposition gives an alternative way of viewing tree amalgamations.

\begin{prop}\label{prop:tree-amalgamation}
Let $L$ be a language equipped with a permutation group $\GammaL$, let $\str A$ be a finite $\GammaL$-structure and let $\str C$ be a finite $\GammaL$-structure. The following statements are equivalent:
\begin{enumerate}
\item\label{treeam1} $\str C$ is a tree amalgamation of copies of $\str A$.
\item\label{treeam2} There exists a sequence $\str A = \str C_1,\ldots, \str C_n=\str C$ of finite $\GammaL$-structures such that for every $1\leq i < n$ there is a $\GammaL$-structure $\str D$, embeddings $\delta_1,\delta_2\colon \str D\to\str A$ and an embedding $\alpha\colon \str A\to \str D_i$ such that $\str D_{i+1}$ is a free amalgamation of $\str A$ and $\str D_i$ with respect to $\delta_1$ and $\alpha\circ\delta_2$.
\end{enumerate}
\end{prop}
Note that if $\str A$ is irreducible (which will always be the case in this paper), the process in point~\ref{treeam2} can be understood as having a graph tree $\str T$ whose vertices precisely correspond to copies of $\str A$ in $\str C$ and each edge determines, how the neighbouring copies of $\str A$ overlap.

\begin{proof}[Proof of Proposition~\ref{prop:tree-amalgamation}]
The direction (\ref{treeam2})$\Rightarrow$(\ref{treeam1}) is trivial, as (\ref{treeam2}) is just a special case of the recursive definition of tree amalgamation of copies of $\str A$.

To obtain the other direction, we will use induction on the recursive construction of $\str C$ to prove an even stronger statement, namely that for every copy of $\str A\in \str C$, we can pick $\str C_1$ to correspond to the given copy. Clearly, this holds if $\str C$ is isomorphic to $\str A$. Suppose now that $\str C$ is the free amalgamation of $\str B_1$ and $\str B_2$ with respect to $\alpha_1\circ\delta_1$ and $\alpha_2\circ\delta_2$ as in Definition~\ref{defn:tree-amalgamation} and without loss of generality assume that the chosen copy of $\str A$ lies in $\str B_1$.

By the induction hypothesis, we get $\str A=\str C_1,\ldots,\str C_n=\str B_1$ such that $\str C_1$ corresponds to the chosen copy. Also by the induction hypothesis, we get $\str A = \str C_1',\ldots,\str C_m'=\str B_2$ such that $\str C_1'$ corresponds to the copy of $\str A$ given by $\alpha_2$. It is easy to see that if, for $1\leq i\leq m$, we put $\str C_{n+1}$ to be the free amalgamation of $\str C_n$ and $\str C_i'$ with respect to $\alpha_1\circ\delta_1$ and $\alpha_2\circ\delta_2$, then $\str C_1,\ldots,\str C_{m+n}$ witnesses that $\str C$ satisfies point~\ref{treeam2}.
\end{proof}

Note that in both equivalent definitions, we require that we only amalgamate over copies of $\str D$ which lie in a copy of $\str A$. The reason for it is that when $\str A$ is irreducible, it allows us to prove the following two observations about tree amalgamations of copies of $\str A$.
\begin{observation}\label{obs:tree-amalgamation_irreducible}
Let $\str A$ be a finite irreducible $\GammaL$-structure, let $\str C$ be a tree amalgamation of copies of $\str A$ witnessed by the sequence $\str A = \str C_1,\ldots, \str C_n=\str C$ of $\GammaL$-structures and let $\str I\subseteq \str C$ be an irreducible structure. Then either $\str I\subseteq \str C_1$, or there is $1 < i \leq n$ such that $\str I\not\subseteq \str C_{i-1}$ and $\str I$ lies fully in the copy of $\str A$ which together with $\str C_{i-1}$ forms $\str C_i$. Consequently, every embedding of an irreducible structure to $\str C$ extends to an embedding of $\str A$ to $\str C$
\end{observation}
\begin{proof}
This follows from the fact that if $\str U$ is a free amalgamation of $\str U_1$ and $\str U_2$ (without loss of generality we can assume that all the embeddings are inclusions and $U=U_1\cup U_2$) and $\str V$ is an irreducible substructure of $\str U$, then $\str V\subseteq \str U_1$ or $\str V\subseteq \str U_2$. The ``consequently'' part is immediate.
\end{proof}
Note that this in particular implies that the only copies of $\str A$ in $\str C$ are those which we added in some step of the construction of $\str C$.

\begin{observation}\label{obs:tree-amalgamation_completion}
Let $\mathcal C$ be a hereditary amalgamation class of finite $\GammaL$-structures and let $\str A\in \mathcal C$ be an irreducible structure. Suppose that $\str C$ is a tree amalgamation of copies of $\str A$. Then there is $\str E\in \mathcal C$ and a homomorphism-embedding $e\colon\str C\to \str E$.
\end{observation}
\begin{proof}
We will proceed by induction on the recursive definition of $\str C$. If $\str C$ is isomorphic to $\str A$, then the statement clearly holds with $e$ being the identity. Otherwise we get $\str B_1$, $\str B_2$, $\str D$, $\delta_1$, $\delta_2$, $\alpha_1$ and $\alpha_2$ as in Definition~\ref{defn:tree-amalgamation}. By the induction hypothesis, we get $\str E_1,\str E_2\in \mathcal C$ and homomorphism-embeddings $e_1\colon\str B_1\to\str E_1$ and $e_2\colon\str B_2\to\str E_2$. Since $\str A$ is irreducible, we get that $e_i\circ\alpha_i$ is an embedding $\str A\to \str E_i$ for $i\in \{1,2\}$, hence in particular the structure induced by $\str E_i$ on $e_i(\alpha_i(\delta_i(D)))$ is isomorphic to $\str D$ for $i\in \{1,2\}$. Therefore, we can put $\str E$ to be the amalgamation of $\str E_1$ and $\str E_2$ with respect to $e_1\circ\alpha_1\circ\delta_1$ and $e_2\circ\alpha_2\circ\delta_2$.
\end{proof}

Note that the fact that we are only amalgamating over structures which lie in a copy of $\str A$ was crucial, because ``being irreducible'' is not a hereditary property (for example, if $L$ is a language containing one ternary relation $\rel{}{}$ and $\str X$ is an $L$-structure such that $X=\{a,b,c\}$ and $\rel{X}{} = \{(a,b,c)\}$, then $\str X$ is irreducible, but the substructure of $\str X$ induced on $\{a,b\}$ is not irreducible).

\medskip

We will make use of the following lemma which has a graph-theoretic proof:
\begin{lemma}
\label{lem:cuts}
Let $L$ be a language equipped with a permutation group $\GammaL$, assume that $L$ contains a binary symmetric relation $E$, and let $\str A$ be a finite irreducible $\GammaL$-structure such that $E_\str A$ is a complete graph.
Let $\str{B}$ be a $\GammaL$-structure satisfying the following:
\begin{enumerate}
\item\label{lem:cuts:1} Every irreducible substructure of $\str B$ is isomorphic to a substructure of $\str{A}$, and
\item $\str B$ contains no induced graph cycles (of length $\geq 4$) in the  relation $E_\str B$.
\end{enumerate}
Then $\str{B}$ is a substructure of a tree amalgamation of copies of $\str{A}$.
\end{lemma}
\begin{proof}
We proceed by induction on $|B|$. If $\str{B}$ is irreducible then the statement follows trivially, hence we can assume that $\str B$ is reducible.

Note that condition~\ref{lem:cuts:1} implies that if $\str C$ is an irreducible substructure of $\str B$, then $E_\str{C}$ is a clique. And conversely, whenever $E_\str{B}$ induces a clique on $C\subseteq B$ then $\cl_\str{B}(C)$ is irreducible: Indeed, suppose for a contradiction that $\cl_\str{B}(C)$ is the free amalgamation of some $\str U_1$ and $\str U_2$ over $\str V$ such that $\str U_1,\str U_2\neq \cl_\str{B}(C)$. If $C\subseteq U_1$, then we would get that $\cl_\str{B}(C)\subseteq \str U_1$ (because $\str U_1$ is a substructure of $\cl_\str{B}(C)$) which is a contradiction, similarly for $U_2$. Hence there are $x_1,x_2\in C$ such that $x_1\in U_1\setminus V$ and $x_2\in U_2\setminus V$. But this implies that $(x_1,x_2)\notin E_\str{B}$, which is a contradiction.

For the following paragraphs, we will mainly consider the graph relation $E_\str{B}$ and we will treat subsets of $B$ as (induced) subgraphs of the graph $(B, E_\str{B})$. We will use the standard terminology of graph theory.

Let $\str C$ be an inclusion minimal substructure of $\str B$ such that $C$ forms a vertex cut of $B$ (i.e. $B\setminus C$ is not connected in $E_\str{B}$) and let $C'\subseteq C$ be an inclusion minimal vertex cut of $B$. Such a $\str C$ exists, because $\str B$ is reducible. Note that from the minimality of $\str C$ it follows that $\str C = \cl_\str{B}(C')$.

First observe that from the minimality of $C'$ it follows that for every pair of distinct vertices $x,y\in C'$ there are be two distinct nonempty connected components $B_1, B_2\subset B\setminus C'$ such that both $B_1, B_2$ contain a vertex adjacent to $x$ as well as a vertex adjacent to $y$. Now observe that $C'$ is a clique: If there was a pair of vertices $x\neq y\in C'$ such that $(x,y)\notin
E_\str{B}$, we could construct an induced cycle of length $\geq 4$ using $x$ and $y$ and vertices of $B_1$, $B_2$ from the previous paragraph. This implies that $\str C$ is irreducible, because it is the closure of a clique.

From the condition on $\str C$ we get that $B\setminus C$ is not connected, that is, it can be split into two non-empty disjoint parts $B_1\cup B_2 = B\setminus C$ such that there are no edges between $B_1$ and $B_2$ (and therefore no relations or functions at all thanks to condition~\ref{lem:cuts:1}). This means that $\str B$ is the free amalgamation of (its substructures induced on) $B_1\cup C$ and $B_2\cup C$ over $\str C$. 

Using the induction hypothesis, we get $\GammaL$-structures $\str D_1$ and $\str D_2$ which are tree amalgamations of copies of $\str A$ such that the substructures induced by $\str B$ on $B_i\cup C$ are substructures of $\str D_i$ for $i\in \{1,2\}$. Since $\str C$ is irreducible, it follows that there are embeddings $\alpha_1\colon \str A\to\str D_1$, $\alpha_2\colon \str A\to\str D_2$ and $\delta_1,\delta_2\colon \str C\to \str A$ (by Observation~\ref{obs:tree-amalgamation_irreducible}) and hence we can put $\str D$ to be the free amalgamation of $\str B_1$ and $\str B_2$ over $\str C$ with respect to $\alpha_1\circ\delta_1$ and $\alpha_2\circ\delta_2$. Clearly $\str B\subseteq \str D$ (up to an isomorphism) and $\str D$ is a tree amalgamation of copies of $\str A$.
\end{proof}
Now we are ready to prove Theorem~\ref{thm:maintree}.
\begin{proof}[Proof of  Theorem~\ref{thm:maintree}]
We intend to use Lemma~\ref{lem:sparsen} as the main ingredient to this proof. However, Lemma~\ref{lem:sparsen} expects that there is a graph edge relation $E$ in the language and that $E_\str A$ is a complete graph, which is not guaranteed by the assumptions of Theorem~\ref{thm:maintree}. For this reason, we extend the language $L$ to $L^+$, adding a binary symmetric relation $E$ fixed by every permutation of the language (assuming without loss of generality that $E\notin L$), put $\str A'$ to be the $\Gamma\!_{L^+}$-structure obtained from $\str A$ by putting $E_{\str A'}$ to be the complete graph on $A$, and put $\str B_{0}'$ to be the $\Gamma\!_{L^+}$-structure obtained from $\str B_0$ by putting $E_{\str B_0'}$ to be the complete graph on $\str B_0'$.

Observe that partial automorphisms of $\str A'$ are precisely the partial automorphisms of $\str A$ (barring the new relation $E$ fixed by every permutation of the language) and $\Aut(\str B_0) = \Aut(\str B_0')$. Therefore, $\str B_0'$ is an EPPA-witness for $\str A'$ and it is coherent if $\str B_0$ is.

Put $N=(n-1){n\choose 2}+1$. Use Proposition~\ref{prop:faithful} on $\str A'$ and $\str B_0'$ to get $\str B_1'$, an irreducible structure faithful (coherent) EPPA-witness for $\str A'$ with a homomorphism-embedding $f_1\colon\str B_1'\to\str B_0'$. Next, by applying Lemma~\ref{lem:sparsen} iteratively $N$ times, we construct a sequence of
$\Gamma\!_{L^+}$-structures $\str{B}_2',\ldots,\allowbreak \str{B}_{N+1}'$ and a sequence
of maps $f_2,\ldots, f_{N+1}$, such that for every $1\leq i\leq N+1$ it holds that
$\str{B}_i'$ is an irreducible structure faithful EPPA-witness for $\str A'$, $f_i$ is a homomorphism-embedding $\str{B}_i'\to \str{B}_{i-1}'$, and if $\str B_{i-1}'$ is coherent then so is $\str B_{i}'$.

Put $\str{B}'=\str{B}_{N+1}'$ and put $\str B$ to be the $\GammaL$-reduct of $\str B'$ forgetting the relation $E$. Since every automorphism of $\str B'$ is also an automorphism of $\str B$, we get that $\str B$ is an EPPA-witness for $\str A$, and if $\str B_0$ was coherent, then so is $\str B$.  To see that $\str B$ is irreducible structure faithful, note that an irreducible substructure of $\str B$ is also an irreducible substructure of $\str B'$.

Let $\str{C}_{N+1}$ be a substructure of $\str{B}$ on
at most $n$ vertices and let $\str C_{N+1}'$ be a substructure of $\str B'$ on the same vertices as $\str C_{N+1}$. Denote by $\str{C}_1',\ldots, \str{C}_{N}'$ the
structures such that for every $1\leq i\leq N$ it holds that
$\str{C}_{i}'=f_{i+1}(\str{C}_{i+1}')$.

Since we used Lemma~\ref{lem:sparsen} $N$ times, let us count how many times one of~(\ref{lem:sparsen:2}) and~(\ref{lem:sparsen:3}) from Lemma~\ref{lem:sparsen} has happened. Clearly, possibility~(\ref{lem:sparsen:2}) could have happened at most $n-1$ times, because $|C_{N+1}'|\leq n$ and $|C_1'|\geq 1$. And for every fixed $m=|C_i'|$, possibility~(\ref{lem:sparsen:3}) could have happened at most ${m \choose 2}\leq {n\choose 2}$ times. Therefore,~(\ref{lem:sparsen:2}) or~(\ref{lem:sparsen:3}) have together happened at most $N-1$ times, which means that there is $2\leq i\leq N+1$ such that possibility~(\ref{lem:sparsen:1}) happened in the $i$-th step. This then means that $\str C_i'$ contains no induced cycles of length $\geq 4$.

Because $\str B_i'$ is an irreducible structure faithful EPPA-witness for $\str A'$, we get that every irreducible substructure of $\str B_i'$ is isomorphic to a substructure of $\str A'$, so in particular, this holds for irreducible substructures of $\str C_i'$. Hence, we can apply Lemma~\ref{lem:cuts} on $\str{C}_i'$ to obtain a tree amalgamation $\str{D}'$ of copies of $\str A'$ and a homomorphism-embedding $f\colon \str C_{N+1}'\to \str D'$ (obtained by composing the output of Lemma~\ref{lem:cuts} with some of the $f_i$'s).

Let $\str D$ be the $\GammaL$-reduct obtained from $\str D'$ by forgetting the relation $E$. It is easy to check that $f$ is a homomorphism-embedding $\str C_{N+1}\to \str D$ and that $\str D$ is a tree amalgamation of copies of $\str A$, which concludes the proof.
\end{proof}

\begin{observation}\label{obs:main_bound}
The number of vertices of the EPPA-witness $\str B$ provided by Theorem~\ref{thm:maintree} can be bounded from above by a function which depends only on the number of vertices of $\str B_0$ and on $n$.
\end{observation}
\begin{proof}
$\str B$ was obtained from $\str B_0$ by iteratively applying Lemma~\ref{lem:sparsen} $N=(n-1){n\choose 2}+1$ times. We know that in each step, the number of vertices of the constructed structure can be bounded by a function of the number of vertices of the original structure (by Observation~\ref{obs:unwind_bound}). The claim then follows.
\end{proof}

\section{A generalisation of the Herwig--Lascar theorem: Proof of Theorem~\ref{thm:main}}\label{sec:main}
Next, we show how Theorem~\ref{thm:maintree} implies Theorem~\ref{thm:main}. Note that unlike Theorem~\ref{thm:main}, Theorem~\ref{thm:maintree} assumes that $\str A$ is irreducible, because otherwise one can not define what tree amalgamation is. In order to deal with it, we extend the language $L$ to $L^+$, adding a binary symmetric relation $E$ fixed by every permutation of the language (assuming without loss of generality that $E\notin L$), and consider the class consisting of finite $\Gamma\!_{L^+}$-structures $\str A$ where $E_\str{A}$ is a complete graph. Moreover, for such $\str A$, we will denote by $\str A^-$ its $\GammaL$-reduct forgetting the relation $E$.

We will need the following technical lemma.
\begin{lemma}\label{lem:infinitecopies}
Let $L$ be a language consisting of relations and unary functions equipped with a permutation group $\GammaL$ and fix a finite $\Gamma\!_{L^+}$-structure $\str A$ such that $E_\str{A}$ is a complete graph. Assume that there is a (not necessary finite) $\GammaL$-structure $\str{M}$ containing $\str{A}^-$ as a substructure such that every partial automorphism of $\str{A}^-$ extends to an automorphism of $\str{M}$. Then, for every tree amalgamation $\str D$ of copies of $\str A$, there is a homomorphism-embedding $h\colon \str D^-\to \str M$. Moreover, for every embedding $\alpha\colon \str A\to \str D$ there is an automorphism $f$ of $\str M$ such that $f(h(\alpha(A))) = A$.
\end{lemma}
\begin{proof}
We proceed by induction on the tree construction of $\str D$ (cf. Definition~\ref{defn:tree-amalgamation}). The claim clearly holds if $\str D$ is isomorphic to $\str A$. Suppose now that $\str B_1$ and $\str B_2$ are tree amalgamations of copies of $\str A$ and $\str E$ is a substructure of $\str A$ with embeddings $\delta_1,\delta_2\colon\str E\to\str A$, $\alpha_1\colon \str A\to \str B_1$ and $\alpha_2\colon\str A\to\str B_2$ such that $\str D$ is the free amalgamation of $\str B_1$ and $\str B_2$ over $\str E$ with respect to $\alpha_1\circ\delta_1$ and $\alpha_2\circ\delta_2$.

By the induction hypothesis, we have homomorphism-embeddings $h_1\colon \str B_1^-\to \str M$ and $h_2\colon \str B_2^-\to\str M$ and automorphisms $f_1,f_2$ of $\str M$ such that $f_i(h_i(\alpha_i(A)))=A$ for $i\in\{1,2\}$. Let $\varphi$ be a partial automorphism of $\str A^-$ sending $f_1(h_1(\alpha_1(\delta_1(\str E^-))))\mapsto f_2(h_2(\alpha_2(\delta_2(\str E^-))))$ and let $\hat\varphi$ be its extension to a partial automorphism of $\str M$. It is easy to check that the function $h\colon D\to M$ defined by
$$h(x) = \begin{cases}
\hat\varphi(f_1(h_1(x))) & \text{ if $x\in B_1$}\\
f_2(h_2(x)) & \text{ otherwise}
\end{cases}$$
is a homomorphism embedding $\str D^-\to \str M$. The moreover part follows straightforwardly as $\str A$ is irreducible and therefore every copy of $\str A$ in $\str D$ is either in $\str B_1$ or in $\str B_2$.
\end{proof}

Now we are ready to prove Theorem~\ref{thm:main}.
\begin{proof}[Proof of Theorem~\ref{thm:main}]
Let $\str A^+$ be the $\Gamma\!_{L^+}$-expansion of $\str A$ adding a clique in the relation $E$. Clearly, $\str A^+$ is in a finite orbit of the action of $\Gamma\!_{L^+}$ by relabelling, hence we can use Theorem~\ref{thm:nreppa} to get a $\Gamma\!_{L^+}$-structure $\str B_0$ which is an irreducible structure faithful coherent EPPA-witness for $\str A^+$. Let $n$ be the number of vertices of the largest structure in $\mathcal F$ and let $\str B$ be given by Theorem~\ref{thm:maintree}. We will show that $\str B^-$ satisfies the statement. Clearly, it is a coherent EPPA-witness for $\str A$. Since every irreducible substructure of $\str B^-$ is all the more so an irreducible substructure of $\str B$, we get that $\str B^-$ is irreducible structure faithful. To finish the proof, it remains to show that $\str B^-\in \Forb(\mathcal F)$.

For a contradiction, suppose that there is $\str F\in \mathcal F$ with a homomorphism-embed\-ding $g\colon \str F\to \str B^-$. We have that $|g(F)|\leq |F|\leq n$. Let $\str C$ be the substructure of $\str B$ induced on $g(F)$. From Theorem~\ref{thm:maintree}, we get a tree amalgamation $\str D$ of copies of $\str A^+$ and a homomorphism-embedding $f\colon \str C\to \str D$. Composing $f\circ g$, we get that $\str F$ has a homomorphism-embedding to $\str D^-$. However, Lemma~\ref{lem:infinitecopies} gives a homomorphism-embedding $\str D^-\to \str M$, hence we get a homomorphism-embedding $\str F\to \str M$, which is a contradiction with $\str M\in \Forb(\mathcal F)$.
\end{proof}

\section{Connections to the structural Ramsey theory: Proof of Theorem~\ref{thm:mainstrong}}\label{sec:ramsey}
Most of the applications of the Herwig--Lascar theorem proceed similarly to
applications of a theorem developed independently in the context of the structural Ramsey
theory~\cite{Hubicka2016}. Both EPPA and the Ramsey property imply the amalgamation property (cf. Observation~\ref{obs:eppaamalgamation} and~\cite{Nevsetril2005}), however, the amalgamation property is not enough to imply either of them. This motivates the following strengthening of (strong) amalgamation introduced in~\cite{Hubicka2016}:

\begin{definition}
\label{defn:completion}
Let $\str{C}$ be a structure. An irreducible structure $\str{C}'$ is a \emph{completion}
of $\str{C}$ if there is a homomorphism-embedding $\str{C}\to\str{C}'$. It is a \emph{strong} completion if the homomorphism-embedding is injective. A completion is \emph{automorphism-preserving} if it is strong and for every $\alpha\in \Aut(\str C)$ there is $\alpha'\in \Aut(\str{C}')$ such that $\alpha\subseteq \alpha'$ and moreover the map $\alpha\mapsto\alpha'$ is a group homomorphism $\Aut(\str C)\to \Aut(\str C')$.
\end{definition}
To see that completion is a strengthening of amalgamation, let
$\K$ be a class of irreducible structures. The amalgamation property for $\K$
can be equivalently formulated as follows: For $\str{A}$, $\str{B}_1$,
$\str{B}_2 \in \K$ embeddings $\alpha_1\colon\str{A}\to\str{B}_1$ and
$\alpha_2\colon\str{A}\to\str{B}_2$, there is $\str{C}\in \K$
which is a  completion of the free amalgamation of $\str{B}_1$ and $\str{B}_2$ over $\str{A}$ with respect
to $\alpha_1$ and $\alpha_2$ (which itself need not be in $\K$). In the same way, strong completion strengthens strong amalgamation and automorphism-preserving completion strengthens the so-called amalgamation property with automorphisms.

\begin{definition}\label{defn:locallyfinite}
Let $L$ be a language equipped with a permutation group $\GammaL$.
Let $\mathcal E$ be a class of finite $\GammaL$-structures and let $\mathcal K$ be a subclass of $\mathcal E$ consisting of irreducible structures. We say
that $\mathcal K$ is a \emph{locally finite subclass of $\mathcal E$} if for every $\str A\in \mathcal K$ and every $\str{B}_0 \in \mathcal E$ there is a finite integer $n = n(\str A, \str {B}_0)$ such that 
every $\GammaL$-structure $\str B$ has a completion $\str B'\in \mathcal K$, provided that it satisfies the following:
\begin{enumerate}
\item\label{locallyfinite:1} Every irreducible substructure of $\str B$ has an embedding to $\str A$,
\item there is a homomorphism-embedding from $\str{B}$ to $\str{B}_0$, and
\item every substructure of $\str{B}$ on at most $n$ vertices has a comple\-tion in $\mathcal K$.
\end{enumerate}
We say that $\mathcal K$ is a \emph{locally finite automorphism-preserving subclass of $\mathcal E$} if $\str B'$ can always be chosen to be automorphism-preserving.
\end{definition}
Note that if $\mathcal K$ is hereditary, point~\ref{locallyfinite:1} implies that every irreducible substructure of $\str B$ is in $\mathcal K$. Note also that we are only promised that every substructure on at most $n$ vertices has a completion in $\mathcal K$, even though we are asking for an automorphism-preserving (hence, in particular, strong) completion.

Luckily, for languages where all functions are unary, one can prove that if a structure has a completion in a strong amalgamation class then it has in fact a strong completion, which makes verifying local finiteness much easier. This was first proved in~\cite{Hubicka2016} as Proposition~2.6, we include a proof for completeness.

\begin{prop}
\label{prop:strongcompletion}
Let $L$ be a language equipped with a permutation group $\GammaL$ such that all function symbols of $L$ are unary and let $\mathcal K$ be a hereditary class of finite irreducible $\GammaL$-structures with the strong amalgamation property.
For every finite $\GammaL$-structure $\str{A}$, it holds that it has a completion in $\mathcal K$ if and only if it has a strong completion in $\mathcal K$.
\end{prop}
\begin{proof}
One implication is trivial. To prove the other, assume to the contrary that there is a $\GammaL$-structure $\str{A}$ with no strong
completion in $\mathcal K$, a $\GammaL$-structure $\str{B}\in \mathcal K$ and a
homomorphism-embedding $f\colon \str{A}\to\str{B}$ (that is, $\str B$ is a completion of $\str A$). Among
all such examples, choose one with $\lvert\{\{u,v\}\subseteq A : f(u) = f(v)\}\rvert$ minimal. Note that this implies that whenever there is a $\GammaL$-structure $\str A'$ and homomorphism-embeddings $g_1\colon \str A\to\str A'$ and $g_2\colon \str A'\to \str B$ such that $f = g_2\circ g_1$ and $g_1$ is surjective, we have that either $g_1$ is injective, or $g_2$ is injective (as otherwise $g_2\colon \str A'\to \str B$ contradicts the minimality).





\begin{figure}
\centering
\includegraphics{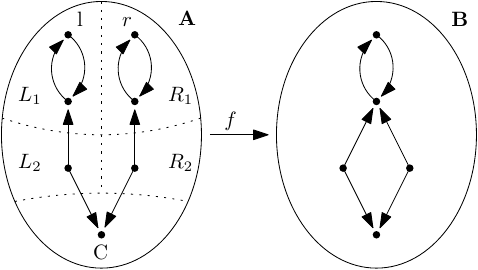}
\caption{An example of a decomposition of a structure $\str{A}$ containing one unary function constructed in the proof of Proposition~\ref{prop:strongcompletion}.}
\label{fig:Merge}
\end{figure}
We decompose the vertex set of $\str{A}$ into five parts denoted by $L_1$, $L_2$, $R_1$, $R_2$, and $C$ as depicted in Figure~\ref{fig:Merge} by the following procedure.

Because $f$ is not a strong completion in $\mathcal K$, we know that there is a pair
of vertices $l\neq r\in A$ such that $f(l)=f(r)$.  Now observe that, by
the non-existence of $\str A'$, for every other pair of vertices $v_1\neq
v_2\in A$ satisfying $f(v_1)=f(v_2)$ it holds that one vertex is in $\cl_\str{A}(l)$
and the other is in $\cl_\str{A}(r)$: Indeed, otherwise we could first identify only vertices from $\cl_\str{A}(l)$ with vertices from $\cl_\str{A}(r)$, yielding such a structure $\str A'$.

Because vertex closures are irreducible
substructures, we know that $f$ identifies two irreducible substructures
$\str{U}=\cl_\str{A}(l)$ and $\str{V}=\cl_\str{A}(r)$ of $\str{A}$ to one and is injective otherwise.

Put $L_1=U\setminus V$ and $R_1=V\setminus U$. Observe that because $l$ and $r$ can be chosen arbitrarily, if a substructure
of $\str{A}$ contains a vertex of $L_1$ then it contains all vertices of $L_1$ (otherwise we would again get a contradiction with the non-existence of $\str A'$). By symmetry, the same holds for
 $R_1$.  Denote by $L_2$ the set of all vertices $v\in A\setminus
L_1$ such that $L_1\subseteq \cl_\str{A}(v)$.  Analogously denote by $R_2$ the set of
all vertices $v\in A\setminus R_1$ such that $R_1\subseteq \cl_\str{A}(v)$.  $L_2$
and $R_2$ are disjoint, because $f$ is an embedding on irreducible substructures,
and thus no vertex closure (which is an irreducible substructure) can contain
both $L_1$ and $R_1$ (as $f(L_1)=f(R_1)$). By a similar irreducibility argument, we get that there is no tuple $\bar{t}\in \rel{A}{}$, $\rel{}{}\in L$, containing both a vertex from $L_1\cup L_2$ and a vertex from $R_1\cup R_2$.

Let $C$ be the set of all vertices
whose vertex closure does not contain $L_1$ nor $R_1$, that is, $C=A\setminus (L_1\cup L_2\cup R_1\cup R_2)$.  
Because all functions are unary, $\str{A}$ induces a substructure $\str{C}$ on $C$.
Similarly, denote by $\str{A}_l$ the substructure induced by $\str{A}$ on $C\cup L_1\cup L_2$, 
and by $\str{A}_r$ the substructure induced by $\str{A}$ on $C\cup R_1\cup R_2$.

Because $\mathcal K$ is hereditary and $f$ is injective on $A\setminus (L_1\cup R_1)$, we know that $\str B\in \mathcal K$ is a strong completion of all of $\str{A}_l$, $\str{A}_r$ and $\str{C}$.
Applying the strong amalgamation property of $\mathcal K$,
there is $\str{D}\in \mathcal K$ which is a strong amalgamation of $f(\str{A}_l)$ and $f(\str{A}_r)$
over $f(\str{C})$, hence a strong completion of $\str A$, which is a contradiction.
\end{proof}
Note that in~\cite{Hubicka2016} it is also observed that the unarity assumption of Proposition~\ref{prop:strongcompletion} cannot be omitted. We now prove Theorem~\ref{thm:mainstrong}.

\begin{proof}[Proof of Theorem~\ref{thm:mainstrong}]
Given $\str{A}\in \K$, use the fact that $\mathcal E$ has (coherent) EPPA to obtain a (coherent) EPPA-witness $\str{B}_0\in \mathcal E$. Let $n=n(\str A, \str B_0)$ be as in the definition of a locally finite subclass and let $\str B_1$ and a homomorphism-embedding $f\colon \str B_1\to \str B_0$ be given by Theorem~\ref{thm:maintree} for $\str A$, $\str B_0$ and $n$.

Because $\str B_1$ is irreducible structure faithful, it follows that every irreducible structure of $\str B_1$ can be sent by an automorphism to $\str A$. We also get that every substructure $\str D\subseteq\str B_1$ on at most $n$ vertices has a homomorphism-embedding to a tree amalgamation of copies of $\str A$. Using Observation~\ref{obs:tree-amalgamation_completion}, we obtain $\str E\in \mathcal K$ and a homomorphism-embedding $\str D\to \str E$, by composing these two homomorphism-embedding, we get that every substructure of $\str B_1$ on at most $n$ vertices has a (not necessarily strong) completion in $\K$, and Proposition~\ref{prop:strongcompletion} gives us that it has a strong completion in $\K$.

Now we can use the fact that $\K$ is a locally finite automorphism-preserving subclass of $\mathcal E$ to get an automorphism-preserving completion $\str B$ of $\str B_1$. Finally, if $\str B_0$ was coherent, then $\str B_1$ and consequently $\str B$ are coherent, too, thanks to the moreover part of Definition~\ref{defn:completion}.
\end{proof}

\section{Applications}\label{sec:applications}
In this section we present three applications of our general results.

\subsection{Free amalgamation classes}
We characterize free amalgamation classes of finite $\GammaL$-structures with relations and unary functions which have EPPA. We start with an easy observation.
\begin{observation}[\cite{hodkinson2003, Evans3, Siniora2}]
\label{obs:free}
Let $\mathcal K$ be a free amalgamation class, let $\str{A}\in \mathcal K$ be a finite structure and let $\str{B}$
be an irreducible structure faithful EPPA-witness for $\str{A}$. Then $\str{B}\in \mathcal K$.
\end{observation}
\begin{proof}
Assume for a contradiction that $\str{B}\notin \mathcal K$. Let $\str{B}_0$ be an
inclusion minimal substructure of $\str{B}$ such that $\str{B}_0\notin \mathcal K$.
Because $\mathcal K$ is a free amalgamation class it follows that $\str{B}_0$ is
irreducible. However, this is a contradiction with the existence of an automorphism $\varphi$ of $\str{B}$
such that $\varphi(\str B_0)\subseteq \str A$.
\end{proof}

Now we can prove Corollary~\ref{cor:free} which characterises free amalgamation classes with EPPA.
\begin{proof}
If there is $\str A\in \K$ which lies in an infinite orbit of the action of $\GammaL$ by relabelling then by Theorem~\ref{thm:negative} there is no finite EPPA-witness for $\str A$, hence $\K$ does not have EPPA.

If $\str A\in \K$ lies in a finite orbit of the action of $\GammaL$ by relabelling then by Theorem~\ref{thm:nreppa} there is a finite irreducible structure faithful coherent EPPA-witness $\str B$ for $\str A$. By Observation~\ref{obs:free} $\str B$ lies in $\mathcal K$.
\end{proof}

\subsection{Metric spaces without large cliques}
We continue with an example of an application of Theorem~\ref{thm:mainstrong}, which was first proved by Conant~\cite[Theorem~3.9]{Conant2015} (see also~\cite{Aranda2017}).
\begin{prop}\label{prop:metric}
Let $\str K_n$ denote the metric space on $n$ vertices where all distances are 1. The class $\mathcal M_n$ of all finite integer-valued metric spaces which do not contain a copy of $\str K_n$ has coherent EPPA for every $n\geq 2$. 
\end{prop}
\begin{proof}
We will consider integer-valued metric spaces to be relational structures in the language $L=\{R^1, R^2, \ldots\}$ (with trivial $\GammaL$), where $(x,y)\in R^a$ if and only if $d(x,y)=a$. We do not explicitly represent $d(x,x)=0$. Let $\mathcal E_n$ be the class of all $L$-structures $\str A$ such that $\rel{A}{i}$ is symmetric and irreflexive for every $\rel{}{i}\in L$, for every pair of vertices $x,y\in A$ it holds that $\{x,y\}$ is in at most one of $\rel{A}{i}$ and $\str K_n\not\subseteq \str A$.

Clearly, $\mathcal E_n$ is a free amalgamation class, and since $\GammaL$ is trivial, we get that every orbit of the action of $\GammaL$ by relabelling has size 1. Therefore, by Corollary~\ref{cor:free}, $\mathcal E_n$ has irreducible structure faithful coherent EPPA. $\mathcal M_n$ is a hereditary subclass of $\mathcal E_n$ and consists of irreducible substructures. We need to verify that $\mathcal M_n$ is a locally finite automorphism-preserving subclass of $\mathcal E_n$ and that it has the strong amalgamation property in order to use Theorem~\ref{thm:mainstrong} and thus finish the proof.

Note that if we have $\str B_0\in \mathcal E_n$ and a finite $\GammaL$-structure $\str B$ with a homomorphism-embedding $f\colon \str B\to \str B_0$, the following holds for $\str B$:
\begin{enumerate}
\item\label{prop:metric:1} $\str K_n\not\subseteq \str B$,
\item the relation $\rel{B}{i}$ is symmetric and irreflexive for every $i\geq 1$,
\item every pair of vertices $x,y\in B$ is in at most one $\rel{B}{i}$ relation, and
\item\label{prop:metric:4} there is a finite set $S\subset \{1,2,\ldots\}$ such that for every $i\in \{1,2,\ldots\}\setminus S$ we have $\rel{B}{i} = \emptyset$ (i.e. $\str B$ uses only distances from $S$).
\end{enumerate}
Note also that whenever we have a structure $\str B$ satisfying conditions~\ref{prop:metric:1}--\ref{prop:metric:4}, we can equivalently view it as an \emph{$S$-edge-labelled graph}, that is, a triple $(B, E, d)$ such that $\{x,y\}\in E$ if and only if there is $i\in S$ such that $\{x,y\}\in \rel{B}{i}$ and $d\colon E\to S$ is such that $d(x,y) = i$ if and only if $\{x,y\}\in \rel{B}{i}$ (note that we write $d(x,y)$ instead of $d(\{x,y\})$).

Let $\str C=(C, E, d)$ be an $\mathbb N^+$-edge-labelled cycle (that is, $(C,E)$ is a graph cycle) and enumerate the vertices as $C=\{c_1, \ldots, c_n\}$ such that $c_i$ and $c_{i+1}$ are adjacent for every $1\leq i\leq n$ (we identify $c_{n+1}$ with $c_1$) and $d(c_1,c_n)$ is maximal. We say that $\str C$ is a \emph{non-metric cycle} if
$$d(c_1,c_n) > \sum_{i=1}^{n-1}d(c_i, c_{i+1}).$$

The following claim is standard and was used many times (e.g.~\cite{solecki2005, Nevsetvril2007, Conant2015, Hubicka2016}). For a proof, see for example Observation~2.1 of~\cite{Hubicka2018metricEPPA}.
\begin{claim}
Let $S\subset \mathbb N^+$ be a finite set of distances and let $\str B=(B, E, d)$ be a finite $S$-edge-labelled graph. There is a metric space $\str M$ on the same vertex set $B$ such that the identity is a homomorphism-embedding $\str B\to\str M$ if and only if there is no non-metric cycle $\str C$ with a homomorphism-embedding $\str C\to \str B$. Moreover, $\Aut(\str M) = \Aut(\str B)$, and if $\str K_n\not\subseteq \str B$, then $\str K_n\not\subseteq \str M$.
\end{claim}
In other words, we have a characterization of edge-labelled graphs with a completion to a metric space. Let's first see how this claim implies both strong amalgamation and local finiteness. For strong amalgamation, it is enough to observe that free amalgamations of metric spaces contain no non-metric cycles (indeed, if there was one, then we could find one in $\str B_1$ or $\str B_2$, which would be a contradiction). For local finiteness observe that there are only finitely many non-metric cycles with distances from a finite set $S$, hence there is an upper bound $n$ on the number of their vertices (which only depends on $S$) and we are done.

To conclude, we give a sketch of proof of the claim. Put $m=\max(2,\max S)$ and define function $d'\colon B^2\to \mathbb N$ as
$$d'(x,y) = \min(m,\min_{\str P\text{ a path $x\to y$ in $\str B$}} \|\str P\|),$$
where by $\|\str P\|$ we mean the sum of distances of $\str P$. It is easy to check that $(B, d')$ is a metric space, that it preserves automorphisms and that $d'|_E = d$ if and only if $\str B$ contains no (homomorphism-embedding of a) non-metric cycle. We remark that $(B, d')$ is called the \emph{shortest path completion} of $\str B$ in~\cite{Hubicka2016}.
\end{proof}

\begin{remark}
The fact that we used $\mathcal E_n$ as the base class in the proof of Proposition~\ref{prop:metric} was a matter of choice. We could also, for example, start with the class of all $L$-structures; the condition that every small enough substructure of $\str B$ has a completion in $\mathcal M_n$ would also ensure that $\rel{B}{i}$ are symmetric and irreflexive, that every pair of vertices is in at most one relation and that $\str B$ does not contain $\str K_n$.
\end{remark}

\subsection{Structures with constants}\label{sec:constants}
We show how languages equipped with a permutation group can help us reduce EPPA for languages with constants (nulary functions) to languages without constants. Since the goal of this section is to illustrate applications of our main theorems, we will only construct EPPA-witnesses for structures where the constants behave in a special way.

To simplify the notation, if $\str A$ is a $\GammaL$-structure and $c$ is a constant symbol of $\GammaL$, we will write $c_\str A$ instead of $c_\str A()$. Moreover, if the image of $c_\str A$ is a singleton $x$ (recall that, in general, functions go to the powerset of $\str A$), we will write $c_\str A = x$ instead of $c_\str A = \{x\}$.

We first give a definition.
\begin{definition}
Let $L$ be a language equipped with a permutation group $\GammaL$ and let $\str A$ be a $\GammaL$-structure. We define the \emph{constant trace} of $\str A$, denoted by $\tr(\str A)$, as
$$\tr(\str A) = \cl_\str{A}(\emptyset).$$
In particular, $\tr(\str A)$ is a (possibly empty) $\GammaL$-structure.
\end{definition}

For example, if $\GammaL$ contains no constants, then the constant traces of all $\GammaL$-structures are empty. If $\GammaL$ contains, say, two constants $a$ and $b$ and a binary relation $E$ and $\str A$ is a $\GammaL$ structure such that $a_\str A$ and $b_\str A$ are singletons, $a_\str A \neq b_\str A$, and moreover $(a_\str A, b_\str A) \in E_\str A$, then $\tr(\str A)$ is the two-vertex $\GammaL$-structure with the corresponding relation $E$.

If $\GammaL$ contains one constant symbol $c$ and one unary function symbol $\func{}{}$ and $\str A$ is a $\GammaL$-structure containing a vertex $x$ such that $c_\str A$ is a singleton, $c_\str A \neq x$ and $x\in \func{A}{}(c_\str A)$, then $\tr(\str A)$ also contains $x$.

\begin{theorem}\label{thm:constants}
Let $L$ be a language equipped with a permutation group $\GammaL$ where the arity of every function is at most 1 and let $\str A$ be a finite $\GammaL$-structure. Let $\Lconst$ be the set of all constant symbols of $L$. Assume the following:
\begin{enumerate}
\item For every $g\in \GammaL$ and every $c\in \Lconst$ it holds that $g(c) = c$.
\item $\Lconst$ is finite.
\item\label{constants:3} For every $c\in \Lconst$ it holds that $c_\str A$ is a singleton.
\item For every $c\neq c'\in \Lconst$ it holds that $c_\str A\neq c_\str A'$.
\item\label{constants:5} For every $c\in \Lconst$ and for every unary function $\func{}{}\in L$ it holds that $\func{A}{}(c_\str A) = \emptyset$.
\item $\str A$ lies in a finite orbit of the action of $\GammaL$ by relabelling.
\end{enumerate}
Then there is a finite $\GammaL$-structure $\str B$ which is an irreducible structure faithful coherent EPPA-witness for $\str A$.
\end{theorem}
We again remark that our goal here was to keep the proof as simple as possible, a similar theorem can be proved with much weaker assumptions. In fact, one can obtain a category theory-like theorem which then makes it possible to lift the main theorems of this paper to work for languages with constants. These results will appear elsewhere.

The structure of the proof will be similar to that of Proposition~\ref{prop:infinite_languages}. That is, we will define a new language without constants and we will reduce the question to the question of EPPA in that language.
\begin{proof}[Proof of Theorem~\ref{thm:constants}]
Without loss of generality we will assume that $L$ does not contain the symbol $\star$. Given a function $f\colon \{1,\ldots,n\}\to \Lconst\cup \{\star\}$, we put $\lvert f \rvert = \lvert\{i\in n : f(i)=\star\}\rvert$. Observe that from assumptions~\ref{constants:3} and~\ref{constants:5} it follows that the vertex set of $\tr(\str A)$ is precisely $\{c_\str A : c\in \Lconst\}$.

Now, we define a language $M$ without constant symbols. Let $R\in L$ be an $n$-ary relation symbol. For every function $f\colon \{1,\ldots,n\}\to \Lconst\cup \{\star\}$ such that $\lvert f \rvert>0$, we put an $\lvert f\rvert$-ary relation symbol $\rel{}{R,f}$ in $M$. Let $F\in L$ be a unary function symbol. For every $c\in S$, we put a unary relation symbol $\rel{}{F,c}$ in $M$. We also put all unary function symbols of $L$ into $M$.

Given $g\in \GammaL$, we define $\pi_g\colon M\to M$ as
$$\pi_g(T) = \begin{cases}
g(T) &\text{ if $T$ is a unary function symbol},\\
\rel{}{g(F),c} &\text{ if $T = \rel{}{F,c}$, where $F\in L$ is a unary function symbol},\\
\rel{}{g(R),f} &\text{ if $T = \rel{}{R,f}$, where $R\in L$ is a relation symbol}.
\end{cases}$$
We put $\GammaM = \{\pi_g : g\in \GammaL\}$. Observe that $\GammaM$ is a permutation group on $M$ ($\pi_{gh} = \pi_g\pi_h$). We claim that $g\mapsto \pi_g$ is a group isomorphism: Clearly it is a surjective homomorphism, injectivity follows from the fact that $M$ contains all unary function symbols of $L$, for every relation symbol $R\in L$ we have $\rel{}{R,\star}\in M$ (where by $\star$ we mean the constant $\star$ function), and every $g\in \GammaL$ fixes $\Lconst$ pointwise.

\medskip

Given an $m$-tuple $(x_1,\ldots,x_m) = \bar x\in A^m$ and a function $f\colon \{1,\ldots,n\}\to \Lconst\cup \{\star\}$ such that $\lvert f\rvert = m$, we define $\bar x^f$ to be the $n$-tuple $(y_1,\ldots, y_n)$, where
$$y_i = \begin{cases}
f(i)_\str A &\text{ if } f(i)\in \Lconst,\\
x_j &\text{ if } f(i)=\star\text{ and }\lvert\{k < i : f(k)=\star\}\rvert = j-1.
\end{cases}$$

Put $D = A\setminus \tr(\str A)$ (that is, the members of $D$ are precisely the non-constant vertices of $\str A$). We claim that for every $n$-tuple $(y_1,\ldots,y_n) = \bar y \in A^n$, there is precisely one triple $(m,\bar x,f)$, where $m\in \mathbb N$, $\bar x\in D^m$ and $f$ is a function $\{1,\ldots,n\}\to \Lconst\cup \{\star\}$ with $\lvert f\rvert = m$, such that $\bar y = \bar x^f$. Indeed, put
$$f(i) = \begin{cases}
c &\text{ if } y_i = c_\str A\text{ for some }c\in\Lconst,\\
\star &\text{ otherwise},
\end{cases}$$
$m = \lvert f\rvert$ and $x_i = y_j$, where $j$ is chosen such that $f(j)=\star$ and $\lvert\{k < j : f(k)=\star\}\rvert = i-1.$

\medskip

Let $\str C$ be a $\GammaM$-structure such that $C$ is disjoint from $K = \tr(\str A)$. We define a $\GammaL$-structure $T(\str C)$ as follows:
\begin{enumerate}
\item The vertex set of $T(\str C)$ is $C\cup K$.
\item The identity on $K$ is an isomorphism between $\tr(\str A)$ and the structure induced by $T(\str C)$ on $K$ (in particular, the constants are defined on $K$ in $T(\str C)$ in the same way as in $\str A$).
\item For every unary function $\func{}{}\in L$ and every $x\in C$, we put
$$\nbfunc{T(\str C)}{}(x) = \func{C}{}(x) \cup \{c_{T(\str C)} : c\in \Lconst\text{ and }x\in\rel{C}{\func{}{},c}\}.$$
\item For every relation $\rel{}{\rel{}{},f}\in M$ and every $\bar x\in \rel{C}{\rel{}{},f}$, we put $\bar x^f\in\nbrel{T(\str C)}{}$.
\end{enumerate}

Note that $(\pi_g,\alpha)$ is an embedding of $\GammaM$-structures $\str E\to\str F$, if and only if $(g,\alpha\cup\id_K)$ is an embedding $T(\str E)\to T(\str F)$. This follows directly from the construction. It also implies that $\str E$ lies in a finite orbit of the action of $\GammaM$ by relabelling if and only if $T(\str E)$ lies in a finite orbit of the action of $\GammaL$ by relabelling.

\medskip

Next, we define a $\GammaM$-structure $\str D$ such that $T(\str D) = \str A$. We put the vertex set of $\str D$ to be $D$, the relations and functions are defined as follows:
\begin{enumerate}
\item For every unary function $\func{}{}\in L$ and every vertex $x\in D$, we put $\func{D}{}(x) = \func{A}{}(x)\setminus \tr(\str A)$.
\item For every unary function $\func{}{}\in L$, every vertex $x\in D$ and every constant $c\in \Lconst$, we put $x\in\rel{D}{\func{}{},c}$ if and only if $c_\str A\in \func{A}{}(x)$.
\item For every $n$-ary relation $\rel{}{}\in L$ and every $\bar y\in \rel{A}{}$ such that $\bar y = \bar x^f$, where $\bar x\in D^m$, $f\colon \{1,\ldots,n\}\to \Lconst\cup \{\star\}$ and $m\geq 1$, we put $\bar x\in \rel{D}{\rel{}{},f}$.
\end{enumerate}
It is straightforward to verify that indeed $T(\str D) = \str A$.

\medskip

Since $\GammaM$ is a language where all functions are unary, by Theorem~\ref{thm:nreppa} we get an irreducible structure faithful coherent EPPA-witness $\str C$ for $\str D$. Without loss of generality we can assume that $C$ is disjoint from $K$. We claim that $\str B = T(\str C)$ is an irreducible structure faithful coherent EPPA-witness for $\str A$.

Let $(g,\alpha)$ be a partial automorphism of $\str A$. This implies that $(\pi_g,\alpha\restriction_D)$ is a partial automorphism of $\str D$, which by the assumption extends to an automorphism $(\pi_g,\theta)$ of $\str C$. This implies that $(g,\theta\cup\id_K)$ is an automorphism of $\str B$ extending $(g,\alpha)$. Since the extensions in $\str C$ can be chosen to be coherent, by the construction we get coherence also for $\str B$.

To get irreducible structure faithfulness of $\str B$, observe that if $\str P\subseteq \str C$ is the free amalgamation of $\str P_1$ and $\str P_2$ over $\str Q$, then $T(\str P)$ is the free amalgamation of $T(\str P_1)$ and $T(\str P_2)$ over $T(\str Q)$. This follows from the fact that functions in a $\GammaM$-structure $\str X$ are subsets of the corresponding functions in $T(\str X)$ and if, for $n\geq 2$, an $n$-tuple is in a relation in $\str X$, then is is a sub-tuple of a tuple in a relation of $T(\str X)$.

Taking the contrapositive, this means that if $\str I$ is an irreducible substructure of $\str B$, then $\str C$ induces an irreducible substructure on $I\setminus K$. Hence, there is an automorphism $(\pi_g,\alpha)\colon \str C\to \str C$ sending $I\setminus K$ to $A$ and thus $(g,\alpha\cup\id_K)$ is an automorphism of $\str B$ such that $(g,\alpha\cup\id_K)(I) \subseteq A$.
\end{proof}

\subsection{EPPA for special non-unary functions}\label{sec:nonunary}
One of our motivations for introducing languages equipped with a permutation group was that it gives a nice formalism to stack several EPPA constructions on top of each other, thereby allowing to prove coherent EPPA for certain classes with non-unary functions. We conclude this paper with two examples of this. This section can be seen as an introduction to Section~\ref{sec:orientations}.

The following theorem is a variant of Ivanov's observation that permomorphisms of Herwig~\cite[Lemma~1]{herwig1998} can be used to prove EPPA of equivalence relations on $n$-tuples~\cite{Ivanov2015}:
\begin{theorem}\label{thm:nonunary}
Let $L$ be a finite language consisting of two unary relations $U$, $V$ and functions $F^1,\ldots, F^n$, each of arity at least 1. Let $\mathcal C$ be the class of all finite $L$-structures $\str A$ satisfying the following:
\begin{enumerate}
\item $U_\str A \cap V_\str A = \emptyset$ and $U_\str A\cup V_\str A = A$,
\item for every $1\leq i \leq n$ it holds that $\dom(F^i_\str A) \subseteq (U_\str A)^{\arityf{i}}$ and $\range(F^i_\str A) \subseteq V_\str A$.
\end{enumerate}
(Equivalently, structures in $\mathcal C$ can be viewed as 2-sorted structures where all the functions go from the first sort to the other.) Then $\mathcal C$ has irreducible structure faithful coherent EPPA.
\end{theorem}
\begin{proof}
Fix $\str A\in \mathcal C$. We will construct $\str B\in \mathcal C$ such that $\str B$ is the desired EPPA-witness. Towards that, we define a language $L^*$ consisting of an $\arityf{i}$-ary relation $\rel{}{i,v}$ for every $1\leq i\leq n$ and every $v\in V_\str A$. Let $\GammaLstar$ be the permutation group obtained by the natural action of $\Sym(V_\str A)$ on $L^\star$. Next we define an $\GammaLstar$-structure $\str A_0$ such that the vertex set of $\str A_0$ is precisely $U_\str A$ and for every tuple $\bar{x}$ of vertices of $\str A_0$ and every relation $\rel{}{i,v}\in L^*$ we put $\bar{x}\in \nbrel{\str A_0}{i,v}$ if and only if $v\in \func{A}{i}(\bar{x})$. Let $\str B_0$ be an irreducible structure faithful coherent EPPA-witness for $\str A_0$ (obtained for example using Theorem~\ref{thm:nreppa}). Without loss of generality we can assume that $\str A_0\subseteq \str B_0$.

Next we reconstruct an $L$-structure $\str B$ using $\str B_0$ as a template as follows:
\begin{enumerate}
\item The vertex set of $\str B$ is the disjoint union $B_0 \cup V_\str A$.
\item $U_\str B = B_0$ and $V_\str B = V_\str A$.
\item For every $1\leq i\leq n$, every $v\in V_\str A$ and every tuple $\bar{x}$ from $B_0$ we put $v\in \func{B}{i}(\bar{x})$ if and only if $\bar{x}\in\nbrel{\str B_0}{i,v}$.
\end{enumerate}
Clearly, $\str B\in \mathcal C$. Since $\str A_0\subseteq \str B_0$, we get that $A\subseteq B$. To see that $\str A$ is in fact a substructure of $\str B$, observe that $U_\str A = U_\str B \cap A$, $V_\str A = V_\str B$ and whenever $\bar{x}$ is a tuple of vertices from $U_\str A$, $v\in V_\str A$ and $1\leq i\leq n$, then $v\in \func{A}{i}(\bar{x})$ if and only if $\bar{x} \in \nbrel{\str A_0}{i,v}$ (by the construction of $\str A$), which happens if and only if $\bar{x} \in \nbrel{\str B_0}{i,v}$ (since $\str A_0\subseteq \str B_0$) and this is true if and only if $v\in \func{B}{i}(\bar{x})$ (by the construction of $\str B$). Hence indeed $\str A\subseteq \str B$.

Now we show how to construct an automorphism of $\str B$ from an automorphism of $\str B_0$ and a permutation of $V_\str A$. Let $f'$ be a permutation of $V_\str A$ and let $f=(f_{L^*}, f_{B_0})$ be an automorphism of $\str B$ such that $f_{L^*}$ is induced by $f'$. Put $\theta = f_{B_0} \cup f'$. We claim that $\theta$ is an automorphism of $\str B$. Clearly, $\theta$ is a bijection $B\to B$ which preserves the unary relations. Given an arbitrary $1\leq i\leq n$, an arbitrary tuple $\bar{x}$ of vertices from $B_0$ and an arbitrary $v\in V_\str A$, we know that $v\in \func{B}{i}(\bar{x})$ if and only if $\bar{x}\in\nbrel{\str B_0}{i,v}$ (by the construction of $\str B$), which happens if and only if $f_{B_0}(\bar{x}) \in \permnbrel{f_{L^*}}{\str B_0}{i,v} = \nbrel{\str B_0}{i,f'(v)}$ (as $f$ is an automorphism and $f_{L^*}$ is induced by $f'$), and by the construction of $\str B$ it is equivalent to $f'(v) \in \func{B}{i}(f'_{B_0}(\bar{x}))$. Hence $\theta$ is an automorphism of $\str B$.

To see that $\str B$ is irreducible structure faithful, it is enough to observe that if $\str C\subseteq \str B$ is irreducible, then either $C$ consists of a single vertex of $V_\str A$, or $\str C = \cl_\str{B}(C\cap U_\str B)$ and for every pair $x\neq y\in C\cap U_\str B$ there is a tuple $\bar{x}$ of vertices of $\str C$ containing both $x$ and $y$, and $1\leq i\leq n$ such that $\func{C}{i}(\bar{x})\neq \emptyset$. Consequently, $\str B_0$ induces an irreducible substructure on $C\cap U_\str B$. By irreducible structure faithfulness there is an automorphism $f=(f_{L^*},f_{B_0})$ of $\str B_0$ such that $f(C\cap U_\str B) \subseteq A_0$. Let $f'$ be an arbitrary permutation of $V_\str A$ inducing $f_{L^*}$ and let $\theta$ be the automorphism of $\str B$ constructed from $f$ and $f'$ in the previous paragraph. Clearly, $\theta(C\cap U_\str B) \subseteq A_0$, and therefore $\theta(\str C) = \theta(\cl_\str B(C\cap U_\str B)) \subseteq \str A$. This finishes the proof of irreducible structure faithfulness.

Next we prove that $\str B$ is an EPPA-witness for $\str A$. Let $\varphi$ be a partial automorphism of $\str A$. Remember that $\varphi$ preserves the unary relations. Let $f'$ be the coherent extension of $\varphi\restriction_{V_\str A}$ to a permutation of $V_\str A$ obtained using Proposition~\ref{prop:setcoherence}. Let $f_{L^*}\in\GammaLstar$ be induced by $f'$ and put $\varphi_0 = (f_{L^*}, \varphi\restriction_{A_0})$. Observe that $\varphi_0$ is a partial automorphism of $\str A_0$ and extend it to an automorphism $f=(f_{L^*},f_{B_0})$ of $\str B_0$ (in a coherent way). Put $\widetilde{\varphi} = f_{B_0} \cup f'$. By the previous paragraphs, $\widetilde{\varphi}$ is an automorphism of $\str B_0$. Moreover, since  $\varphi\restriction_{V_\str A} \subseteq f'$ and $\varphi\restriction_{A_0}\subseteq f_{B_0}$, we get that $\widetilde{\varphi}$ extends $\varphi$.

To finish the proof, note that since both $f$ and $f'$ were chosen to be coherent, $\widetilde{\varphi}$ is coherent as well.
\end{proof}

\begin{remark}
This construction can be carried out more generally for infinitely many functions, more than 2 unary marks (as long as all functions go in one direction) and more complicated structures living on each unary mark (as long as the whole multi-sorted structure still lies in a finite orbit of the relabelling action). This will appear elsewhere. In the next section, we adapt this construction for a class which does not a priori look multi-sorted.
\end{remark}

\subsection{EPPA for $k$-orientations with $d$-closures}\label{sec:orientations}
In this section we extend the construction from Section~\ref{sec:nonunary} and prove EPPA for the class of all $k$-orientations with $d$-closures, thereby confirming a conjecture from~\cite{Evans2}. We only define the relevant classes and prove EPPA for them here, to get more context (for example the connection with Hrushovski's predimension constructions and the importance for the structural Ramsey theory), see~\cite{Evans2}.

Let $\str G$ be an oriented graph (that is, if there is an edge from vertex $u$ to vertex $v$ then there is no edge from $v$ to $u$). We say that it is a \emph{$k$-orientation} if the out-degree of every vertex is at most $k$. We say that a vertex $x\in G$ is a \emph{root} if its out-degree is strictly smaller than $k$. Let $\mathcal D^k$ be the class of all finite $k$-orientations. While $\mathcal D^k$ is not an amalgamation class, there are two natural expansions which do have the free amalgamation property:

\begin{definition}
Let $L$ be the graph language with a single binary relation $E$ and let $L_s$ be its expansion by a unary function symbol $\func{}{}$.

Let $\str G$ be a $k$-orientation. By $s(\str G)$ we denote the $L_s$-expansion of $\str G$ putting
$$\nbfunc{s(\str G)}{}(x) = \{y\in G : y\hbox{ is reachable from $x$ by an oriented path}\}.$$
Here, an oriented path from $x$ to $y$ is a sequence $x=v_1, v_2, \ldots, v_m = y$ with $m\geq 1$ such that for every $1\leq i < m$ it holds that $(v_i,v_{i+1})\in E_\str G$. Put $\mathcal D_s^k = \{s(\str G) : \str G\in \mathcal D^k\}$.
\end{definition}

Recall that $\cl_\str{A}(x)$ denotes the smallest substructure of $\str A$ containing $x$ and is called the closure of $x$ in $\str A$. For $\str G\in \mathcal D_s^k$ and $y\in G$, we denote by $\roots{}{y}$ the set of all roots of $\str G$ which are in $\cl_\str{G}(y)$. Define $\mathcal D_{s^+}^k$ to be the subclass of $\mathcal D_s^k$ such that $\str G \in \mathcal D_{s^+}^k$ if and only if for every $y\in G$ it holds that $\roots{}{y} \neq \emptyset$.

\begin{definition}
Let $L_d$ be an expansion of $L_s$ adding an $n$-ary function symbol $\func{}{n}$ for every $n\geq 1$.

Given $\str G\in\mathcal D_s^k$, we denote by $d(\str G)$ the $L_d$-expansion of $\str G$ putting $\nbfunc{d(\str G)}{n}(x_1,\ldots,\allowbreak x_n) = \emptyset$ if $(x_1,\ldots,x_n)$ is not a tuple of distinct roots and
$$\nbfunc{d(\str G)}{n}(x_1,\ldots, x_n) = \{y\in G : \roots{}{y} = \{x_1,\ldots,x_n\}\}$$
if $(x_1,\ldots,x_n)$ is a tuple of distinct roots.
Put $\mathcal D_{d}^k = \{d(\str G) : \str G\in \mathcal D_{s^+}^k\}$.
\end{definition}
Note that in the definition of $\mathcal D_{d}^k$ we are only considering members of $\mathcal D_{s^+}^k$. The reason is that if there was a vertex with $\roots{}{y} = \emptyset$, it would be in the closure of the empty set, i.e. we would need to add constants. It is possible to do so, but it would make the construction a bit complicated and for the applications we have in mind it does not make any difference.

\medskip

It is easy to see that $\mathcal D_s^k$ is a free amalgamation class. Combining with Corollary~\ref{cor:free}, we get the following theorem proved by Evans, Hubi\v cka and Ne\v set\v ril~\cite{Evans2,Evans3}.
\begin{theorem}
\label{thm:D0EPPA}
$\mathcal D^k_s$ has irreducible structure faithful coherent EPPA for every $k\geq 1$.
\end{theorem}

It is again straightforward to verify (and it was done in~\cite{Evans2}) that $\mathcal D_d^k$ is a free amalgamation class. Since it contains non-unary functions, the results of this paper cannot be applied directly to prove that $\mathcal D_d^k$ has irreducible structure faithful coherent EPPA. However, we can use the fact that the non-unary functions go from root vertices to non-root vertices and show the following theorem, which was conjectured to hold in~\cite[Conjecture~7.5]{Evans2}.
\begin{theorem}
\label{thm:DF}
$\mathcal D^k_d$ has irreducible structure faithful coherent EPPA for every $k\geq 1$.
\end{theorem}
In the rest of this section, we will prove this theorem. The proof is based on the following observation: Let $S$ be a set consisting of root vertices only, let $S_1$ be the $L_s$-closure of $S$ (i.e. we ignore the $\func{}{n}$ functions) and let $S_2$ be the $L_d$-closure of $S$ (i.e. we also consider the $\func{}{n}$ functions). Then the root vertices in $S_1$ are precisely the root vertices in $S_2$. Consequently, if one is interested in root vertices only, all closures are unary, even in the presence of higher-arity functions. Thus, we can view structures from $\mathcal D^k_d$ as two-sorted structures (one sort being the roots and the other being the non-roots) in which all non-unary functions go from one sort to the other, which allows us to use a similar structure of arguments as in Section~\ref{sec:nonunary}.

\medskip

Fix $\str A\in \mathcal D^k_d$ and denote by $\str A_0$ its $L_s$-reduct (so $\str A_0\in \mathcal D^k_s$). Let $\str B_0\in \mathcal D^k_s$ be an irreducible structure faithful coherent EPPA-witness for $\str A_0$ given by Theorem~\ref{thm:D0EPPA}.

Let $\mathfrak P$ be the set of all pairs $(x,(x_1,\ldots, x_n))$ such
that $x$ is a non-root vertex of $\str{B}_0$, $(x_1,\ldots,x_n)$ is a tuple of distinct root vertices of $\str B_0$ and $\mathrm{roots}_{\str B_0}(x) = \{x_1,\ldots, x_n\}$. Note that we have such a pair for each possible permutation of $\{x_1,\ldots, x_n\}$. Given $P = (x,(x_1,\ldots, x_n)) \in \mathfrak P$, we define $\pi(P)=x$ to be the \emph{projection} and put $|P| = n$.

Denote by $L^+$ the expansion of $L_s$ adding a $|P|$-ary relation symbol
$\rel{}{P}$ for every $P\in \mathfrak P$ and a $(|P|+1)$-ary relation symbol $\relE{}{P}$ for every $P\in \mathfrak P$.
Let $\Gamma\!_{L^+}$ be the permutation group on $L^+$ consisting of all permutations of the $\rel{}{P}$ and $\relE{}{P}$ symbols induced by the natural action of $\Aut(\str{B}_0)$ on $\mathfrak P$. In particular, $E$ and $F$ are fixed by $\Gamma\!_{L^+}$.

Denote by $\str{A}_1$ the $\Gamma\!_{L^+}$-structure created from $\str A_0$ by removing all non-root vertices, keeping the edges between root vertices, putting $\nbfunc{\str{A}_1}{}(v)=\nbfunc{\str A_0}{}(v)\cap A_1$, adding $(x_1,\ldots, x_n)\in \nbrel{\str{A}_1}{(x,(x_1,\ldots, x_n))}$ if and only if 
$x$ is a non-root vertex of $\str A_0$ and $\mathrm{roots}_{\str A_0}(x) = \{x_1,\ldots, x_n\}$, and adding $(a, x_1,\ldots, x_n)\in \nbrelE{\str{A}_1}{(x,(x_1,\ldots, x_n))}$ if and only if $(x_1,\ldots, x_n)\in \nbrel{\str{A}_1}{(x,(x_1,\ldots, x_n))}$, $a$ is a root vertex of $\str A_0$ and $(a,x)\in E_{\str A_0}$. Let $\str{B}_1$ be an irreducible structure faithful coherent EPPA-witness for $\str{A}_1$ given by Theorem~\ref{thm:nreppa}.

We will now reconstruct an $L_d$-structure $\str B\in \mathcal D_d^k$ from $\str B_1$ such that $\str B$ will be an irreducible structure faithful coherent EPPA-witness for $\str A$. The general idea is to put back the non-root vertices according to the $\rel{}{P}$ and $\relE{}{P}$ relations using $\str B_0$ as a template.

Let $\mathcal T_0$ be the set consisting of all pairs $(P, \bar{x})$ such that $P\in \mathfrak P$, $\bar{x}$ is a tuple of vertices of $\str B_1$ and $\bar{x}\in\nbrel{\str B_1}{P}$. We say that $(P, \bar{x})\sim (P', \bar{x}')$ if $\pi(P) = \pi(P')$ and $\bar{x}$ and $\bar{x}'$ are different permutations of the same set. Let $\mathcal T$ consist of exactly one (arbitrary) member of each equivalence class of $\sim$ on $\mathcal T_0$.

Put $B = B_1 \cup \mathcal T$. For $u,v\in B$, we put $(u,v)\in E_\str{B}$ if and only if one of the following holds:
\begin{enumerate}[label=C\arabic*]
\item\label{C1} $u,v\in B_1$ and $(u,v)\in E_{\str{B}_1}$,
\item\label{C2} $u\in B_1$, $v=((x,(x_1,\ldots, x_n)), (w_1,\ldots,w_n))\in \mathcal T$ and $(u, w_1,\ldots,w_n)\in \nbrelE{\str B_1}{(x,(x_1,\ldots, x_n))}$,
\item\label{C3} $u = ((x,(x_1,\ldots, x_n)), (w_1,\ldots,w_n))\in \mathcal T$, $v\in B_1$, there is $1\leq i\leq n$ such that $v=w_i$ and $(x,x_i)\in E_{\str B_0}$, or
\item\label{C4} $u = ((x,(x_1,\ldots, x_n)), (w_1,\ldots,w_n))\in \mathcal T$, $v = ((y,(y_1,\ldots, y_m)),\allowbreak(t_1,\ldots,\allowbreak t_m))\in \mathcal T$, $\{t_1,\ldots,t_m\}\subseteq \{w_1,\ldots,w_n\}$ and $(x,y)\in E_{\str B_0}$.
\end{enumerate}
For every $x\in B$ we put $$\func{B}{}(x) = \{y\in B : y\hbox{ is reachable from $x$ in $B$ by an oriented path}\}.$$
Finally, we put $\func{B}{n}(x_1,\ldots, x_n) = \emptyset$ if $(x_1,\ldots,x_n)$ is not a tuple of distinct vertices of $\str{B}_1$ and $\func{B}{n}(x_1,\ldots, x_n) = \{y\in B : \mathrm{support}(y) = \{x_1,\ldots,x_n\}\}$ if $(x_1,\ldots,x_n)$ is a tuple of distinct vertices of $\str{B}_1$.
Here, $\mathrm{support}(v)$ is defined as follows:
\begin{enumerate}
\item If $v\in B_1$, we put $\mathrm{support}(v)=\cl_{\str{B}_1}(v)$.
\item Otherwise $v\in \mathcal T$ and thus $v=(P,\bar{x})$ for some choice of $P$ and $\bar{x}$.
In this case we put $\mathrm{support}(v)=\cl_{\str{B}_1}(\bar{x})$ (where by $\cl_{\str{B}_1}(\bar{x})$ we mean the smallest substructure of $\str B_1$ containing all vertices from $\bar x$).
\end{enumerate}
Depending on the context, we may consider $\mathrm{support}(v)$ to be a substructure of $\str B_1$ or just a subset of $B_1$.

\begin{lemma}\label{lem:ori_aux}
Let $(w_1,\ldots,w_n) \in \nbrel{\str B_1}{(x,(x_1,\ldots, x_n))}$. There is automorphism $f$ of $\str B_1$ such that $f(\{w_1,\ldots,w_n\}) \subseteq A_1$. If there is also $u\in B_1$ such that $(u, w_1,\ldots,w_n)\in \nbrelE{\str B_1}{(x,(x_1,\ldots, x_n))}$ then $f$ can be chosen so that also $f(u)\subseteq A_1$.

Moreover, whenever $f$ is an automorphism of $\str B_1$ such that $f(\{w_1,\ldots,w_n\}) \subseteq A_1$ and $f'$ is an automorphism of $\str B_0$ such that $f_L$ is induced by $f'$ then the following hold:
\begin{enumerate}
\item $(f(w_1),\ldots,f(w_n)) \in \nbrel{\str A_1}{(f'(x), (f'(x_1),\ldots,f'(x_n)))}$,
\item for every $1\leq i\leq n$ it holds that $f(w_i) = f'(x_i)$, and
\item $f'(\{x,x_1,\ldots,x_n\})\subseteq A_0$,
\item $\roots{\str B_0}{f'(x)} = f'(\{x_1,\ldots,x_n\})$.
\end{enumerate}
If there is also $u\in B_1$ such that $f(u)\in A_1$, then $(u, w_1,\ldots,w_n)\in \nbrelE{\str B_1}{(x,(x_1,\ldots, x_n))}$ if and only if $(f(u),f'(x)) \in E_{\str{A}_0}$.
\end{lemma}
\begin{proof}
The first part is straightforward: Since $(w_1,\ldots,w_n)$ (or $(u, w_1,\ldots,w_n)$ respectively) is in a relation of $\str B_1$, we get that $\cl_{\str{B}_1}(\{w_1,\ldots,w_n\})$ (or $\cl_{\str{B}_1}(\{u, \allowbreak w_1,\ldots,w_n\})$ respectively) is an irreducible substructure of $\str B_1$, and so there is an automorphism $f$ of $\str B_1$ with the desired properties by irreducible structure faithfulness of $\str B_1$.

Suppose now that we have such automorphisms $f$ and $f'$. The first statement is just rephrasing that $f$ is an automorphism with $f_L$ induced by $f'$. From the construction of $\str A_1$ it follows that whenever $(t_1,\ldots,t_n)\in \nbrel{\str A_1}{(y,(y_1,\ldots,y_n))}$, then $t_i=y_i$ for every $1\leq i\leq n$, which implies the second point. The third point is a direct consequence of the second point and the construction of $\str A_1$. To see the fourth point, note that $\str A_0$ is a substructure of $\str B_0$, hence $\roots{\str B_0}{f'(x)} = \roots{\str A_0}{f'(x)} = f'(\{x_1,\ldots,x_n\})$.

If there is also $u\in B_1$ such that $f(u)\subseteq A_1$, then directly from the definition of the relations on $\str A_1$ it follows that $(u, w_1,\ldots,w_n)\in \nbrelE{\str B_1}{(x,(x_1,\ldots, x_n))}$ if and only if $(f(u),f'(x)) \in E_{\str{A}_0}$.
\end{proof}

\begin{observation}
If $v=((x,(x_1,\ldots,x_n)),(w_1,\ldots,w_n))\in \mathcal T$ then $\mathrm{support}(v) = \{w_1,\ldots,w_n\}$.
\end{observation}
\begin{proof}
From the definition of $\mathcal T$, we know that $(w_1,\ldots,w_n)\in \nbrel{\str B_1}{(x,(x_1,\ldots,x_n))}$. So, by Lemma~\ref{lem:ori_aux}, we get automorphisms $f$ and $f'$ such that $f(w_i) = f'(x_i) \in A_1$ for every $i$ and $f'(\{x_1,\ldots,x_n\})$ are the only roots reachable from $f'(x)$ in $\str B_0$. Consequently, they are all the more so the only roots reachable from $f'(\{x_1,\ldots,x_n\}) = f(\{w_1,\ldots,w_n\})$ in $\str A_0$ and hence $\cl_{\str{B}_1}(f(\{w_1,\ldots,w_n\})) = f(\{w_1,\ldots,w_n\})$. Sending it back by $f^{-1}$ then gives $$\mathrm{support}(v)=\cl_{\str{B}_1}(\{w_1,\ldots,w_n\}) = \{w_1,\ldots,w_n\}.$$
\end{proof}

The following observation follows directly from the construction of $\str B$.
\begin{observation}\label{obs:supports}
Whenever $(u,v)\in E_\str{B}$, we have that $\mathrm{support}(v)\subseteq \mathrm{support}(u)$. 
\end{observation}
\begin{proof}
We have to distinguish four cases:
\begin{enumerate}
\item If $u,v\in B_1$, by~\ref{C1} we know that $(u,v)\in E_{\str{B}_1}$. This implies that $v\in \cl_{\str{B}_1}(u)$, so $\cl_{\str{B}_1}(v)\subseteq \cl_{\str{B}_1}(u)$ and hence $\mathrm{support}(v)\subseteq \mathrm{support}(u)$.
\item If $u\in B_1$, $v=((x,(x_1,\ldots, x_n)), (w_1,\ldots,w_n))\in \mathcal T$, by~\ref{C2} we know that $(u, w_1,\ldots,w_n)\in \nbrelE{\str B_1}{(x,(x_1,\ldots, x_n))}$. By definition, $\mathrm{support}(u) = \cl_{\str{B}_1}(u)$ and $\mathrm{support}(v) = \{w_1,\ldots,w_n\}$. Using Lemma~\ref{lem:ori_aux} we get automorphisms $f$ and $f'$ such that $\mathrm{roots}_{\str A_0}(f'(x)) = f'(\{x_1,\ldots, x_n\})$ and $(f(u),f'(x)) \in E_{\str A_0}$. So
$$f(\{w_1,\ldots,w_n\}) = f'(\{x_1,\ldots, x_n\}) \subseteq F_{\str A_0}(f'(x)) \subseteq F_{\str A_0}(f(u)).$$ Consequently, $\{w_1,\ldots, w_n\} \subseteq F_{\str A_1}(u)$, and hence $$\mathrm{support}(v) = \{w_1,\ldots, w_n\} \subseteq \cl_{\str{B}_1}(u) = \mathrm{support}(u).$$

\item If $u = ((x,(x_1,\ldots, x_n)), (w_1,\ldots,w_n))\in \mathcal T$ and $v\in B_1$, by~\ref{C3} we have $1\leq i\leq n$ such that $v=w_i$. Then $$\support{v} = \cl_{\str B_1}(w_i) \subseteq \cl_{\str B_1}(\{w_1,\allowbreak \ldots,w_n\}) = \support{u}.$$

\item If $u = ((x,(x_1,\ldots, x_n)), (w_1,\ldots,w_n))\in \mathcal T$ and $v = ((y,(y_1,\ldots, y_m)), (t_1,\allowbreak \ldots,t_m))\in \mathcal T$, by~\ref{C4} we get immediately that $\mathrm{support}(v) = \{t_1,\ldots,t_m\}\subseteq \{w_1,\ldots,w_n\} = \mathrm{support}(u)$.
\end{enumerate}
\end{proof}

Our next goal is to show that $\str B\in \mathcal D^k_d$. Towards that direction we define the following procedure to map portions of $\str{B}$ to substructures of $\str{A}$.
Given a vertex $v\in B$ and an automorphism $f=(f_L,f_{B_1})$ of $\str{B}_1$ such that $f(\mathrm{support}(v))$ is a substructure of $\str{A}_1$ we define \emph{$f$-correspondence} $c_f(v)\in A$ as follows:
\begin{enumerate}
 \item If $v\in B_1$, we put $c_f(v)=f_{B_1}(v)$.
\item Otherwise $v\in \mathcal T$. Then $v=(P,\bar{x})$ (for some choice of $P$ and $\bar{x}$) and we put $c_f(v)=\pi(f_L(P))$. (Here, by $f_L(P)$ we mean the $P'$ such that $f_L(\rel{}{P}) = \rel{}{P'}$.)
\end{enumerate}

\begin{claim}[on correspondence]\label{cl:correspondence}
Let $f=(f_L,f_{B_1})$ be an automorphism of $\str{B}_1$ and $v\neq v'\in B$ such that both $f(\mathrm{support}(v))$ and $f(\mathrm{support}(v'))$ are substructures of $\str{A}_1$. Then
\begin{enumerate}[label=P\arabic*]
 \item\label{P1} $c_f(v)\neq c_f(v')$.
 \item\label{P2} $(v,v')\in E_\str{B}$ if and only if $(c_f(v), c_f(v'))\in E_\str{A}$.
\end{enumerate}
\end{claim}
\begin{proof}
Let $f'$ be an automorphism of $\str B_0$ inducing $f_L$. If $v,v'\in B_1$ then we know that $c_f(v)=f(v)\neq
f(v')=c_f(v')$ and~\ref{P1} follows.  If precisely one of $v$, $v'$ is
in $v\in B_1$ then it follows that precisely one of $c_f(v),c_f(v')$ is a root of
$\str{A}$ and~\ref{P1} follows as well. 

So $v = (P,\bar{x})\in \mathcal T$ and $v'=(P',\bar{x}')\in \mathcal T$. We will show that $\pi(P)\neq \pi(P')$, which would imply that $c_f(v)=f'(\pi(P))\neq f'(\pi(P'))=c_f(v')$, hence~\ref{P1} holds. For a contradiction, suppose that $\pi(P) = \pi(P')$. By the construction we have that $\bar{x} \in\nbrel{\str B_1}{P}$ and $\bar{x}' \in\nbrel{\str B_1}{P'}$ and Lemma~\ref{lem:ori_aux} then implies that $f(\bar{x}) = \roots{\str B_0}{f'(\pi(P))} = \roots{\str B_0}{f'(\pi(P'))} = f(\bar{x}')$ (where we consider $\bar{x}$ and $\bar{x}'$ as sets), hence $\bar{x}$ and $\bar{x}'$ are different permutations of the same set. This is, however, in a contradiction with the definition of $\mathcal T$ and the fact that $v\neq v'$, which finishes the proof of~\ref{P1}.

\medskip

If $v,v'\in B_1$,~\ref{P2} immediately follows from~\ref{C1}. If $v\in B_1$ and $v'=((x,(x_1,\allowbreak \ldots, x_n)), (w_1,\ldots,w_n))\in \mathcal T$, we know by~\ref{C2} that $(v,v')\in E_\str{B}$ if and only if $(v, w_1,\ldots,w_n)\in \nbrelE{\str B_1}{(x,(x_1,\ldots, x_n))}$. By Lemma~\ref{lem:ori_aux} we know that this happens if and only if $(c_f(v), c_f(v')) = (f(v),f'(x)) \in E_{\str A_0} = E_\str A$.

Now suppose that $v=((x,(x_1,\ldots, x_n)), (w_1,\ldots,w_n))\in \mathcal T$ and $v'\in B_1$. If there is $1\leq i\leq n$ such that $v' = w_i$, we know that $(v,v')\in E_\str{B}$ if and only if $(x,x_i)\in E_{\str B_0}$ by~\ref{C3}. Lemma~\ref{lem:ori_aux} tells us that $f'(x_i) = f(w_i) = f(v')$, and since $f'$ is an automorphism of $\str B_0$, we get that $(x,x_i)\in E_{\str B_0}$ if and only if $(c_f(v),c_f(v')) = (f'(x),f'(x_i)) \in E_{\str A_0} = E_\str A$. So $v' \neq w_i$ for any $i$. In that case $(v,v')\notin E_\str B$ by~\ref{C3}. But, again using Lemma~\ref{lem:ori_aux}, we know that $\roots{\str B_0}{f'(x)} = f(\{w_1,\ldots,w_n\})$, and as $f(v')$ is a root of $\str B_0$, we know that $f(v')\notin \roots{\str B_0}{f'(x)}$, so in particular $(c_f(v),c_f(v')) = (f'(x),f(v'))\notin E_{\str A}$.


Finally, suppose that $v = ((x,\bar{x}), \bar{w})\in \mathcal T$ and $v' = ((y,\allowbreak \bar{y}), \bar{t})\in \mathcal T$. If $(c_f(v),\allowbreak c_f(v')) = (f'(x),f'(y)) \in E_\str A$, then we know that (as sets) 
$$f(\bar{t}) = f'(\bar{y}) = \roots{\str A}{f'(y)} \subseteq \roots{\str A}{f'(x)} = \allowbreak f'(\bar{x}) = f(\bar{w}),$$
so $\bar{t}\subseteq \bar{w}$ (as sets) and hence $(v,v')\in E_\str B$ by~\ref{C4}. So $(c_f(v),c_f(v')) = (f'(x),f'(y)) \notin E_\str A$. But then, as $f'$ is an automorphism, we get that $(x,y)\notin E_\str A$ and thus $(v,v')\notin E_\str B$ by~\ref{C4}.
\end{proof}

\begin{corollary}\label{cor:preserves_paths}
Let $v_1,\ldots,v_m$ be a sequence of vertices of $\str B$ such that $(v_i,v_{i+1}) \in E_\str B$ for every $1\leq i < m$ and let $f$ be an automorphism of $\str B_1$ such that $\support{v_1}\subseteq \str A_1$. Then $c_f(v_i)\in A$ for every $1\leq i\leq m$ and $(c_f(v_i),c_f(v_{i+1}))\in E_\str A$ for every $1\leq i < m$. Moreover, such an automorphism always exists.
\end{corollary}
\begin{proof}
Put $\str S = \support{v_1}$. First note that $\str S$ is an irreducible substructure of $\str B_1$ (it is either the closure of a vertex, or a tuple in a relation), so irreducible structure faithfulness of $\str B_1$ gives the moreover part. By Observation~\ref{obs:supports} we know that $\support{v_i} \subseteq \str S$ for every $1\leq i\leq m$. Hence $c_f(v_i)$ is defined for every $1\leq i \leq m$ and an application of Claim~\ref{cl:correspondence} finishes the proof.
\end{proof}

\begin{claim}\label{cl:neighbours}
Let $v\in B$ and let $f$ be an automorphism of $\str B_1$ with $f(\mathrm{support}(v)) \subseteq \str A_1$. Then $c_f$ restricts to a bijection between the out-neighbours of $v$ in $\str B$ and the out-neighbours of $c_f(v)$ in $\str A$.
\end{claim}
\begin{proof}
Pick an arbitrary $v'\in B$ such that $(v,v')\in E_\str B$. By Corollary~\ref{cor:preserves_paths} we get that $(c_f(v),c_f(v'))\in E_\str A$ and moreover if $v''\in B$ is a different out-neighbour of $v$ then by Claim~\ref{cl:correspondence} we know that $c_f(v') \neq c_f(v'')$. So $c_f$ restricts to an injective function from the out-neighbours of $v$ in $\str B$ to the out-neighbours of $c_f(v)$ in $\str A$.



To prove that it is surjective, pick an arbitrary $y\in A$ such that $(c_f(v),y)\in E_\str A$. We will find $y'\in \str B$ such that $(v,y')\in E_\str B$ and $c_f(y')=y$. If $v\in B_1$, we know that $c_f(v) = f(v)$ and hence $c_f(v)$ is a root of $\str A$. If $y$ is also a root of $\str A$, it follows that $(v,f^{-1}(y))\in E_{\str B_1}$ and so $(v,f^{-1}(y)) \in E_\str B$ by~\ref{C1}, hence we can put $y'=f^{-1}(y)$. If $y$ is a non-root of $\str A$ then, by the construction of $\str A_1$, there is a tuple $\bar{y}=(y_1,\ldots,y_n)$ of vertices of $\str A_1$ such that $\bar{y} \in \nbrel{\str A_1}{(y,\bar{y})}$ and $(f(v),y_1,\ldots,y_n) \in \nbrelE{\str A_1}{(y,\bar{y})}$ (without loss of generality the enumeration of $\bar{y}$ is chosen so that $((y,\bar{y}),\bar{y})\in \mathcal T$). Let $f'$ be an automorphism of $\str B_0$ which induces $f_L$. Putting $y'' = (((f')^{-1}(y),f^{-1}(\bar{y})),f^{-1}(\bar{y})) \in \mathcal T_0$ and picking $y'\in \mathcal T$ such that $y'\sim y''$, we can see that indeed $c_f(y') = y$ and $(v,y')\in E_\str B$ (by~\ref{C2}).

So suppose $v = ((x,(x_1,\ldots, x_n)), (w_1,\ldots,w_n))\in \mathcal T$. Let $f'$ be an automorphism of $\str B_0$ which induces $f_L$. We know that $c_f(v) = f'(x)$ and that $f'(x)$ is a non-root of $\str A$.  If $y$ is a root of $\str A$, we get that $y\in \roots{\str A}{f'(x)}$, and hence, by the construction of $\str A_1$, there is $1\leq i\leq n$ such that $f'(x_i) = y$. By Lemma~\ref{lem:ori_aux} we know that $f(w_i) = f'(x_i) = y$, hence $w_i = f^{-1}(y) \in B_1$. If we put $y' = f^{-1}(y)$, we get that $c_f(y') = y$ and, by~\ref{C3}, $(v,y')\in E_\str B$.

The last case is that $y$ is a non-root of $\str A$. Since $(f'(x),y)\in E_\str A$, we get that $$\roots{\str A}{y} \subseteq \roots{\str A}{f'(x)} = f(\{w_1,\ldots,w_n\}) = f'(\{x_1,\ldots,x_n\}).$$ Let $\bar{y}$ be an enumeration of $\roots{\str A}{y}$. By the construction of $\str A_1$ we get that $\bar{y}\in\nbrel{\str A_1}{(y,\bar{y})}$ and so $f^{-1}(\bar{y})\in\nbrel{\str B_1}{(f'^{-1}(y),f'^{-1}(\bar{y}))}$. Put $y' = ((f'^{-1}(y),f'^{-1}(\bar{y})),f^{-1}(\bar{y}))$ and assume that enumeration of $\bar{y}$ was chosen so that $y'\in\mathcal T$. Then $c_f(y') = y$ and, by~\ref{C4}, $(v,y')\in E_\str B$, which concludes the proof.
\end{proof}

\begin{corollary}\label{cor:roots}
$\str B$ is a $k$-orientation and the roots of $\str B$ are precisely members of $B_1$.
\end{corollary}
\begin{proof}
Pick an arbitrary $v\in \str B$. Since $\support{v}$ is an irreducible substructure of $\str B_1$, there is an automorphism $f$ of $\str B_1$ sending $\support{v}$ to $A_1$. Hence, by Claim~\ref{cl:neighbours}, the out-degree of $v$ in $\str B$ is the same as the out-degree of $c_f(v)$ in $\str A$. Consequently, the out-degree of $v$ in $\str B$ is at most $k$ (hence $\str B$ is a $k$-orientation) and it is less than $k$ if and only if $c_f(v)$ is a root of $\str A$ which happens if and only if $c_f(v) = f(v)$, i.e. if $v\in B_1$.
\end{proof}

\begin{corollary}\label{cor:path_witnesses}
Given $u\in \str B$, an automorphism $f\colon \str B_1\to\str B_1$ sending $\support{u}$ into $A_1$ and a sequence $c_f(u)=v_1,\ldots,v_m$ of vertices of $\str A$ such that $(v_i,v_{i+1}) \in E_\str A$ for every $1\leq i < m$, there is a sequence $u=v_1',\ldots,v_m'$ of vertices of $\str B$ such that $(v_i',v_{i+1}') \in E_\str B$ for every $1\leq i < m$ and $c_f(v_i') = v_i$ for every $1\leq i\leq m$.
\end{corollary}
\begin{proof}
We will prove this by induction on $m$. For $m=1$ the statement is trivial. For the induction step, suppose that the statement is true for $m-1$. By the induction hypothesis we have a sequence $v_1',\ldots,v_{m-1}'$ of vertices of $\str B$ such that $(v_i',v_{i+1}') \in E_\str B$ for every $1\leq i < m-1$ and $c_f(v_i') = v_i$ for every $1\leq i\leq m-1$. By Observation~\ref{obs:supports} we know that $\support{v_i'}\subseteq \support{u}$ for every $1\leq i\leq m-1$, hence Claim~\ref{cl:neighbours} for $v_{m-1}'$ tells us that $c_f$ is a bijection between the out-neighbours of $v_{m-1}'$ and $c_f(v_{m-1}') = v_{m-1}$. Therefore in particular, there is some $v_m'\in B$ such that $c_f(v_m') = v_m$ and $(v_{m-1}',v_m')\in E_\str B$ which concludes the proof.
\end{proof}

\begin{claim}\label{cl:supports_roots}
For every $u\in \str B$ it holds that $\mathrm{roots}_{\str B}(u) = \mathrm{support}(u)$.
\end{claim}
\begin{proof}
First we prove that $\mathrm{roots}_{\str B}(u) \subseteq \mathrm{support}(u)$. Pick an arbitrary $v\in \mathrm{roots}_{\str B}(u)$. This means that $v\in B_1$ (by Corollary~\ref{cor:roots}) and that there is a sequence $u=v_1,\ldots,v_m=v$ of vertices of $\str B$ such that $(v_i,v_{i+1}) \in E_\str B$ for every $1\leq i < m$. By Corollary~\ref{cor:preserves_paths} we get an automorphism $f$ of $\str B_1$ such that $c_f(v_i)\in A$ for every $1\leq i\leq m$ and $(c_f(v_i),c_f(v_{i+1}))\in E_\str A$ for every $1\leq i < m$. Consequently, $c_f(v) = f(v) \in \roots{\str A}{c_f(u)}$ and so, by the construction of $\str B_1$ and $\str B$, $v\in \support{u}$.

To see that $\mathrm{roots}_{\str B}(u) \supseteq \mathrm{support}(u)$, pick an arbitrary $v\in \support{u}$ and let $f$ be an automorphism of $\str B$ sending $\support{u}$ to $A_1$. By the definition of $\str B_1$ and $\support{u}$ this means that that $f(v)\in\roots{\str A}{c_f(u)}$, so in particular $c_f(v)=f(v)\in B_1$ and there is a sequence $c_f(u)=v_1,\ldots,v_m=c_f(v)$ of vertices of $\str A$ such that $(v_i,v_{i+1}) \in E_\str A$ for every $1\leq i < m$. Using Corollary~\ref{cor:path_witnesses} we get a sequence $v_1',\ldots,v_m'$ of vertices of $\str B$ such that $(v_i',v_{i+1}') \in E_\str B$ for every $1\leq i < m$ and $c_f(v_i') = v_i$ for every $1\leq i\leq m$. In particular, $v_1' = u$ and $v_m' = v\in B_1$, hence $v\in\mathrm{roots}_{\str B}(u)$ (by Corollary~\ref{cor:roots}).
\end{proof}

\begin{claim}\label{cl:substructures}
Let $\str{D}_1$ be a substructure of $\str{B}_1$ and $f$ an automorphism
of $\str{B}_1$ such that $f(\str{D}_1)$ is a substructure of $\str{A}_1$.  Put
$$D = \{v \in B : \mathrm{support}(v)\subseteq D_1\}.$$
Then $\str B$ induces a substructure $\str D$ on $D$ and $c_f$ is an isomorphism from $\str{D}$ to a substructure induced by $\str{A}$ on $\cl_{\str{A}}(f(D_1))$.
\end{claim}
\begin{proof}
Since $\str D_1$ is a substructure of $\str B_1$, it follows that $D_1\subseteq D$. Let $u\in D$ and $v\in B$ be vertices such that $(u,v)\in E_{\str B}$. From Observation~\ref{obs:supports} if follows that $\mathrm{support}(v)\subseteq \mathrm{support}(u) \subseteq D_1$ and so $v\in D$. This means that there are no outgoing edges from $D$ in $\str B$ and thus $D$ is closed on the function $\func{B}{}$. By Claim~\ref{cl:supports_roots}, $D$ is also closed on all functions $\func{B}{n}$, hence indeed $\str B$ induces a substructure $\str D$ on $D$.

Next we will prove that $c_f$ is a bijection $D\to \cl_{\str{A}}(f(D_1))$. Clearly, $\dom(c_f) \supseteq D$. Fix $v\in D$. If $v\in D_1$ then $c_f(v) = f(v)\in f(D_1)$. Conversely, all roots in $\cl_{\str{A}}(f(D_1))$ are from $f(D_1)$, because $f(\str D_1) \subseteq \str A_1$. 

So suppose $v = ((x,(x_1,\ldots, x_n)), (w_1,\ldots,w_n))\in \mathcal T$ with $\{w_1,\ldots,w_n\} \subseteq D_1$. Let $f'$ be an automorphism of $\str B_0$ which induces $f_L$. By Lemma~\ref{lem:ori_aux} we get that $f(w_i) = f'(x_i)$ for every $1\leq i\leq n$ and $\roots{\str A}{f'(x)} = f'(\{x_1,\ldots,x_n\})$. Since $c_f(v) = f'(x)$, it follows that $\roots{\str A}{c_f(v)} = f(\{w_1,\ldots,w_n\}) \subseteq f(D_1)$, hence indeed $c_f(v)\in \cl_{\str{A}}(f(D_1))$. Conversely, let $v$ be a non-root vertex of $\cl_{\str{A}}(f(D_1))$ and let $\bar{y} = \roots{\str A}{v}$ be an arbitrary enumeration. We know that $\bar{y} \subseteq f(D_1)$ and by the construction we also get that $\bar{y} \in \nbrel{\str A_1}{(v,\bar y)}$. Hence we can reconstruct $v'' = (((f')^{-1}(v),f^{-1}(\bar{y})),f^{-1}(\bar{y})) \in \mathcal T_0$ and $v'\in \mathcal T$ with $v'\sim v''$ such that $c_f(v') = v$. So indeed $\range(c_f) = \cl_{\str{A}}(f(D_1))$. From~\ref{P1} of Claim~\ref{cl:correspondence} we get that $c_f$ is a bijection $D\to \cl_{\str{A}}(f(D_1))$.

Finally, from~\ref{P2} of Claim~\ref{cl:correspondence} it follows that for every $u,v\in D$ we have $(u,v)\in E_\str D$ if and only if $(c_f(u),c_f(v))\in E_\str A$. Since all functions in $\str B$ and $\str A$ are defined from the graph structure, it follows that $c_f$ indeed is an isomorphism $\str{D} \to \cl_\str{A}(f(D_1))$.
\end{proof}

\begin{lemma}\label{lem:orientations:irreducibles}
Let $\str D\subseteq \str B$ be irreducible. Then $\str B_1$ induces an irreducible substructure on $D\cap B_1$.
\end{lemma}
\begin{proof}
First note that if $\str D$ is a substructure of $\str B$, then $\str B_1$ induces a substructure on $D\cap B_1$. Indeed, by Corollary~\ref{cor:roots} we know that $D\cap B_1$ are precisely the root vertices of $\str D$. If there is $v\in \cl_{\str B_1}(D\cap B_1)\setminus D$ then by Corollary~\ref{cor:roots} it is a root vertex of $\str B$. This means that there is $u\in D\cap B_1$ such that $v\in \cl_{\str B_1}(u) =  \mathrm{support}(u) = \mathrm{roots}_{\str B}(u)$ (by Claim~\ref{cl:supports_roots}). But this implies that $v\in \cl_{\str B}(u)$, hence $v\in \str D$, a contradiction.

Put $D_1 = D\cap B_1$ and let $\str D_1$ be the substructure induced by $\str B_1$ on $D_1$. We know that $D_1$ are precisely the root vertices of $\str D$. From the definition of $\str B$ and Claim~\ref{cl:supports_roots} it follows that
$$D = \{v \in B : \mathrm{support}(v)\subseteq D_1\} = \{v \in B : \mathrm{roots}_{\str B}(v)\subseteq D_1\}.$$






We will now prove that if $\str D_1$ is reducible then $\str D$ is also reducible. Taking the contrapositive then proves the statement of this claim. Suppose that there are substructures $\str D_1^b, \str D_1^l, \str D_1^r \subseteq \str D_1$ such that $\str D_1$ is the free amalgamation of $\str D_1^l$ and $\str D_1^r$ over $\str D_1^b$ (in particular, $\str D_1^b \subseteq \str D_1^l$ and $\str D_1^b \subseteq \str D_1^r$). Put
\begin{align*}
D^l &= \{v \in B : \mathrm{support}(v)\subseteq D_1^l\},\\
D^r &= \{v \in B : \mathrm{support}(v)\subseteq D_1^r\},\\
D^b &= \{v \in B : \mathrm{support}(v)\subseteq D_1^b\}
\end{align*}
and let $\str D^l$, $\str D^r$ and $\str D^b$ be the substructures of $\str B$ induced on $D^l$, $D^r$ and $D^b$ respectively. Clearly, $\str D^l, \str D^r, \str D^b \subseteq \str D$, $\str D^b\subseteq \str D^l$ and $\str D^b\subseteq \str D^r$. We will prove that $\str D$ is the free amalgamation of $\str D^l$ and $\str D^r$ over $\str D^b$.

Since both $\str D^l$ and $\str D^r$ are substructures of $\str B$, they are in particular closed on functions $F$ and $\func{}{n}$. If there were vertices $u\in D^l$, $v\in D^r$ such that $(u,v)\in E_\str D$, then $v\in F_\str D(u)$, which is a contradiction. Hence there are no edges spanning vertices of both $\str D^l$ and $\str D^r$ at the same time.

If $v\in D^b$ then, as $\str D^b$ is a substructure, it follows that $\mathrm{roots}_{\str B}(v) \subseteq D^b$, similarly for $\str D^l$ and $\str D^r$. It follows that whenever $\bar x$ contains vertices from both $D^l\setminus D^b$ and $D^r\setminus D^b$ then $\func{D}{\lvert \bar x\rvert}(\bar x) = \emptyset$. Consequently, $\str D$ is the free amalgamation of $\str D^l$ and $\str D^r$ over $\str D^b$.
\end{proof}

\begin{corollary}\label{cor:binddk}
$\str{B}\in \mathcal D_d^k$.
\end{corollary}
\begin{proof}
Let $\str D$ be an irreducible substructure of $\str B$. Since $\mathcal D_d^k$ is a free amalgamation class, it suffices to prove that $\str D\in\mathcal D_d^k$. By Lemma~\ref{lem:orientations:irreducibles} we know that $\str B_1$ induces an irreducible substructure on $D\cap B_1$ and by Claim~\ref{cl:supports_roots} this substructure is non-empty. Now we can use irreducible structure faithfulness of $\str B_1$ and Claim~\ref{cl:substructures} to get an embedding $c_f\colon \str D\to\str A$. As $\str A\in\mathcal D_d^k$, we get that $\str D\in\mathcal D_d^k$, hence indeed $\str{B}\in \mathcal D_d^k$.
\end{proof}

\begin{lemma}\label{lem:orientation_automorphism}
Let $f = (f_L,f_{B_1})$ be an automorphism of $\str B_1$ and let $f'$ be an automorphism of $\str B_0$ which induces $f_L$. Let $\iota_0$ be the map on $\mathcal T_0$ given by $\iota_0((x, \bar{x}), \bar{w}) = ((f'(x), f'(\bar{x})), f_{B_1}(\bar{w}))$ and let $\iota$ bet the map induced by $\iota_0$ on $\mathcal T$ (it is easy to see that $\sim$ is a congruence with respect to $\iota_0$). Define $\theta = f_{B_1} \cup \iota$. Then $\theta$ is an automorphism of $\str B$.
\end{lemma}
\begin{proof}
Since $f'$ is an automorphism of $\str B_0$, it follows that $\iota_0$ is a bijection $\mathcal T_0 \to \mathcal T_0$. Consequently, $\iota$ is a bijection $\mathcal T \to \mathcal T$. As $f_{B_1}$ is a bijection $B_1\to B_1$ and $B$ is the disjoint union of $B_1$ and $\mathcal T$, it follows that $\theta$ is a bijection $B\to B$. By Corollary~\ref{cor:binddk} we know that the functions $\func{B}{}$ and $\func{B}{n}$ are defined in $\str B$ from the graph structure. To see that $\theta$ is an automorphism of $\str B$, it remains to prove that for every $u,v\in B$ we have $(u,v)\in E_\str B$ if and only if $(\theta(u),\theta(v))\in E_\str B$. We will distinguish four cases.

\begin{enumerate}
\item First suppose that $u,v\in B_1$. By~\ref{C1} we know that $(u,v)\in E_\str{B}$ if and only if $(u,v)\in E_{\str{B}_1}$. Since $f$ is an automorphism of $\str B_1$, we know that $(u,v)\in E_{\str{B}_1}$ if and only if $(\theta(u),\theta(v)) = (f_{B_1}(u),f_{B_1}(v))\in E_{\str{B}_1}$, so indeed $(u,v)\in E_\str B$ if and only if $(\theta(u),\theta(v))\in E_\str B$.
\item If $u\in B_1$ and $v=((x,\bar{x}), (w_1,\ldots,w_n))\in \mathcal T$, we know by~\ref{C2} that $(u,v)\in E_\str B$ if and only if $(u, w_1,\ldots,w_n)\in \nbrelE{\str B_1}{(x,\bar{x})}$. Again, since $f$ is an automorphism of $\str B_1$, we get that 
$$(u, w_1,\ldots,w_n)\in \nbrelE{\str B_1}{(x,\bar{x})}$$
if and only if
$$f_{B_1}((u, w_1,\ldots,w_n)) \in \nbrelE{\str B_1}{(f'(x),f'(\bar{x}))}.$$
As $\theta(u) = f_{B_1}(u)$ and $\theta(v) = ((f'(x),f'(\bar{x})), f_{B_1}((w_1,\ldots,w_n)))$, we get from~\ref{C2} that this happens if and only if $(\theta(u),\theta(v)) \in E_\str{B}$.
\item If $u = ((x,(x_1,\ldots, x_n)), (w_1,\ldots,w_n))\in \mathcal T$ and $v\in B_1$, we know by~\ref{C3} that $(u,v)\in E_\str B$ if and only if there is $1\leq i\leq n$ such that $v=w_i$ and $(x,x_i)\in E_{\str B_0}$. This is equivalent to $f_{B_1}(v) = f_{B_1}(w_i)$ and $(f'(x),f'(x_i)) \in E_{\str B_0}$ which is in turn  once again equivalent to $(\theta(u),\theta(v)) \in E_\str{B}$.
\item Finally, if $u = ((x,\bar{x}), \bar{w})\in \mathcal T$, $v = ((y,\bar{y}), \bar{t})\in \mathcal T$ then (by~\ref{C4}) $(u,v)\in E_\str B$ if and only if $\bar{t}\subseteq \bar{w}$ (as sets) and $(x,y)\in E_{\str B_0}$. This is equivalent to $f_{B_1}(\bar{t}) \subseteq f_{B_1}(\bar{w})$ (as sets) and $(f'(x),f'(y))\in E_{\str B_0}$, or in other words, $(\theta(u),\theta(v)) \in E_\str{B}$.
\end{enumerate}
\end{proof}

Note that $\cl_\str A(A_1) = A$. Therefore, Claim~\ref{cl:substructures} for $\str D_1 = \str A_1$ and the identity automorphism gives us an isomorphism $c_\id$. Put $\psi = c_\id^{-1}\colon\str A\to\str B$ and denote $\str A' = \psi(\str A)$. This will be the copy of $\str A$ in $\str B$ whose automorphisms we are going to extend. Note that for every $v\in A_1$ we have that $\psi(v) = v$ and for every $v\in A\setminus A_1$ we have that $\pi(\psi(v)) = v$. It follows that $A_1\subseteq A'$ and $$A' = \{v\in B : \support{v} \subseteq A_1\}.$$

\begin{prop}\label{prop:biswitness}
$\str B$ is an irreducible structure faithful coherent EPPA-witness for $\str A'$.
\end{prop}
\begin{proof}
First we refine the proof of Corollary~\ref{cor:binddk} to prove that $\str B$ is irreducible structure faithful. Let $\str D$ be an irreducible substructure of $\str B$. By Lemma~\ref{lem:orientations:irreducibles} we know that $\str B_1$ induces an irreducible substructure $\str D_1$ on $D\cap B_1$. Now we can use irreducible structure faithfulness of $\str B_1$ to get an automorphism $f=(f_L, f_{B_1})$ of $\str B_1$ such that $f(D_1)\subseteq A_1$. Let $f'$ be an automorphism of $\str B_0$ inducing $f_L$.

Use Lemma~\ref{lem:orientation_automorphism} to construct an automorphism $\theta$ of $\str B$. Clearly, $\theta(D_1)\subseteq A_1\subseteq A'$. Hence it remains to prove that $\theta(D\setminus D_1) \subseteq A'$ (where $D\setminus D_1$ are precisely the non-root vertices of $\str D$). Pick an arbitrary $v\in D\setminus D_1$. We know that $\support{v}\subseteq D_1$. By definition of $\theta$ we also know that $\support{\theta(v)} = f_{B_1}(\support{v}) \subseteq A_1$, and so $\theta(v)\in A'$ and thus indeed $\theta(D)\subseteq A'$.

To see that $\str{B}$ is an EPPA-witness for $\str{A}'$, let $\varphi$ be a partial automorphism
of $\str{A}'$. Consider $\varphi_0 = \psi^{-1}\circ \varphi \circ \psi$ as a partial automorphism of $\str{A}_0$ (note that $\varphi_0(v) = \varphi(v)$ for every $v\in A_1$) and extend
it to an automorphism $\widetilde{\varphi_0}$ of $\str{B}_0$. Let $\phi_L \in \Gamma\!_{L^+}$ be the permutation of $L^+$ given by $\widetilde{\varphi_0}$ and put $\phi_{A_1} = \varphi_0\restriction_{A_1} = \varphi\restriction_{A_1}$. Note that $\phi_{A_1}$ is a bijection $A_1\to A_1$, because $\varphi_0$ preserves whether a vertex is a root or not. Put $\phi = (\phi_L, \phi_{A_1})$. It is easy to verify that $\phi$ is a partial automorphism of $\str A_1$: It preserves the $L_s$ relations and functions, because $\varphi_0$ does. Suppose that $\bar{w} \subseteq \dom(\phi)$ and $\bar{w} \in \nbrel{\str A_1}{(x,\bar{w})}$, then this means that, in $\str A$, there is a vertex $x$ such that $\roots{\str A}{x} = \bar{w}$. This implies that $x \in \func{A}{n}(\bar{w})$ and hence $x\in \dom(\varphi)$. Consequently, $\phi_L(\rel{}{(x,\bar{w})}) = \rel{}{(\varphi(x),\varphi(\bar{w}))}$ and so indeed $\phi_{A_1}(\bar{w}) \in \permnbrel{\phi_L}{\str A_1}{(x,\bar{w})}$.

Let $\widetilde{\phi}\colon \str B_1\to\str B_1$ be the extension of $\phi$ and use Lemma~\ref{lem:orientation_automorphism} for $\widetilde{\varphi_0}$ (as $f'$) and $\widetilde{\phi}$ (as $f$) to get an automorphism $\widetilde{\varphi}$ of $\str B$ (called $\theta$ in the statement of Lemma~\ref{lem:orientation_automorphism}). By the construction, we know that $\varphi\restriction_{A_1} \subseteq \widetilde{\varphi}$. In order to prove that $\widetilde{\varphi}$ extends $\varphi$ it thus remains to argue that $\widetilde{\varphi}(v) = \varphi(v)$ for every $v\in \dom(\varphi)\cap \mathcal T$.

Pick an arbitrary $v = ((x,\bar{x}), \bar{w})\in \dom(\varphi)\cap \mathcal T$. Since $v\in \dom(\varphi)$, we know that $v\in A'$, hence $\bar{w} = \support{v} \subseteq A_1$, and consequently $\bar{x} = \bar{w}$ (as tuples). By the construction we know that (up to applying the same permutation on $\bar{x}$ and $\bar{w}$ to pick the correct member of the $\sim$-equivalence class), 
$$\widetilde{\varphi}(v) = ((\widetilde{\varphi_0}(x),\widetilde{\varphi_0}(\bar{x})), \widetilde{\phi}(\bar{w})),$$
which is equal to
$$((\varphi_0(x),\varphi_0(\bar{x})), \varphi(\bar{w})) = ((\varphi_0(x),\varphi(\bar{w})), \varphi(\bar{w})).$$

In particular, $\pi(\widetilde{\varphi}(v)) = \varphi_0(x)$. By the construction we know that $v=\psi(x)$ and consequently $\pi(v) = x$. By the definition of $\varphi_0$ we know that $\psi\circ\varphi_0 = \varphi\circ\psi$, so in particular $\varphi_0(x) = \pi(\psi(\varphi_0(x))) = \pi(\varphi(\psi(x)))$. So indeed, $\pi(\widetilde{\varphi}(v)) = \pi(\varphi(v))$.

We know that $\roots{\str B}{v} = \bar{w}$ (as sets), hence $\bar{w}\subseteq \dom(\varphi)$. Since $\varphi$ is a partial automorphism, $\varphi(\roots{\str B}{v}) = \roots{\str B}{\varphi(v)}$. But $\roots{\str B}{\varphi(v)} = \support{\varphi(v)}$ (by Claim~\ref{cl:supports_roots}) and we know that $\support{\varphi(v)} = \varphi(\bar{w})$. It follows that $\widetilde{\varphi}(v) = \varphi(v)$ and hence $\widetilde{\varphi}$ indeed extends $\varphi$, which concludes the argument.

Since both $\widetilde{\varphi_0}$ and $\widetilde{\phi}$ can be chosen coherently, it follows that $\widetilde{\varphi}$ is also coherent and hence $\str{B}$ is a coherent EPPA-witness for $\str A'$.
\end{proof}

\begin{proof}[Proof of Theorem~\ref{thm:DF}]
In this section we constructed, for an arbitrary $\str A\in\mathcal D_d^k$ a structure $\str B$. By Corollary~\ref{cor:binddk}, $\str{B}\in \mathcal D_d^k$ and by Proposition~\ref{prop:biswitness} $\str B$ is an irreducible structure faithful coherent EPPA-witness for an isomorphic copy of $\str A$. Hence indeed $\mathcal D_d^k$ has irreducible structure faithful coherent EPPA.
\end{proof}

\section{Conclusion}
\label{sec:conclusion}
Comparing known EPPA classes and known Ramsey classes one can easily identify two main weaknesses of the state-of-the-art EPPA constructions.
\begin{enumerate}
 \item The need for automorphism-preserving completion procedure is not necessary in the Ramsey context.  The example of two-graphs~\cite{eppatwographs} shows that there are classes with EPPA which do not admit automorphism-preserving completions (see~\cite{Konecny2019a} for a more systematic treatment of certain classes of this kind). Understanding the situation better might lead to solutions of some of the long standing open problems in this area including the question whether the class of all finite tournaments has EPPA (see~\cite{Sabok}, \cite{HubickaSemigenericAMUC} and~\cite{HubickaSemigeneric} for recent progress on this problem).
 \item There is a lack of general EPPA constructions for classes with non-unary function symbols.  Again, there are known classes with non-unary function symbols that have EPPA (e.g. finite groups or classes from Section~\ref{sec:nonunary}).
 It is however not known whether, for example, the class of all finite partial Steiner systems or the class of all finite equivalences on unordered pairs with two equivalence classes have EPPA.
\end{enumerate}
On the other hand, in this paper we consider $\GammaL$-structures which, in the finite language case, reduce to the usual model-theoretic
structures in the Ramsey context, because the action of $\GammaL$ must be trivial in order for the class to be rigid. This has some additional applications including:
\begin{enumerate}
 \item Elimination of imaginaries for classes having definable equivalence classes (see~\cite{Ivanov2015,Hubicka2016}),
 \item representation of special non-unary functions which map vertices of one type to vertices of different type (see Section~\ref{sec:nonunary} or Theorem~\ref{thm:DF}), or
 \item representation of antipodal structures and switching classes (\cite{eppatwographs,Konecny2019a}).
\end{enumerate}

We refer the reader to~\cite{Hubicka2016,Aranda2017,Konecny2018bc,Konecny2018b,Hubicka2017sauer} for various examples of (automorphism-preserving) locally finite subclasses. 

\medskip

One of the main weaknesses of Theorems~\ref{thm:nreppa}, \ref{thm:main} and \ref{thm:mainstrong} is that they only allow unary functions. It would be interesting to know whether they hold without this restriction.
\begin{question}~\label{q:nonunary}
Do Theorems~\ref{thm:nreppa}, \ref{thm:main} and \ref{thm:mainstrong} hold for languages with non-unary functions?
\end{question}
A positive answer to Question~\ref{q:nonunary} would have some applications which are interesting on their own and have been asked before. We present two of them as separate questions.
\begin{question}
Let $L$ be the language consisting of a single binary function and let $\mathcal C$ be the class of all finite $L$-structures (say, such that the image of every pair of vertices has cardinality at most one). Does $\mathcal C$ have EPPA?
\end{question}

\begin{question}
Does the class of all finite partial Steiner triple systems have EPPA, where one only wants to extend partial automorphism between closed substructures? (A sub-hypergraph $H$ of a Steiner triple system $S$ is \emph{closed} if whenever $\{x,y,z\}$ is a triple of $S$ and $x,y\in H$, then $z\in H$.)
\end{question}



\section*{Acknowledgment}
We would like to thank Andy Zucker, the anonymous referee, and David M. Evans for their helpful comments which significantly improved the quality of this paper.

\bibliography{ramsey.bib}

\newcommand{\etalchar}[1]{$^{#1}$}
\begin{thebibliography}{ABWH{\etalchar{+}}17b}

\bibitem[ABWH{\etalchar{+}}17a]{Aranda2017a}
Andres Aranda, David Bradley-Williams, Eng~Keat Hng, Jan Hubi{\v c}ka,
  Miltiadis Karamanlis, Michael Kompatscher, Mat{\v e}j Kone{\v c}n{\'y}, and
  Micheal Pawliuk.
\newblock Completing graphs to metric spaces.
\newblock {\em Electronic Notes in Discrete Mathematics}, 61:53--60, 2017.
\newblock The European Conference on Combinatorics, Graph Theory and
  Applications (EUROCOMB'17).

\bibitem[ABWH{\etalchar{+}}17b]{Aranda2017c}
Andres Aranda, David Bradley-Williams, Eng~Keat Hng, Jan Hubi{\v c}ka,
  Miltiadis Karamanlis, Michael Kompatscher, Mat{\v e}j Kone{\v c}n{\'y}, and
  Micheal Pawliuk.
\newblock Completing graphs to metric spaces.
\newblock arXiv:1706.00295, accepted to Contributions to Discrete Mathematics,
  2017.

\bibitem[ABWH{\etalchar{+}}17c]{Aranda2017}
Andres Aranda, David Bradley-Williams, Jan Hubi{\v c}ka, Miltiadis Karamanlis,
  Michael Kompatscher, Mat{\v e}j Kone{\v c}n{\'y}, and Micheal Pawliuk.
\newblock Ramsey expansions of metrically homogeneous graphs.
\newblock Submitted, arXiv:1707.02612, 2017.

\bibitem[AN19]{Andreka2019}
H.~Andr\'eka and I.~N\'emeti.
\newblock Extending partial isomorphisms, a small construction.
\newblock private communication, 2019.

\bibitem[BPT11]{Bodirsky2011a}
Manuel Bodirsky, Michael Pinsker, and Todor Tsankov.
\newblock Decidability of definability.
\newblock In {\em Proceedings of the 2011 IEEE 26th Annual Symposium on Logic
  in Computer Science}, LICS '11, pages 321--328, Washington, DC, USA, 2011.
  IEEE Computer Society.

\bibitem[BWC20]{BradleyEPPA}
David Bradley-Williams and Peter~J. Cameron.
\newblock Notes on eppa witnesses.
\newblock In preparation, 2020.

\bibitem[CH03]{Cherlin2003}
Gregory Cherlin and Ehud Hrushovski.
\newblock {\em Finite Structures with Few Types}.
\newblock Princeton University Press, 2003.

\bibitem[Che17]{Cherlin2013}
Gregory Cherlin.
\newblock Homogeneous ordered graphs and metrically homogeneous graphs.
\newblock Submitted, December 2017.

\bibitem[Con19]{Conant2015}
Gabriel Conant.
\newblock Extending partial isometries of generalized metric spaces.
\newblock {\em Fundamenta Mathematicae}, 244:1--16, 2019.

\bibitem[EHKN20]{eppatwographs}
David~M. Evans, Jan Hubi{\v{c}}ka, Mat{\v {e}}j Kone{\v {c}}n{\'{y}}, and
  Jaroslav Ne\v{s}et\v{r}il.
\newblock E{P}{P}{A} for two-graphs and antipodal metric spaces.
\newblock {\em Proceedings of the American Mathematical Society},
  148:1901--1915, 2020.

\bibitem[EHN17]{Evans3}
David~M. Evans, Jan Hubi{\v c}ka, and Jaroslav Ne{\v{s}}et{\v{r}}il.
\newblock {R}amsey properties and extending partial automorphisms for classes
  of finite structures.
\newblock To appear in Fundamenta Mathematicae, arXiv:1705.02379, 2017.

\bibitem[EHN19]{Evans2}
David~M. Evans, Jan Hubi{\v c}ka, and Jaroslav Ne{\v{s}}et{\v{r}}il.
\newblock Automorphism groups and {R}amsey properties of sparse graphs.
\newblock {\em Proceedings of the London Mathematical Society},
  119(2):515--546, 2019.

\bibitem[Eva]{EvansBonn}
David~M. Evans.
\newblock Homogeneous structures, $\omega$-categoricity and amalgamation
  constructions.
\newblock Notes on a Minicourse given at HIM, Bonn, September 2013.

\bibitem[Fra53]{Fraisse1953}
Roland Fra{\"\i}ss{\'e}.
\newblock Sur certaines relations qui g\'en\'eralisent l'ordre des nombres
  rationnels.
\newblock {\em Comptes Rendus de l'Academie des Sciences}, 237:540--542, 1953.

\bibitem[Hal49]{hall1949}
Marshall Hall.
\newblock Coset representations in free groups.
\newblock {\em Transactions of the American Mathematical Society},
  67(2):421--432, 1949.

\bibitem[Her95]{Herwig1995}
Bernhard Herwig.
\newblock Extending partial isomorphisms on finite structures.
\newblock {\em Combinatorica}, 15(3):365--371, 1995.

\bibitem[Her98]{herwig1998}
Bernhard Herwig.
\newblock Extending partial isomorphisms for the small index property of many
  $\omega$-categorical structures.
\newblock {\em Israel Journal of Mathematics}, 107(1):93--123, 1998.

\bibitem[HHLS93]{hodges1993b}
Wilfrid Hodges, Ian Hodkinson, Daniel Lascar, and Saharon Shelah.
\newblock The small index property for $\omega$-stable $\omega$-categorical
  structures and for the random graph.
\newblock {\em Journal of the London Mathematical Society}, 2(2):204--218,
  1993.

\bibitem[HJKS19a]{HubickaSemigenericAMUC}
Jan Hubi\v{c}ka, Colin Jahel, Mat{\v e}j Kone{\v c}n{\'y}, and Marcin Sabok.
\newblock Extending partial automorphisms of n-partite tournaments.
\newblock {\em Acta Mathematica Universitatis Comenianae}, 88(3), 2019.

\bibitem[HJKS19b]{HubickaSemigeneric}
Jan Hubi\v{c}ka, Colin Jahel, Mat{\v e}j Kone{\v c}n{\'y}, and Marcin Sabok.
\newblock Extension property for partial automorphisms of the $n$-partite and
  semi-generic tournaments.
\newblock To appear, 2019.

\bibitem[HKN18]{Hubicka2017sauer}
Jan Hubi{\v{c}}ka, Mat{\v {e}}j Kone{\v {c}}n{\'{y}}, and Jaroslav
  Ne\v{s}et\v{r}il.
\newblock Semigroup-valued metric spaces: {R}amsey expansions and {E}{P}{P}{A}.
\newblock In preparation, 2018.

\bibitem[HKN19]{Hubicka2018metricEPPA}
Jan Hubi{\v{c}}ka, Mat{\v{e}}j Kone{\v{c}}n{\'y}, and Jaroslav
  Ne{\v{s}}et{\v{r}}il.
\newblock A combinatorial proof of the extension property for partial
  isometries.
\newblock {\em Commentationes Mathematicae Universitatis Carolinae},
  60(1):39--47, 2019.

\bibitem[HL00]{herwig2000}
Bernhard Herwig and Daniel Lascar.
\newblock Extending partial automorphisms and the profinite topology on free
  groups.
\newblock {\em Transactions of the American Mathematical Society},
  352(5):1985--2021, 2000.

\bibitem[HN19]{Hubicka2016}
Jan Hubi{\v{c}}ka and Jaroslav Ne\v{s}et\v{r}il.
\newblock All those {R}amsey classes ({R}amsey classes with closures and
  forbidden homomorphisms).
\newblock {\em Advances in Mathematics}, 356C:106791, 2019.

\bibitem[HO03]{hodkinson2003}
Ian Hodkinson and Martin Otto.
\newblock Finite conformal hypergraph covers and {G}aifman cliques in finite
  structures.
\newblock {\em Bulletin of Symbolic Logic}, 9(03):387--405, 2003.

\bibitem[HPSW18]{Sabok}
Jingyin Huang, Michael Pawliuk, Marcin Sabok, and Daniel Wise.
\newblock The {H}rushovski property for hypertournaments and profinite
  topologies.
\newblock arXiv:1809.06435, 2018.

\bibitem[Hru92]{hrushovski1992}
Ehud Hrushovski.
\newblock Extending partial isomorphisms of graphs.
\newblock {\em Combinatorica}, 12(4):411--416, 1992.

\bibitem[Hub20]{hubika2020structural}
Jan Hubi{\v{c}}ka.
\newblock Structural {R}amsey {T}heory and the {E}xtension {P}roperty for
  {P}artial {A}utomorphisms.
\newblock Introduction to habilitation thesis, arXiv:2010.05041, 2020.

\bibitem[Iva15]{Ivanov2015}
Aleksander Ivanov.
\newblock An $\omega$-categorical structure with amenable automorphism group.
\newblock {\em Mathematical Logic Quarterly}, 61(4-5):307--314, 2015.

\bibitem[JT20]{jahel2020invariant}
Colin Jahel and Todor Tsankov.
\newblock Invariant measures on products and on the space of linear orders.
\newblock {\em arXiv:2007.00281}, 2020.

\bibitem[Kec12]{Kechris2012}
A.~Kechris.
\newblock Dynamics of non-archimedean polish groups.
\newblock In {\em European Congress of Mathematics Kraków, 2 – 7 July,
  2012}, pages 375--397, 2012.

\bibitem[Kon18]{Konecny2018bc}
Mat{\v e}j Kone{\v c}n{\'y}.
\newblock Combinatorial properties of metrically homogeneous graphs.
\newblock Bachelor's thesis, Charles University, 2018.
\newblock arXiv:1805.07425.

\bibitem[Kon19]{Konecny2018b}
Mat{\v e}j Kone{\v c}n{\'y}.
\newblock Semigroup-valued metric spaces.
\newblock Master's thesis, Charles University, 2019.
\newblock arXiv:1810.08963.

\bibitem[Kon20]{Konecny2019a}
Mat{\v e}j Kone{\v c}n{\'y}.
\newblock Extending partial isometries of antipodal graphs.
\newblock {\em Discrete Mathematics}, 343(1):111633, 2020.

\bibitem[KPT05]{Kechris2005}
Alexander~S. Kechris, Vladimir~G. Pestov, and Stevo Todor{\v c}evi{\' c}.
\newblock Fra{\"\i}ss{\'e} limits, {R}amsey theory, and topological dynamics of
  automorphism groups.
\newblock {\em Geometric and Functional Analysis}, 15(1):106--189, 2005.

\bibitem[KR07]{Kechris2007}
Alexander~S. Kechris and Christian Rosendal.
\newblock Turbulence, amalgamation, and generic automorphisms of homogeneous
  structures.
\newblock {\em Proceedings of the London Mathematical Society}, 94(2):302--350,
  2007.

\bibitem[Ne{\v{s}}05]{Nevsetril2005}
Jaroslav Ne{\v{s}}etril.
\newblock Ramsey classes and homogeneous structures.
\newblock {\em Combinatorics, probability and computing}, 14(1-2):171--189,
  2005.

\bibitem[Ne{\v{s}}07]{Nevsetvril2007}
Jaroslav Ne{\v{s}}et{\v{r}}il.
\newblock Metric spaces are {R}amsey.
\newblock {\em European Journal of Combinatorics}, 28(1):457--468, 2007.

\bibitem[NVT15]{NVT14}
Lionel Nguyen Van~Th{\'e}.
\newblock {A survey on structural {R}amsey theory and topological dynamics with
  the {K}echris-{P}estov-{T}odorcevic correspondence in mind}.
\newblock {\em Selected Topics in Combinatorial Analysis}, 17(25):189--207,
  2015.

\bibitem[Ott17]{otto2017}
Martin Otto.
\newblock Amalgamation and symmetry: From local to global consistency in the
  finite.
\newblock {\em arXiv:1709.00031}, 2017.

\bibitem[Pes08]{Pestov2008}
Vladimir~G. Pestov.
\newblock A theorem of {H}rushovski--{S}olecki--{V}ershik applied to uniform
  and coarse embeddings of the {U}rysohn metric space.
\newblock {\em Topology and its Applications}, 155(14):1561--1575, 2008.

\bibitem[PS18]{paolini2018automorphism}
Gianluca Paolini and Saharon Shelah.
\newblock The automorphism group of {H}all’s universal group.
\newblock {\em Proceedings of the American Mathematical Society},
  146(4):1439--1445, 2018.

\bibitem[Ros11a]{rosendal2011b}
Christian Rosendal.
\newblock Finitely approximable groups and actions part {I}{I}: {G}eneric
  representations.
\newblock {\em The Journal of Symbolic Logic}, 76(04):1307--1321, 2011.

\bibitem[Ros11b]{rosendal2011}
Christian Rosendal.
\newblock Finitely approximate groups and actions part {I}: The
  {R}ibes--{Z}alesski{\u\i} property.
\newblock {\em The Journal of Symbolic Logic}, 76(04):1297--1306, 2011.

\bibitem[RZ93]{Ribes1993}
Luis Ribes and Pavel~A Zalesski{\u\i}.
\newblock On the profinite topology on a free group.
\newblock {\em Bulletin of the London Mathematical Society}, 25(1):37--43,
  1993.

\bibitem[Sab17]{sabok2017automatic}
Marcin Sabok.
\newblock Automatic continuity for isometry groups.
\newblock {\em Journal of the Institute of Mathematics of Jussieu}, pages
  1--30, 2017.

\bibitem[Sin17]{Siniora2}
Daoud Siniora.
\newblock {\em Automorphism Groups of Homogeneous Structures}.
\newblock PhD thesis, University of Leeds, March 2017.

\bibitem[Sol05]{solecki2005}
S{\l}awomir Solecki.
\newblock Extending partial isometries.
\newblock {\em Israel Journal of Mathematics}, 150(1):315--331, 2005.

\bibitem[Sol09]{solecki2009}
S{\l}awomir Solecki.
\newblock Notes on a strengthening of the {H}erwig--{L}ascar extension theorem.
\newblock Unpublished note, 2009.

\bibitem[SS19]{Siniora}
Daoud Siniora and S{\l}awomir Solecki.
\newblock Coherent extension of partial automorphisms, free amalgamation, and
  automorphism groups.
\newblock {\em The Journal of Symbolic Logic}, 2019.

\bibitem[Ver08]{vershik2008}
Anatoly~M. Vershik.
\newblock Globalization of the partial isometries of metric spaces and local
  approximation of the group of isometries of {U}rysohn space.
\newblock {\em Topology and its Applications}, 155(14):1618--1626, 2008.

\end{thebibliography}
\end{document}